\documentclass[pdflatex]{sn-jnl}

\usepackage{graphicx}%
\usepackage{multirow}%
\usepackage{amsmath,amssymb,amsfonts}%
\usepackage{amsthm}%
\usepackage{mathrsfs}%
\usepackage[title]{appendix}%
\usepackage{xcolor}%
\usepackage{textcomp}%
\usepackage{manyfoot}%
\usepackage{booktabs}%
\usepackage{algorithm}%
\usepackage{algorithmicx}%
\usepackage{algpseudocode}%
\usepackage{listings}%
\usepackage{tikz}%

\newtheorem{theorem}{Theorem}
\newtheorem{proposition}{Proposition}
\newtheorem{lemma}{Lemma}
\newtheorem{corollary}{Corollary}

\newtheorem{definition}{Definition}
\newcommand{\boxto}{%
	\mathrel{\mathop\Box}\mathrel{\mkern-2.5mu}\rightarrow
}
\newcommand{\diamondto}{%
	\mathrel{\mathop\Diamond}\mathrel{\mkern-2.8mu}\rightarrow
}
\newcommand{\boxTo}{\ensuremath{%
		\mathrel{\Box\kern-1.5pt\raise1pt\hbox{$\mathord{\Rightarrow}$}}}}
\newcommand{\diamondTo}{\ensuremath{%
		\mathrel{\Diamond\kern-1.5pt\raise1pt\hbox{$\mathord{\Rightarrow}$}}}}

\newcommand*\vDdash{%
	\mathrel{%
		\ooalign{$\vdash$\cr$\vDash$}%
	}%
}
\def\@clipped@vdash{%
	\raise .6ex\hbox{\clipbox{0pt .6ex 0pt .6ex}{$\vdash$}}%
}

\newcommand{\NEG}{\mbox{$\sim$}}
\begin{document}

\title[Connexive modal and conditional logic]{On connexivity in modal and conditional contexts}


\author*[1]{\fnm{Grigory K.} \sur{Olkhovikov} 
}\email{grigory.olkhovikov@\{rub.de, gmail.com\} 
}

%

\affil*[1]{\orgdiv{Department of Philosophy I}, \orgname{Ruhr University Bochum}, \orgaddress{\street{Universit\"atsstra{\ss}e 150}, \city{Bochum}, \postcode{44780}, \country{Germany}}}

%


\abstract{We define and axiomatize three new logics based on the connexive logic $\mathsf{C}$, the modal logic $\mathsf{CnK}$ and the conditional logics $\mathsf{CnCK}$ and $\mathsf{CnCK}_R$. These logics display strong connexivity properties and are connected to one another, since $\mathsf{CnCK}_R$ is the reflexive extension of $\mathsf{CnCK}$ and $\mathsf{CnK}$ is faithfully embeddable into both $\mathsf{CnCK}$ and $\mathsf{CnCK}_R$ in a multitude of natural ways. We argue that all the three logics provide (albeit in different ways) natural expansions of $\mathsf{C}$ to their respective languages that preserve and further develop several core properties of $\mathsf{C}$, especially its connexivity profile.}




\maketitle

%
%

\section{Introduction}
In the present paper we extend the propositional logic $\mathsf{C}$, originally introduced in \cite{w}, first with a set of modal operators to obtain the modal logic $\mathsf{CnK}$, and then with a set of conditional operators to get the conditional logics $\mathsf{CnCK}$ and $\mathsf{CnCK}_R$; we also axiomatize the resulting systems. 

Logic $\mathsf{C}$ is a non-classical logic which is relatively well-known and well-researched in the existing literature; even though, admittedly, it is not quite a household name yet. Its closest relatives are the numerous variations of the Nelson's logic of strong negation (originally introduced in \cite{nelson}, but $\mathsf{C}$ is in fact closer to its paraconsistent version first proposed independently in \cite{almukdad} and \cite{kutschera}; see also \cite[Appendix B, \S 2]{prawitz}). Despite its somewhat niche character, $\mathsf{C}$ is known to display a rather unusual combination of rare properties (most of which can be ultimately traced back to its reading of false implicative sentences) and thus provides a worthy and intriguing object of study. 

It is therefore important to properly take this reading into account in any attempts to generalize the treatment of implication in $\mathsf{C}$ by expanding its language with modal operators or with binary operators for would- and might-conditionals. 

In this paper, we make exactly this attempt and the main focus of the paper is on the way in which the properties of strict implication in $\mathsf{CnK}$ and of would- and might-conditionals in $\mathsf{CnCK}$ mirror some important properties of the implication in $\mathsf{C}$. This link is admittedly less tight in the case of $\mathsf{CnCK}$, which provides a very wide generalization of the implication of $\mathsf{C}$.\footnote{Something like this also happens in the classical case: compare the properties of would-conditionals in the minimal classical logic of conditionals $\mathsf{CK}$, introduced in \cite{chellas}, and the properties of material implication.} However, many important properties of the implication in $\mathsf{C}$ can be recovered in several natural extensions of $\mathsf{CnCK}$ and we briefly consider one such extension which we call $\mathsf{CnCK}_R$.

To the best of our knowledge, $\mathsf{CnK}$, $\mathsf{CnCK}$, and $\mathsf{CnCK}_R$ have not been hitherto considered in the existing literature. Still, our work on modal and conditional expansions of $\mathsf{C}$ is far from being the first of its kind. The work on modal extensions of $\mathsf{C}$ is, in fact, as old as $\mathsf{C}$ itself ---  the first modal $\mathsf{C}$-based modal logic was introduced and axiomatized already in \cite{w}, the same paper that introduced $\mathsf{C}$ itself. However, the line of thought that naturally leads to reflecting the properties of the connexive implication of $\mathsf{C}$ in the properties of modalities in the way it is done in $\mathsf{CnK}$, seems to be relatively recent. For example, in \cite[pp. 248--249]{omw3}, the authors describe the tweaking of the falsity conditions for the modalities in the $\mathsf{C}$-based logic introduced in \cite{w}, which leads exactly to the semantics of $\mathsf{CnK}$. The falsity condition for $\Box$ in $\mathsf{CnK}$ is given in  \cite{omw3} explicitly, whereas the falsity condition for $\Diamond$ is only alluded to implicitly, but the allusion is clear enough. The authors of \cite{omw3} do not axiomatize the modal logic arising in this way nor do they look into its properties otherwise; rather, they put it up as a worthy and interesting object of study and end their treatment of the topic by formulating ``the following [open] problem, which may seem to be more like a research proposal'', namely to find ``interesting applications of systems of $\mathsf{FDE}$-based modal logics (including [the modal logic proposed in \cite{w}] in particular) with non-standard modalities obtained by keeping familiar support of truth (falsity) conditions of the modalities and tweaking their support of falsity (truth) conditions'' (\cite[p. 249]{omw3}, adjusted for notation and context). While the present paper cannot be viewed as giving anything approaching a complete answer to this broad question, we hope that it contributes to the discussion of the said question by looking into one prominent example of the modal logics in question, axiomatizing it and exploring some of its elementary properties.\footnote{Another strain of existing work is due to Nissim Francez, who introduces in \cite{francez-mod} a logic called $\mathsf{CS5}$, a connexive version of $\mathsf{S5}$. This logic, however, seems to be incomparable with $\mathsf{CnK}$ due to the classical restrictions imposed in $\mathsf{CS5}$ on the evaluation of propositional letters (see \cite[Section 4]{francez-mod}). Any discussion of differences between $\mathsf{CS5}$ and $\mathsf{CnK}$ is a non-trivial task given that (1) $\mathsf{CS5}$ is defined by a hyper-sequent calculus and lacks a Hilbert-style axiomatization, (2) the semantics given for the $\Diamond$-free fragment of $\mathsf{CS5}$ is not proven (or even explicitly claimed) to be adequate for the hyper-sequent calculus that defines $\mathsf{CS5}$, and (3) the semantics of $\Diamond$ in  $\mathsf{CS5}$ is not discussed in \cite{francez-mod} at all.} 

Even though the conditional logics presented in this paper seem to be the first extensions of $\mathsf{C}$ with conditional operators ever proposed in the existing literature, similar ideas for the semantics of conditionals have been previously tried in a somewhat different propositional environment. The semantics of $\mathsf{CnCK}$ presented in this paper is in fact inspired by the semantics of the conditional logic  $\mathsf{cCCL}$ introduced in \cite{wu}, even though the two logics end up being rather different and are easily shown to be incomparable. Finally, the work on $\mathsf{CnCK}$ presented below is a natural continuation of our previous work on conditional logics, which spans several years and includes systems like $\mathsf{IntCK}$ and $\mathsf{N4CK}$.\footnote{This work has been presented in \cite{olkhovikov1}, \cite{olkhovikov2}, and \cite{olkhovikov3}.} This work, therefore, is best appreciated against the backdrop of the intricate web of connections and embeddings that tie together all of the systems mentioned in this paragraph and include several others that we omit to mention here.

However, in order to keep this text a reasonably sized paper and not a book, we had to confine ourselves to the discussion of the mutual relations of the four logics mentioned above, viz. $\mathsf{C}$,  $\mathsf{CnK}$, $\mathsf{CnCK}$, and $\mathsf{CnCK}_R$ in an accessible and self-contained manner and to keep the mapping-out of their connections to any other related work to an absolute minimum, thus leaving any detailed treatment of it to future publications.

These plans unfold into the following structure for the remaining part of the paper. In Section \ref{S:preliminaries} we fix notational conventions and introduce common terminology applicable to every logic considered below. All of them are defined in the paper by a somewhat uncommon bi-valuational variant of Kripke semantics that needs to be explained in advance. Another subject touched upon in this section is the definition and taxonomy of various connexive properties of logics. The research on connexive logics (see, e.g. \cite{w-sep}) is currently in a very active phase; consequently, the terminology and, more generally, the ideas about a correct angle of view on this whole subject remain in constant flux. Our paper, despite its limited aims, has to take some sort of position in this ongoing debate, given that $\mathsf{C}$ is currently viewed as a primary example of a \textit{hyperconnexive logic} and its connexive properties, therefore, also have to be reflected by its natural expansions to modal and conditional languages. Section \ref{S:c} then gives a quick and self-contained introduction to $\mathsf{C}$ and thus prepares the reader for the main results. Section \ref{S:modal} introduces and axiomatizes the modal logic $\mathsf{CnK}$. It is followed by Section \ref{S:conditional} devoted to the introduction and axiomatization of $\mathsf{CnCK}$; next comes Section \ref{S:reflexive}, in which $\mathsf{CnCK}_R$, a natural extension of $\mathsf{CnCK}$, is evaluated against the backdrop of the properties of $\mathsf{C}$ explained in Section 2 and linked to $\mathsf{CnK}$. 

Admittedly, all of this barely scratches the surface of what needs to be done to fully motivate our choice of $\mathsf{CnCK}$ as the correct minimal $\mathsf{C}$-based conditional logic; we expect to devote at least two further papers to this task. Section \ref{S:conclusions} provides an outlook on the contents of these papers as well as some other related work that we hope to publish in the near future.  

\section{Preliminaries and common notions}\label{S:preliminaries}

\subsection{Notational conventions}\label{sub:notational-conventions}
We use this subsection to fix some notations to be used throughout the paper.

We identify the natural numbers with finite ordinals. We denote by $\omega$ the smallest infinite ordinal. 
For any $n \in \omega$, we will denote by $\bar{o}_n$ the sequence
$(o_1,\ldots, o_n)$ of objects of any kind; moreover, somewhat
abusing the notation, we will sometimes denote $\{o_1,\ldots,
o_n\}$ by $\{\bar{o}_n\}$. The ordered $1$-tuple will
be identified with its only member.
For any given $m,n \in \omega$, the notation $(\bar{p}_m)^\frown(\bar{q}_n)$ denotes the concatenation of $\bar{p}_m$ and $\bar{q}_n$.

We will use IH as the abbreviation for Induction Hypothesis in the inductive proofs; $\alpha:=\beta$ means that we define $\alpha$ as $\beta$. We will use the usual notations for sets and functions. As for the sets, we will write $X\Subset Y$, iff $X\subseteq Y$ and $X$ is finite. We will extensively use ordered couples of sets which we will also call \textit{bi-sets}. The usual set-theoretic relations and operations on bi-sets will be understood componentwise, so that, e.g. $(X,Y)\subseteq(Z,W)$ means that $X \subseteq Z$ and $Y \subseteq W$ and similarly in other cases.

Relations will be understood as sets of ordered tuples where the length of the tuple defines the \textit{arity} of the relation. Given binary relations $R \subseteq X \times Y$ and $S\subseteq Y\times Z$, we set $R\circ S:= \{(a,c)\mid\text{ for some }b \in Y,\,(a, b)\in R\,\&\,(b,c)\in S\}$, and $R^{-1}:=\{(b,a)\mid (a,b)\in R\}$.

Functions will be understood as relations with special properties; we will write $f:X \to Y$ to denote a function $f \subseteq X\times Y$ such that its left projection is all of $X$. If  $f: X \to Y$ and $Z\subseteq X$ then we will denote the image of $Z$ under $f$ by $f(Z)$. 

\subsection{Languages and logics}\label{sub:languages-and-logics}
Within this paper, a \textit{language} is understood as the result of closing the set $Prop:=\{p_i\mid i\in \omega\}$ of propositional letters under the applications of a finite connective set $L = \{f^{i_1}_1,\ldots,f^{i_n}_n\}$ where $i_1,\ldots,i_n \in \omega$ indicate the arities. In this setting, a language can be identified with its own set of generating connectives and will be therefore also denoted by $L$; the inclusion of languages is then identical with inclusion of their corresponding connective sets. Given a language $L$, we will denote by $\mathcal{L}$ its generated formula set.

Below we will consider a relatively limited family of languages understood in this sense. It consists of the following three languages:
\begin{align*}
	PL&:=\{\wedge^2, \vee^2, \to^2, \NEG^1\}\\
	MD&:=PL\cup \{\Box^1, \Diamond^1\}\\
	CN&:=PL\cup \{\boxto^2, \diamondto^2\}
\end{align*}
We will refer to them as the \textit{propositional}, the \textit{modal}, and the \textit{conditional} language, respectively. Their sets of formulas will be denoted by $\mathcal{PL}$, $\mathcal{MD}$ and $\mathcal{CN}$, respectively. 

Given a language $L = \{f^{i_1}_1,\ldots,f^{i_n}_n\}$, $\phi, \psi \in \mathcal{L}$ and a $p \in Prop$, we can define $\phi[\psi/p]$, the result of replacing of all occurrences of $p$ in $\phi$ with occurrences of $\psi$ by the following induction on the construction of $\phi$:
\begin{align*}
	q[\psi/p]&:=\begin{cases}
		\psi,\text{ if }q = p;\\
		q,\text{ otherwise.}
	\end{cases}\\
	f^{i_j}_j(\chi_1,\ldots,\chi_{i_j})[\psi/p]&:= f^{i_j}_j(\chi_1[\psi/p],\ldots,\chi_{i_j}[\psi/p]) &&1 \leq j \leq n
\end{align*}
%
%

%
%
Although one and the same logic can often be formulated over different languages, in this paper we will abstract away from such subtleties, and will simply treat a logic as a consequence relation (in other words, as a set of consecutions)  $\mathsf{L} \subseteq \mathcal{P}(\mathcal{L})\times\mathcal{P}(\mathcal{L})$, for some language $L$, where $(\Gamma, \Delta)\in \mathsf{L}$ iff $\Delta$ $\mathsf{L}$-follows from $\Gamma$ (we will also denote this by $\Gamma\models_\mathsf{L}\Delta$). We will say that $(\Gamma, \Delta)$ is $\mathsf{L}$-satisfiable iff $(\Gamma, \Delta)\notin \mathsf{L}$. Given a $\phi \in \mathcal{L}$, $\phi$ is $\mathsf{L}$-valid (we will also write $\phi \in \mathsf{L}$) iff $(\emptyset, \{\phi\}) \in \mathsf{L}$ and $\phi$ is $\mathsf{L}$-satisfiable iff $(\{\phi\},\emptyset) \notin \mathsf{L}$.

The inclusion of logics is then just inclusion of (consequence) relations. 


\subsection{Connexivity and other properties of logics}\label{sub:properties-of-logics}
Many important properties of logics can be defined already within this (admittedly, rather rudimentary) setting.

For example, we will call a logic $\mathsf{L}\subseteq \mathcal{P}(\mathcal{L})\times\mathcal{P}(\mathcal{L})$ \textit{trivial} iff $\mathsf{L}= \mathcal{P}(\mathcal{L})\times\mathcal{P}(\mathcal{L})$. We will call $\mathsf{L}$ \textit{negation-inconsistent}\footnote{Such definitions make sense since every language considered in this paper extends $PL$.} iff, for some $\phi \in \mathcal{L}$ we have $\phi\wedge\NEG\phi\in \mathsf{L}$. For many logics, for instance, for all extensions of intuitionistic logic which are closed for \textit{modus ponens} and formula substitutions, the triviality and the negation-inconsistency are equivalent; however, we will see that this is not so for the logics considered in this paper.

Next, we will say that $\mathsf{L}$ has the \textit{Disjunction Property} (DP) iff for all $\phi, \psi \in \mathcal{L}$ we have $\phi \vee \psi \in \mathsf{L}$ iff at least one of $\phi, \psi$ is in $\mathsf{L}$. Moreover, $\mathsf{L}$ has the \textit{Constructible Falsity Property} (CFP) iff for all $\phi, \psi \in \mathsf{L}$, we have $\NEG(\phi \wedge \psi) \in \mathsf{L}$ iff at least one of $\NEG\phi, \NEG\psi$ is in $\mathsf{L}$.

Furthermore, let $\ast^2$ be a connective that is either in $L$ or is definable over $L$. We will say that $\ast^2$ is \textit{plainly connexive} in $\mathsf{L}$ iff (1)  for all $\phi, \psi \in \mathcal{L}$, $\mathsf{L}$ includes the following theorems:
\begin{align}
	&\NEG(\NEG\phi\ast\phi)\label{E:ATast}\tag{AT$\ast$}\\
	&(\phi\ast\NEG\psi)\ast\NEG(\phi\ast\psi)\label{E:BTast}\tag{BT$\ast$}
\end{align}
and (2) for some $\phi, \psi \in \mathcal{L}$, $\mathsf{L}$ fails to include the following theorem:
\begin{equation}\label{E:nSast}\tag{nonSym$\ast$}
	(\phi\ast\psi)\ast(\psi\ast\phi)
\end{equation}
We can express the same thought more concisely if we say that $\mathsf{L}$ validates both \eqref{E:ATast} and \eqref{E:BTast} and fails \eqref{E:nSast}. We will adopt this more concise phrasing for the rest of the paper.
 
For historical reasons (further explained in, e.g. \cite{w-sep}), \eqref{E:ATast} and \eqref{E:BTast} are often referred to in the literature as Aristotle's Thesis and Boethius' Thesis for $\ast$, respectively. As indicated by our choice of labels, we adopt this naming also in the current paper. The scheme \eqref{E:nSast} is often described as the non-symmetry requirement; it is usually included into the definition of connexivity to rule out the plain connexivity of connectives like $\leftrightarrow$ over $\mathsf{CL}$.

Moreover, $\ast^2$ is \textit{weakly connexive} in $\mathsf{L}$ iff (1) $\mathsf{L}$ validates both \eqref{E:ATast} and the following consecution scheme:
\begin{align}
	&(\phi\ast\NEG\psi)\models_\mathsf{L}\NEG(\phi\ast\psi)\label{E:WBTast}\tag{WBT$\ast$}
\end{align}
and (2) $\mathsf{L}$ fails the consecution scheme
\begin{equation}\label{E:WnSast}\tag{WnonSym$\ast$}
	(\phi\ast\psi)\models_\mathsf{L}(\psi\ast\phi)
\end{equation} 
Schema \eqref{E:WBTast} is meant to be read (again, in accord with the existing tradition on connexive logics) the \textit{weak Boethius' Thesis for $\ast$}.

Finally, $\ast^2$ is \textit{fully connexive} in $\mathsf{L}$ (we will sometimes just say \textit{connexive} in this case) iff $\ast^2$ is both plainly and weakly connexive in $\mathsf{L}$. 

Possible natural strengthenings of connexivity are a popular subject in the existing literature on the topic. For our purposes, we only briefly address our understanding of one such strengthening, namely, hyperconnexivity. Assuming the hypotheses of our previous connexivity definitions, we say that $\ast^2$ is \textit{plainly hyperconnexive} in $\mathsf{L}$ iff $\ast^2$ is plainly connexive in $\mathsf{L}$, and, in addition, validates the following scheme (which we will also refer to, following the connexive tradition, as the \textit{Converse Boethius' Thesis}):
\begin{align}
	&\NEG(\phi\ast\psi)\ast(\phi\ast\NEG\psi)\label{E:CBTast}\tag{CBT$\ast$}
\end{align}
Next, $\ast^2$ is \textit{weakly hyperconnexive} in $\mathsf{L}$ iff $\ast^2$ is weakly connexive in $\mathsf{L}$, and also validates the following consecution scheme:
\begin{align}
	&\NEG(\phi\ast\psi)\models_\mathsf{L}(\phi\ast\NEG\psi)\label{E:WCBTast}\tag{WCBT$\ast$}
\end{align}
We are not aware of any discussion of this scheme in the literature on connexive logics, but, continuing with the naming style for the previous scheme, it seems natural to christen it as \textit{Weak Converse Boethius' Thesis}.

And, just as with connexivity, $\ast^2$ is \textit{fully hyperconnexive} in $\mathsf{L}$ (we will sometimes just say \textit{hyperconnexive} in this case) iff $\ast^2$ is both plainly and weakly hyperconnexive in $\mathsf{L}$.

Sometimes binary connectives fail to be plainly or weakly connexive while still displaying enough connexive features to merit a special denotation. We will speak of such connectives as being partially connexive. More precisely, $\ast^2$ is \textit{partially connexive} in $\mathsf{L}$ iff $\mathsf{L}$ validates at least one of \eqref{E:ATast}, \eqref{E:BTast}, and fails \eqref{E:nSast}.
Furthermore, $\ast^2$ is \textit{weakly partially connexive} over $\mathsf{L}$ iff $\mathsf{L}$ validates \eqref{E:WBTast} and fails \eqref{E:WnSast}. Finally, $\ast^2$ is \textit{weakly partially hyperconnexive} in $\mathsf{L}$ iff it is weakly partially connexive in $\mathsf{L}$ and verifies \eqref{E:WCBTast}.

Even partially connexive connectives are rarely found among the usual suspects in the field of non-classical logics; in classical logic, they are known to be completely absent. However, we will see many such connectives below, which motivates our quick tour through the connexive terminology.

\subsection{A detour: on terminological differences} 
Before we move on to the next subsection, we would like to briefly comment on our use of connexive terminology. 

Apart from the naming of the main connexive theses, there seems to be no general consensus as to the correct definition of either the basic notion of connexivity or its different variants like hyperconnexivity. Despite the best efforts by some authors (we are thinking about \cite{omw2}) to build an order out of this creative chaos, almost every paper in this field comes with its own take on connexivity properties. Unfortunately, this paper is no exception, so an apology and an explanation both seem to be in order. Below we list and explain the main points of disagreement arising between our terminology and the one found in other sources on connexive logic (especially the ones in the vicinity of \cite{omw1} and \cite{omw2}) and motivate these differences in the context of the setup and the approach of this paper.

1. \textit{Connexivity as a property of implication vs connexivity as a property of an abstract connective}. Connexivity is viewed in this paper as a property applicable to an arbitrary binary connective; however, historically, connexivity was most often understood as a pattern of interaction between \textit{implication} and negation. Accordingly, the sources like \cite{omw1} and \cite{w-sep} only formulate Aristotle's and Boethius' theses for $\to^2$. On the other hand, a tendency to view connexivity as an abstract property of connectives and families thereof is also very much present in the literature. E.g. in \cite{francez} we find a case for connexive conjunction. An even more recent \cite{ce} suggests a view of connexivity which is way more general than our approach. For example, their proposed generalization of (BT$\to$) not only replaces $\to$ in the connexive theses with an arbitrary binary connective as in our \eqref{E:BTast} above; it also replaces every single occurrence of $\to$ and $\NEG$ in (BT$\to$) with a different arbitrary connective of corresponding arity so that this thesis now assumes the form
$$
(\phi\ast_1-_1\psi)\ast_2-_2(\phi\ast_3\psi).
$$
Connexivity, under this reading, is then a property of a family of connectives, in this case of the family consisting of $\{\ast_1,\ldots,\ast_3,-_1,-_2\}$.

Our own opinion is that connexivity properties must be about \textit{something like} implication and negation in the end of the day. Therefore, our sympathies decidedly gravitate towards the more austere approach of \cite{omw2}. However, we also find ourselves unable to carry it through in all its purity within this paper. Nor, in fact, does it appear to be fully possible for the authors of \cite{omw2} itself. When it is explained in \cite[p. 23]{omw2} that the function of the (nonSym$\to$) is to avoid the recognition of connexivity of $\leftrightarrow$ in $\mathsf{CL}$ they already presuppose, in a sense, the applicability of connexive theses to a connective that is different from $\to$, namely to $\leftrightarrow$. When in \cite{omw1a} the same authors argue for what they call \textit{totally connexive} character of the defined\footnote{$\Rightarrow$ is defined in terms of $\to$ and $\NEG$. The same definition of $\Rightarrow$ is also adopted in the present paper, see Section \ref{S:c} below.} connective $\Rightarrow$ in $\mathsf{C3}$, the same type of presupposition is made w.r.t. $\Rightarrow$. And in the paper \cite{u} published in a special journal issue co-edited by both authors of \cite{omw2}, the author just replaces $\to$ in both (AT$\to$) and (WBT$\to$) with the conditional operator $\boxto$ and, without further ado, calls this the definition of connexivity.

It is true that in all these cases one still applies connexive theses to something \textit{implication-like}, something that is either interpreted as a conditional (like $\boxto$ and $\Rightarrow$) or at least defined in terms of $\to$ (like $\leftrightarrow$ and $\Rightarrow$). However, it does not seem to be an easy (or even possible) task to define the \textit{implication-likeness} of a connective with any degree of precision. For example both $\wedge$ and $\Rightarrow$ are definable in classical propositional in terms of $\NEG$ and $\to$; yet the latter has the intuitive appeal of an implication-like connective, whereas the former lacks it altogether. 

This is why we have defined our connexive theses for an arbitrary binary connective even though we will only instantiate them for connectives that appear to our intuition as implication-like, e.g. $\boxto$, $\diamondto$, $\Rightarrow$, etc.\footnote{Theoretically, the replacement of negation with a negation-like abstract unary connective in connexive theses is also interesting. However, in this respect we prefer to stick to the austerity of the original definition, and to speak of negation only, for the following reasons: (1) we do not consider any negation-like connectives different from $\NEG$ below and (2) generalizing negation would also necessitate a longer definition of connexivity as in this case one would have to give every connexive thesis in two different forms.}

2. \textit{Different forms of connexive theses}. It is customary to give both Aristotle's and Boethius' theses in two forms each and to demand that a fully connexive connective satisfies both of these forms. In the previous subsection, however, we represent each of the connexive theses by a single formula (or consecution) scheme. Were we to follow the more traditional approach, our definition of plain connexivity should have been extended by the requirement to validate the following two schemes:
\begin{align}
	&\NEG(\phi\ast\NEG\phi)\label{E:AT'ast}\tag{AT'$\ast$}\\
	&(\phi\ast\psi)\ast\NEG(\phi\ast\NEG\psi)\label{E:BT'ast}\tag{BT'$\ast$}
\end{align}
Moreover, the schemes \eqref{E:WBTast}, \eqref{E:CBTast}, and \eqref{E:WCBTast} should have been duplicated according to this same pattern.

We do not do this since every connective $\ast$ that is definable in every logic considered in this paper satisfies one of \eqref{E:ATast}, \eqref{E:AT'ast} iff it satisfies both schemes. The same holds for the two forms of Boethius' Thesis and for other connexive theses. This situation is due to the properties of negation in our logics of interest which, in particular, always satisfy the law of double negation elimination in a rather strong sense, cf. Lemma \ref{L:c-strong-properties} below.

In choosing a simplified description of connexive theses, we follow the example of \cite{u} where both (AT$\boxto$) and (BT$\boxto$) are given in a single form --- apparently for the same reasons as the ones outlined above, i.e. due to the negation properties.

3. \textit{Plain, weak, and full connexivity}. Our definitions of weak connexivity and partial connexivity correspond (modulo the replacement of $\to$ by abstract binary connective $\ast^2$) to the definition of this notion given in \cite[pp.18--19]{omw2}. However, what we define as plain connexivity is called in \cite[p. 3]{omw2} and in many other sources on connexive logics just connexivity. In contrast to this well-established terminological convention, in this paper we define connexivity (or full connexivity) as a stronger notion, requiring a fully connexive logic to be both plainly and weakly connexive.

The reason for this choice is that connexivity is supposed to be stronger than weak connexivity. This is acknowledged also by the authors of \cite{omw2}, who, on p.18, describe a weakly connexive logic as a logic that satisfies Boethius' Thesis ``only in the rule form''. However, as long is $\ast^2$ is non detachable in $\mathsf{L}$ in the sense that $\mathsf{L}$ fails the consecution scheme $\phi, (\phi\ast\psi)\models_\mathsf{L} \psi$ --- which seems to be assumed in \cite{omw2} ---  \eqref{E:BTast} does not imply \eqref{E:WBTast}; also \eqref{E:WCBTast} and \eqref{E:WnSast} no longer follow from \eqref{E:CBTast} and \eqref{E:nSast}, respectively. In such a context the ``only'' of \cite[p. 18]{omw2} no longer makes sense. Since the failure of detachment is typical for conditional operators like $\boxto$ and $\diamondto$, the reference to plain connexivity as (simply) connexivity becomes misleading in the context of conditional logic. For this reason, we thought it necessary to rename the connexivity of \cite{omw2} as plain connexivity and to give full connexivity a stronger reading.\footnote{Another option would be to rename weak connexivity into rule-connexivity and plain connexivity into formula-connexivity, but we preferred the current choice as it is closer to the terminological conventions used in the existing literature. However, we would advocate switching from plain/weak to rule/formula terminology in the future discussions of connexivity properties.}

For similar reasons, our notion of plain hyperconnexivity is just the hyperconnexivity of \cite{omw2}. Our notions of weak hyperconnexivity and weak partial connexivity appear to be new though. We include them since they come in handy in addressing the properties of the logics to be considered below and since together with other notions, they create a nicely rounded-up nomenclature of connexivity properties.

\subsection{Bi-valuational Kripke semantics}\label{sub:semantics}

All of the logics that we will focus on in this paper will have a common property, namely, that their sets of consequences can be seen as induced by their intended bi-valuational relational (or Kripke) semantics. For all of these logics, their sets of consequences can be alternatively seen as induced by their complete Hilbert-style axiomatizations. In this subsection, we take a quick look at the first of these two approaches to defining a logic. Bi-valuational Kripke semantics is a variant of Kripke semantics in which the models expand their respective Kripke frames with two functions evaluating the atoms instead of just one, as in the usual Kripke semantics. Of course, the restrictions imposed on a frame that is being extended with either one or two functions depend on the character of language which this semantics is supposed to interpret. Since in this paper we are going to consider logics over three languages, $PL$, $MD$, and $CN$ we need a definition of a model that is flexible enough to accommodate our plans: 
\begin{definition}\label{D:Krip}
For any $n \in \{1,2\}$, a structure $\mathcal{M} = (W, R_1,\ldots, R_n, V^+, V^-)$ is called a bi-valuational Kripke (or relational) model iff the following conditions are satisfied:
\begin{itemize}
	\item $W \neq \emptyset$ ($W$ is called the set of \textit{worlds} or \textit{nodes} of $\mathcal{M}$).
	
	\item $R_1 \subseteq W \times W$ is reflexive and transitive. We will normally denote $R_1$ by $\leq$ in what follows.
	
	\item In case $n = 1$, $\mathcal{M}$ is called a \textit{propositional model}.
	
	\item In case $n = 2$,  we must have either have $R_2 \subseteq W\times W$ in which case $\mathcal{M}$ is called a \textit{modal model} or $R_2\subseteq W \times(\mathcal{P}(W))^2\times W$, in which case $\mathcal{M}$ is called a \textit{conditional model}. We will normally denote $R_2$ by $R$ in what follows, and will call it modal or conditional accessibility relation depending on the type of the model.
	
	\item For all $\ast\in \{+, -\}$, $V^\ast:Prop \to \mathcal{P}(W)$ is such that, for every $p\in Prop$:
	$$
	w \leq v\text{ and }p \in V^\ast(w)\text{ implies }p \in V^\ast(v).
	$$
	
	\item The class of all bi-valuational Kripke models (resp. propositional, modal, conditional) models will be denoted by $\mathbb{K}$, (resp. $\mathbb{P}$, $\mathbb{M}$, $\mathbb{C}$). It is clear that $\mathbb{K} = \mathbb{P}\cup \mathbb{M}\cup\mathbb{C}$.
\end{itemize}
\end{definition}
The representation of a Kripke model given in Definition \ref{D:Krip} will serve as a default; namely, given an $\mathcal{M}$ in $\mathbb{P}$ (resp. $\mathbb{M}\cup \mathbb{C}$) we will assume that $\mathcal{M}$ is given in the form $(W, \leq, V)$ (resp. $(W, \leq, R, V)$), and that any decorations are inherited by the model components, in other words, that $\mathcal{M}_n, \mathcal{M}'$ in $\mathbb{P}$ (resp. $\mathbb{M}\cup \mathbb{C}$) are given as $(W_n, \leq_n, V_n), (W', \leq', V')$ (resp. $(W_n, \leq_n, R_n, V_n), (W', \leq', R', V')$) unless explicitly stated otherwise.

Next, \textit{pointed models} can be introduced in the bi-valuational Kripke semantics in the usual way: for a class $Mod \subseteq \mathbb{K}$, we the class of pointed $Mod$-models is defined as
$$
Pt(Mod):= \{(\mathcal{M}, w)\mid\mathcal{M}\in Mod,\,w \in W\}.
$$
In plain Kripke semantics the only evaluation function $V$ of a Kripke model is usually thought of as assigning extensions to propositional atoms (understood as sets of pointed models) in the sense that $V(p)$ is the set of worlds in which $p$ is verified or true at $\mathcal{M}$. 
 The falsifications of atoms are then read off their verifications in some way or another (e.g. by stipulating that $p$ is falsified iff it fails to be verified at a given node, or, perhaps, at every pointed node in a certain subset of $W$).
 
More generally, a sentence is thought of as a certain subset of $W$, and a given class of plain Kripke models can only supply a semantics for a certain fixed language $L$ if the evaluations of atoms are extended to arbitrary $\mathcal{L}$-formulas by means of some satisfaction relation $\triangleright\subseteq Pt(Mod)\times\mathcal{L}$. For such a satisfaction relation, an infix notation is normally adopted, so that one writes $\mathcal{M}, w\triangleright\phi$ meaning $((\mathcal{M}, w), \phi)\in \triangleright$. One can then associate with a given $\phi\in \mathcal{L}$ its \textit{plain extension} in $\mathcal{M}$ understood as the subset of $W$ in which $\phi$ is verified and setting
$$
\lvert \phi\rvert_\mathcal{M}:= \{w \in W\mid\mathcal{M}, w \triangleright \phi\},
$$
ensuring that the equation $V(p) = \lvert p\rvert_\mathcal{M}$ is verified for every $p \in Prop$. 

The most popular way to define such a $\triangleright$ is by induction on the construction of $\phi \in \mathcal{L}$, so that, if  $L = \{f^{i_1}_1,\ldots,f^{i_n}_n\}$, the inductive definition will include $n + 1$ clause $cl_1,\ldots, cl_{n + 1}$ (one for each connective of $L$ and one for the atoms\footnote{Some of these clauses may distinguish between several cases, of course, but for our purposes we will still count this as a single clause.}) . It is then convenient to write down the definition of the satisfaction relation in the following form:
\begin{equation}\label{E:tright}
	\triangleright:= Ind(\{\overline{cl}_{n + 1}\})
\end{equation}
A transition from plain to bi-valuational Kripke semantics is often prompted by a dissatisfaction with the indirect handling of falsifications in plain Kripke semantics, and allows to encode falsifications of atoms more or less independently of their verifications by means of an additional evaluation function. In this setting, $V^+$ is normally read as assigning the atoms their extensions and $V^-$ as assigning these same atoms their anti-extensions, so that a given $p \in Prop$ is thought of as being verified and falsified, within a given $\mathcal{M}\in \mathbb{K}$ at the nodes in sets $V^+(p)$ and $V^-(p)$, respectively. Thus, an atom $p \in Prop$ is being assigned a \textit{bi-extension in} $\mathcal{M}$ which can be represented as $\|p\|_\mathcal{M} = (V^+(p), V^-(p))$. 


In order to set up a logic on top of a given set of intended bi-valuational Kripke models, one usually proceeds in the same way as in plain Kripke semantics as described above, the only difference being that this time we have to extend two evaluation functions to the set of all $\mathcal{L}$-formulas rather than one; accordingly, we need two satisfaction relations $\triangleright^+, \triangleright^-\subseteq Pt(Mod)\times\mathcal{L}$. Their meanings are in harmony with the interpretation of $V^+$ and $V^-$ in that $\triangleright^+$ encodes verifications of a formula at nodes of a model, whereas $\triangleright^-$ encodes falsifications. In this way every $\phi\in \mathcal{L}$ gets assigned its bi-extension in $\mathcal{M}$ understood as:
$$
\|\phi\|_\mathcal{M}:= (\lvert\phi\rvert^+_\mathcal{M},\lvert\phi\rvert^-_\mathcal{M}) = (\{w \in W\mid\mathcal{M}, w \triangleright^+ \phi\},\{w \in W\mid\mathcal{M}, w \triangleright^+ \phi\}),
$$
where $\lvert\phi\rvert^+_\mathcal{M}$, $\lvert\phi\rvert^-_\mathcal{M}$, clearly being just the plain extensions assigned to $\phi$ by the two respective satisfaction relations, can be called the \textit{positive} and the \textit{negative} projection of $\|\phi\|_\mathcal{M}$. 

Just as in the plain case, the definitions of $\triangleright^+$ and $\triangleright^-$ typically proceed by induction on the construction of $\phi \in \mathcal{L}$, and, in case we have $L = \{f^{i_1}_1,\ldots,f^{i_n}_n\}$, consist of $(n + 1)$ inductive clause each. Only this time the inductive clauses for $\triangleright^+$ can cite both $\triangleright^+$- and $\triangleright^-$-satisfactions for subformulas of the main formula (and, of course, vice versa), so that the two relations are usually defined by a mutual induction rather than by two independent plain inductions. In order to reflect this fact, we need to change our notation slightly. If, say $\triangleright^+$ is given by inductive clauses $pcl_1,\ldots, pcl_{n + 1}$, $\triangleright^-$ is given by inductive clauses $ncl_1,\ldots, ncl_{n + 1}$, and some of these clauses may refer to the other satisfaction relation in the pair, we will represent this fact by  writing:
$$
(\triangleright^+,\triangleright^-):= Ind(\{\overline{pcl}_{n + 1}\},\{\overline{ncl}_{n + 1}\}).
$$
Our mutual induction is thus seen to define a pair of two mutually connected relations simultaneously.

Another useful piece of notation will allow us to extend satisfaction relations with new inductive clauses. For simplicity, take the plain case, assuming that $\triangleright$ is given by \eqref{E:tright} and another satisfaction relation $\triangleright'$ (say, for a richer language $L'\supseteq L$) is given by all the clauses of $\triangleright$ plus some additional clauses $cl'_1,\ldots, cl'_m$. We will express this fact by writing $\triangleright' = \triangleright\cup \{cl'_1,\ldots, cl'_m\}$ or, equivalently, $\triangleright = \triangleright'\setminus \{cl'_1,\ldots, cl'_m\}$. We will use the same notation for the satisfaction relations in the context of bi-valuational Kripke semantics.

Next, given a satisfaction relation $\triangleright$ (whether as a single one in the plain case or one of a pair of relations in bi-valuational case), we 
can canonically extend $\triangleright$ to bi-sets of $L$-formulas by setting
$$
\mathcal{M}, w\triangleright(\Gamma, \Delta)\text{ :iff }(\forall \phi\in \Gamma)(\mathcal{M}, w\triangleright\phi)\text{ and }(\forall \psi\in \Delta)(\mathcal{M}, w\not\triangleright\psi)
$$
In other words, we will write $\mathcal{M}, w\triangleright (\Gamma, \Delta)$ iff every formula in $\Gamma$ and none of the formulas in $\Delta$ are $\triangleright$-satisfied at $(\mathcal{M}, w)$.

We can now give a more precise definition of a bi-valuational Kripke semantics and the logics it induces. We understand bi-valuational Kripke semantics\footnote{Needless to say, plain Kripke semantics admits of a similar representation as a logic-forming operator. We omit the exact definition since we only consider in this paper logics given by bi-valuational semantics.} as an operator $\mathfrak{L}$ which, given a triple of arguments of the form $(L, Mod, (\triangleright^+,\triangleright^-))$ such that $L$ is a language, $Mod\subseteq \mathbb{K}$ is the class of intended models for $L$, and $\triangleright^+, \triangleright^-\subseteq Pt(Mod)\times\mathcal{L}$, returns a logic $\mathfrak{L}(L, Mod, (\triangleright^+,\triangleright^-))$ in the sense of Section \ref{sub:languages-and-logics}, such that we have, for any given $\Gamma, \Delta \subseteq \mathcal{L}$:
$$
(\Gamma, \Delta)\in \mathfrak{L}(L,Mod, (\triangleright^+,\triangleright^-))\text{ :iff }(\forall(\mathcal{M}, w)\in Pt(Mod))(\mathcal{M}, w\not\triangleright^+(\Gamma, \Delta)).
$$
In other words, $\Gamma$ $\mathfrak{L}(L,Mod, (\triangleright^+,\triangleright^-))$-entails $\Delta$ iff $\triangleright^+$-satisfaction of every formula in $\Gamma$ entails $\triangleright^+$-satisfaction of at least one formula in $\Delta$. 

In case $\mathsf{L} = \mathfrak{L}(L,Mod, (\triangleright^+,\triangleright^-))$, we will often write $\models^+_\mathsf{L}$ (resp. $\models^-_\mathsf{L}$) instead of $\triangleright^+$ (resp. $\triangleright^-$). This is usually done in the literature on plain Kripke semantics and will not lead to any confusions also in the case of bi-valuational semantics.

It is easy to see that our definition of  $\mathfrak{L}$ prioritizes $\triangleright^+$ over $\triangleright^-$. This choice corresponds to the selection of logics that we are going to consider below. It is, however, possible to make other choices in this respect, and there exists a whole literature on this subject (see, e.g. \cite{wn}).

\subsection{Hilbert-style axiomatizations}\label{sub:hilbert}
As for the Hilbert-style systems, all of them will be given by a finite number of axiomatic schemes $\alpha_1,\ldots,\alpha_n$ augmented with a finite number of inference rules (i.e. consecution schemes) $\rho_1,\ldots,\rho_m$  satisfied by a given logic, so the most general format sufficient for the present paper is $\Sigma(\bar{\alpha}_n; \bar{\rho}_m)$. 

All of the Hilbert-style systems considered in this paper, extend the axiomatic system\footnote{In fact, $\mathtt{S}_0$ is a standard axiomatization of the $\NEG$-free fragment of intuitionistic propositional logic.} $\mathtt{S}_0:= \Sigma(\bar{\alpha}_8;\eqref{E:mp})$, where we assume that:
\begin{align*}
	&\phi \to(\psi\to\phi)\,(\alpha_1),\quad(\phi\to(\psi\to\chi))\to((\phi\to\psi)\to(\phi\to\chi))\,(\alpha_2),\\
	&\quad(\phi\wedge\psi)\to\phi\,(\alpha_3),\quad(\phi\wedge\psi)\to\psi\,(\alpha_4),\quad \phi\to(\psi\to (\phi\wedge\psi))\,(\alpha_5),\\
	&\quad\phi\to(\phi\vee \psi)\,(\alpha_6),\quad\psi\to(\phi\vee \psi)\,(\alpha_7),\quad (\phi\to\chi)\to((\psi \to \chi)\to ((\phi\vee\psi)\to \chi))\,(\alpha_8)
\end{align*}
and:
\begin{align}
	\text{From }\phi, \phi \to \psi&\text{ infer }\psi\label{E:mp}\tag{MP}
\end{align}
Therefore, for our purposes in this paper, it is important to be able to refer to Hilbert-style systems as extensions of other systems. If $\mathtt{S} = \Sigma(\bar{\alpha}_n; \bar{\rho}_m)$, and $\beta_1,\ldots,\beta_k$ are some new axiomatic schemes and $\sigma_1,\ldots,\sigma_r$ are some new consecution schemes, then we will write $\mathtt{S}+(\bar{\beta}_k;\bar{\sigma}_r)$ to denote the system $\Sigma((\bar{\alpha}_n)^\frown(\bar{\beta}_k);(\bar{\rho}_m)^\frown(\bar{\sigma}_r))$. We will also write $\mathtt{S} \subseteq \mathtt{S}'$ to express that $\mathtt{S}' = \mathtt{S}+(\bar{\beta}_k;\bar{\sigma}_r)$ for appropriate $\bar{\beta}_k$ and $\bar{\sigma}_r$.

Axiomatic systems can be viewed as operators generating logics when applied to languages. More precisely, if $\mathtt{S} = \Sigma(\bar{\alpha}_n; \bar{\rho}_m)$ and $L$ is a language then $\mathsf{L} = \mathtt{S}[L]$ can be described as follows. We say that a $\phi \in \mathcal{L}$ is \textit{provable} in $\mathtt{S}$ iff there exists a finite sequence $\psi_1,\ldots,\psi_k$ of formulas in $\mathcal{L}$ such that every formula in this sequence is either a substitution instance of one of $\alpha_1,\ldots,\alpha_n$ or results from an application of one of $\rho_1,\ldots,\rho_m$ to some earlier formulas in the sequence and $\psi_k = \phi$; we will say that $(\Gamma, \Delta) \in \mathtt{S}[L]$ iff $\Gamma, \Delta\subseteq L$ and there exists a sequence $\chi_1,\ldots,\chi_r$ of formulas in $\mathcal{L}$ such that every formula in it is either in $\Gamma$, or is provable in $\mathtt{S}$ or results from an application of \eqref{E:mp} to a pair of earlier formulas in the sequence, and, for some $\theta_1,\ldots,\theta_s\in \Delta$ we have $\chi_r = \theta_1\vee\ldots\vee\theta_s$. This definition makes sense in the context of our paper, since every language that we are going to consider contains $\vee$, and every axiomatic system that we are going to consider contains \eqref{E:mp}. We will also express the fact that  $(\Gamma, \Delta) \in \mathtt{S}[L]$ by writing $\Gamma\vdash_\mathtt{S}\Delta$.

In addition, we will also speak of the rules \textit{derivable} in a logic generated by an axiomatic system. If $\Gamma\cup \{\phi\}\Subset \mathcal{L}$, then we will say that $\phi$ is derivable from $\Gamma$ in $\mathtt{S}[L]$ (and will write $\Gamma\vDdash_\mathtt{S}\phi$) iff there exists a finite sequence $\psi_1,\ldots,\psi_k$ of formulas in $\mathcal{L}$ such that every formula in this sequence is either in $\Gamma$, or is provable in $\mathtt{S}$, or results from an application of one of $\rho_1,\ldots,\rho_m$ to some earlier formulas in the sequence, and $\psi_k = \phi$. It is easy to see that, for a $\mathsf{L} = \mathtt{S}[L]$ and a $\phi \in \mathcal{L}$, we will have $\phi \in \mathsf{L}$ iff $\vdash_\mathtt{S}\phi$ iff $\vDdash_\mathtt{S}\phi$.

Finally, given any $\Gamma, \Delta\subseteq \mathcal{L}$, and an axiomatic system $\mathtt{S}$, we say that $(\Gamma, \Delta)$ is $\mathtt{S}$-\textit{consistent}  iff $\Gamma\not\vdash_\mathtt{S}\Delta$; we will say that $(\Gamma, \Delta)$ is $L$-\textit{complete} iff $\Gamma\cup \Delta = \mathcal{L}$. We will say that $(\Gamma, \Delta)$ is $\mathtt{S}[L]$-\textit{maximal} iff $(\Gamma, \Delta)$ is both $\mathtt{S}$-consistent and $L$-complete.

\section{Logic $\mathsf{C}$ and some of its properties}\label{S:c}

We start by recalling some facts about the logic $\mathsf{C}$. We define this logic by its Kripke semantics:
$$
\mathsf{C}:= \mathfrak{L}(PL, \mathbb{P}, (\models_{\mathsf{C}}^+, \models_{\mathsf{C}}^-)),
$$
where the satisfaction relations are defined by
$$
(\models_{\mathsf{C}}^+, \models_{\mathsf{C}}^-):= Ind(\{\{\eqref{Cl:at+}, \eqref{Cl:con+}, \eqref{Cl:dis+}, \eqref{Cl:neg+}, \eqref{Cl:im+}\}\}, \{\eqref{Cl:at-}, \eqref{Cl:con-}, \eqref{Cl:dis-}, \eqref{Cl:neg-}, \eqref{Cl:im-}\}),
$$
assuming the following inductive clauses:
\begin{align}
	\mathcal{M}, w&\models^+_{\mathsf{C}} p \text{ iff } w \in V^+(p)\qquad\qquad\qquad\qquad\text{for $p\in Prop$}\label{Cl:at+}\tag{at+}\\
	\mathcal{M}, w&\models^-_{\mathsf{C}} p \text{ iff } w \in V^-(p)\qquad\qquad\qquad\qquad\text{for $p\in Prop$}\label{Cl:at-}\tag{at-}\\
	\mathcal{M}, w&\models_{\mathsf{C}}^+ \psi \wedge \chi \text{ iff } \mathcal{M}, w\models_{\mathsf{C}}^+ \psi\text{ and }\mathcal{M}, w\models_{\mathsf{C}}^+ \chi\label{Cl:con+}\tag{$\wedge+$}\\
	\mathcal{M}, w&\models_{\mathsf{C}}^- \psi \wedge \chi \text{ iff } \mathcal{M}, w\models_{\mathsf{C}}^- \psi\text{ or }\mathcal{M}, w\models_{\mathsf{C}}^- \chi\label{Cl:con-}\tag{$\wedge-$}\\
	\mathcal{M}, w&\models_{\mathsf{C}}^+ \psi \vee \chi \text{ iff } \mathcal{M}, w\models_{\mathsf{C}}^+ \psi\text{ or }\mathcal{M}, w\models_{\mathsf{C}}^+ \chi\label{Cl:dis+}\tag{$\vee+$}\\
	\mathcal{M}, w&\models_{\mathsf{C}}^- \psi \vee \chi\text{ iff } \mathcal{M}, w\models_{\mathsf{C}}^- \psi\text{ and }\mathcal{M}, w\models_{\mathsf{C}}^- \chi\label{Cl:dis-}\tag{$\vee-$}\\
	\mathcal{M}, w&\models_{\mathsf{C}}^+ \NEG\psi \text{ iff } \mathcal{M}, w\models_{\mathsf{C}}^- \psi\label{Cl:neg+}\tag{$\NEG+$}\\
	\mathcal{M}, w&\models_{\mathsf{C}}^- \NEG\psi\text{ iff } \mathcal{M}, w\models_{\mathsf{C}}^+ \psi\label{Cl:neg-}\tag{$\NEG-$}\\
	\mathcal{M}, w&\models_{\mathsf{C}}^+ \psi \to \chi \text{ iff } (\forall v \geq w)(\mathcal{M}, v\models_{\mathsf{C}}^+ \psi\text{ implies }\mathcal{M}, v\models_{\mathsf{C}}^+ \chi)\label{Cl:im+}\tag{$\to+$}\\
	\mathcal{M}, w&\models_{\mathsf{C}}^- \psi \to \chi \text{ iff } (\forall v \geq w)(\mathcal{M}, v\models_{\mathsf{C}}^+ \psi\text{ implies }\mathcal{M}, v\models_{\mathsf{C}}^- \chi)\label{Cl:im-}\tag{$\to-$}	
\end{align}
This logic is known to be axiomatizable; more precisely, it was proven in \cite{w}, that we have $\mathsf{C} = \mathtt{C}[PL]$, where $\mathtt{C} = \mathtt{S}_0+(\eqref{Ax:double-neg},\ldots,\eqref{Ax:to-neg};)$ such that
\begin{align}
	\NEG\NEG\phi&\leftrightarrow\phi\label{Ax:double-neg}\tag{$\alpha_9$}\\
	\NEG(\phi\wedge\psi)&\leftrightarrow(\NEG\phi\vee\NEG\psi)\label{Ax:wedge-neg}\tag{$\alpha_{10}$}\\
	\NEG(\phi\vee\psi)&\leftrightarrow(\NEG\phi\wedge\NEG\psi)\label{Ax:vee-neg}\tag{$\alpha_{11}$}\\
	\NEG(\phi\to\psi)&\leftrightarrow(\phi\to\NEG\psi)\label{Ax:to-neg}\tag{$\alpha_{12}$}	
\end{align}
As usual, the equivalence $\phi\leftrightarrow\psi$ abbreviates $(\phi\to\psi)\wedge(\psi\to\phi)$.
 
Let us quickly survey the properties of $\mathsf{C}$. First of all, the monotonicity property stated in Definition \ref{D:Krip} for the atoms can be lifted to arbitrary propositional formulas; in other words, we have the following
\begin{lemma}\label{L:monotonicity}
	Given a $\phi \in \mathcal{PL}$, a $\star \in \{+, -\}$, an $\mathcal{M} \in \mathbb{P}$, and any $w,v \in W$ such that $w \leq v$, $\mathcal{M}, w\models_{\mathsf{C}}^\star \phi$ implies $\mathcal{M}, v\models_{\mathsf{C}}^\star \phi$. 
\end{lemma}
Lemma \ref{L:monotonicity} is proved by a straightforward induction on the complexity of $\phi \in \mathcal{PL}$. Although relatively trivial, Lemma \ref{L:monotonicity} is instrumental in establishing several properties of $\mathsf{C}$ that allow to describe it as a \textit{constructive} logic. More precisely, we can show 
\begin{proposition}\label{P:C-constructivity}
	$\mathsf{C}$ has both DP and CFP.
\end{proposition}
\begin{proof}[Proof (a sketch)]
	The proof of DP is standard: for a given $\phi_1\vee\phi_2 \in \mathcal{PL}$, we assume $(\mathcal{M}_1, w_1), (\mathcal{M}_2, w_2) \in Pt(\mathbb{P})$ such that, for all $i \in \{1,2\}$ we have $\mathcal{M}_i,w_i\not\models_{\mathsf{C}}^+\phi_i$. We then form $\mathcal{M}$ by taking the disjoint union of $\mathcal{M}_1$ and $\mathcal{M}_2$, and adding a fresh world $w$ as an immediate predecessor to $w_1$ and $w_2$ and a successor to none. Next, we show the following:
	
	\textit{Claim}. For every $i \in \{1,2\}$, every $v \in W_i$, every $\psi \in \mathcal{PL}$, and every $\star \in \{+, -\}$, we have $\mathcal{M}_i, v\models_{\mathsf{C}}^\star \psi$ iff $\mathcal{M}, v\models_{\mathsf{C}}^\star \psi$.
	
	The proof of the claim proceeds by induction on the construction of $\psi$. Now, this Claim implies that $\mathcal{M},w_i\not\models_{\mathsf{C}}^+\phi_i$ for all $i \in \{1,2\}$, whence $\mathcal{M},w\not\models_{\mathsf{C}}^+\phi_1\vee\phi_2$ by Lemma \ref{L:monotonicity}.
	
	In order to show CFP, we repeat this proof, now choosing $(\mathcal{M}_1, w_1), (\mathcal{M}_2, w_2) \in Pt(\mathbb{P})$ in such a way that $\mathcal{M}_i,w_i\not\models_{\mathsf{C}}^-\phi_i$ for $i \in \{1,2\}$.
\end{proof}
$\mathsf{C}$ thus has a claim to constructivity, as it shares DP with such contructive logics as intuitionistic logic and Nelson's logic of strong negation.\footnote{In fact, $\mathsf{C}$ is known to conservatively extend the $\NEG$-free fragment of intuitionistic logic.}
Some other properties of $\mathsf{C}$, however, are not so easily found among other non-classical logics. For example, the following proposition shows that every set of propositional formulas can be satisfied in $\mathsf{C}$:
\begin{proposition}\label{P:universal-c-satisfaction}
	Every $\Gamma\subseteq \mathcal{PL}$ is satisfiable in $\mathsf{C}$, in other words, $(\Gamma, \emptyset) \notin \mathsf{C}$.
\end{proposition}
\begin{proof}[Proof (a sketch)]
	Consider $\mathcal{M}:= (\{w\}, \{(w,w)\}, \{(p, \{w\}\mid p \in Prop)\}, \{(p, \{w\}\mid p \in Prop)\})\in \mathbb{P}$. An easy induction on the construction of the formula shows that we have $\mathcal{M}, w\models_{\mathsf{C}}^+ \phi$ for every $\phi \in \mathcal{PL}$. Therefore $(\Gamma, \emptyset) \notin \mathsf{C}$ for every $\Gamma \subseteq \mathcal{PL}$.
\end{proof}
Furthermore, most of the well-known non-classical propositional logics, constructive or otherwise, are sublogics of classical propositional logic $\mathsf{CL}$. However, this is not the case for $\mathsf{C}$, which, for this reason, provides us with a rare example of \textit{contraclassical} logic. The contraclassicality of $\mathsf{C}$ is already witnessed by the right-to-left direction of \eqref{Ax:to-neg} which is clearly invalid in $\mathsf{CL}$. In fact, it can be shown that the example of \eqref{Ax:to-neg}, which lays down a principle of interaction between $\to$ and $\NEG$ in $\mathsf{C}$, is very typical in the sense that every classically invalid theorem of $\mathsf{C}$ must contain occurrences of both $\to$ and $\NEG$. The contraclassicality of $\mathsf{C}$ is therefore due to its view of negation-implication interaction, and this interaction has a clear connexive flavor in that the above-mentioned axiomatic scheme \eqref{Ax:to-neg} is just the conjunction of (BT$\to$) and (CBT$\to$). The validity of other connexive principles is also easily established so that we get:
\begin{proposition}\label{P:c-connexivity}
	$\to$ is fully hyperconnexive in $\mathsf{C}$.
\end{proposition}
\begin{proof}[Proof (a sketch)] As mentioned above, (BT$\to$) and (CBT$\to$) follow from \eqref{Ax:to-neg}, whence, by \eqref{E:mp}, we also get (WBT$\to$) and (WCBT$\to$). As for (AT$\to$), by $\mathtt{C}\supseteq\mathtt{S}_0$ we can prove $\NEG\phi\to\NEG\phi$, whence, by \eqref{E:mp} and \eqref{Ax:to-neg}, we get that $\NEG(\NEG\phi\to\phi)$.

Finally, a simple semantic argument shows that both $(p_1\to p_2)\to (p_2\to p_1)\notin\mathsf{C}$ and $(p_1\to p_2)\not\models_{\mathsf{C}} (p_2\to p_1)$, thus establishing (nonSym$\to$) and (WnonSym$\to$), respectively.
\end{proof}
Another consequence of the connexive character of $\mathsf{C}$ is its strong degree of paraconsistency:
\begin{proposition}\label{P:c-negation-non-trivial}
	$\mathsf{C}$ is both non-trivial and negation-inconsistent.
\end{proposition}
\begin{proof}[Proof (a sketch)]
	For non-triviality, recall that $(p_1\to p_2)\not\models_{\mathsf{C}} (p_2\to p_1)$ is easily established semantically. For negation-inconsistency, note that the following formula is valid in $\mathsf{C}$:
		\begin{equation}\label{Contr}
		((p_1\wedge\NEG p_1)\to p_1)\wedge\NEG((p_1\wedge\NEG p_1)\to p_1)
	\end{equation}
	 Indeed, the first conjunct is an instance of ($\alpha_3$) and the second one is obtained from $(p_1\wedge\NEG p_1)\to \NEG p_1$, which is an instance of ($\alpha_4$), by applying \eqref{Ax:to-neg} and \eqref{E:mp}.
\end{proof}
Now, $\mathsf{CL}$ is well-known to be Post-complete, in other words, any addition of a new principle to it results in a trivial logic. In view of Proposition \ref{P:c-negation-non-trivial}, $\mathsf{C}$ must fail some classic propositional theorems and $\mathsf{CL}\setminus\mathsf{C}$ must be non-empty. In fact, we will see below that $\mathsf{C}$ even happens to fail some important intuitionistic principles.

We start by introducing two further connectives definable over $PL$:
\begin{itemize}
	\item $\phi\Rightarrow \psi$, called \textit{strong implication} and understood as an abbreviation for $(\phi\to\psi)\wedge(\NEG\psi\to\NEG\phi)$.
	
	\item $\phi\Leftrightarrow \psi$, called \textit{strong equivalence} and understood as an abbreviation for $(\phi\Rightarrow\psi)\wedge(\psi\Rightarrow\phi)$ (or, equivalently over $\mathsf{C}$, for $(\phi\leftrightarrow\psi)\wedge(\NEG\phi\leftrightarrow\NEG\psi)$).
\end{itemize}
Introduction of $\Rightarrow$ and $\Leftrightarrow$ into intuitionistic or classical logic would have been meaningless as they are clearly intuitionistically equivalent to $\to$ and $\leftrightarrow$, respectively. This, however, is far from being the case in $\mathsf{C}$, as the following proposition shows:
\begin{proposition}\label{P:c-strong-implication}
	Let $p, q \in Prop$ and let $\phi, \psi, \theta\in \mathcal{PL}$. Then the following statements hold for $\gg \in \{\to, \Rightarrow\}$:
	\begin{enumerate}
		\item $(\phi \Rightarrow \psi)\gg(\NEG\psi\Rightarrow\NEG\phi)\in \mathsf{C}$.
		
		\item However, the same cannot be said about $\to$, since we have $(p \rightarrow q)\gg(\NEG q\rightarrow\NEG p)\notin \mathsf{C}$.
		
		\item We have $(\phi \Rightarrow \psi)\gg(\phi\rightarrow\psi)\in \mathsf{C}$, but not vice versa, so that $\Rightarrow$ is stronger than $\to$.
		
		\item If $(\phi\Leftrightarrow\psi) \in \mathsf{C}$ then also $(\theta[\phi/p]\Leftrightarrow\theta[\psi/p]) \in \mathsf{C}$.
		
		\item However, the same cannot be said about $\leftrightarrow$ since we have, e.g. $((p \wedge q) \to p)\leftrightarrow(p \to p) \in \mathsf{C}$, but $\NEG(((p \wedge q) \to p)\leftrightarrow \NEG(p \to p) \notin \mathsf{C}$.
	\end{enumerate}
\end{proposition}
\begin{proof}[Proof (a sketch)]
Part 1 is proven by checking the definition of $\Rightarrow$, as for Part 2, consider  $\mathcal{M}_0\in \mathbb{P}$, where $W_0 = \{w\}$, $\leq_0 = \{(w,w)\}$, $V_0^+(p) = V_0^+(q) = V_0^-(q) = W_0$, and $V_0^-(p) = V_0^\ast(r) = \emptyset$ for every $r \in Prop\setminus\{p,q\}$ and every $\ast\in \{+, -\}$. Then we have $\mathcal{M}_0, w\not\models^+_{\mathsf{C}}(p \rightarrow q)\to(\NEG q\rightarrow\NEG p)$, hence also $\mathcal{M}_0, w\not\models^+_{\mathsf{C}}(p \rightarrow q)\Rightarrow(\NEG q\rightarrow\NEG p)$.

As for Part 3, its positive statement is easily checked, and for the model $\mathcal{M}_0$ constructed for Part 2 we also have that $\mathcal{M}_0, w\not\models^+_{\mathsf{C}}(p \rightarrow q)\to(p\Rightarrow q)$, hence also $\mathcal{M}_0, w\not\models^+_{\mathsf{C}}(p \rightarrow q)\Rightarrow(p\Rightarrow q)$.

As for Part 4, note that $(\phi\Leftrightarrow\psi) \in \mathsf{C}$ means that we have $\|\phi\|_\mathcal{M} = \|\psi\|_\mathcal{M}$ for every $\mathcal{M}\in \mathbb{P}$, so that one easily proves the statement by induction on the construction of $\theta\in\mathcal{PL}$.

Finally, we briefly address Part 5. That we have $((p \wedge q) \to p)\leftrightarrow(p \to p) \in \mathsf{C}$ is a consequence of $\mathtt{C}\supseteq \mathtt{S}_0$. As for $\NEG(((p \wedge q) \to p)\leftrightarrow\NEG(p \to p) \notin \mathsf{C}$, consider $\mathcal{M}_1\in \mathbb{P}$, where $W_1 = \{w,v\}$, $\leq_1$ is the reflexive closure of $\{(w,v)\}$, $V_1^+(p) = W$, $V_1^-(p) = V_1^+(q) = \{v\}$, and $V_1^-(q) = V_2^\ast(r) = \emptyset$ for every $r \in Prop\setminus\{p,q\}$ and every $\ast\in \{+, -\}$. Then we must have $\mathcal{M}_1, w\not\models^+_{\mathsf{C}}\NEG(p\to p)$ (the counterexample is given by $w$ itself) but $\mathcal{M}_1, w\models^+_{\mathsf{C}}\NEG(((p \wedge q) \to p)$.
\end{proof}
Proposition \ref{P:c-strong-implication} shows that in the context of $\mathsf{C}$ both $\Rightarrow$ and $\Leftrightarrow$ provide us with stronger, and contraposing versions of the rather weak $\to$ and $\leftrightarrow$. In particular, the consideration of $\Leftrightarrow$ within logics like $\mathsf{C}$ is indispensable since $\Leftrightarrow$ defines the substitutivity format for the logic.

The following lemma lists some of the properties of $\Leftrightarrow$ in $\mathsf{C}$:
\begin{lemma}\label{L:c-strong-properties}
	For all $\phi, \psi, \chi \in \mathcal{PL}$, the following formulas are theorems of $\mathsf{C}$:
	\begin{align}
		&\phi\Leftrightarrow\phi\label{E:t-refl}\tag{t1}\\
		&(\phi\Leftrightarrow\psi)\to(\psi\Leftrightarrow\phi)\label{E:t-sym}\tag{t2}\\
		&((\phi\Leftrightarrow\psi)\wedge(\psi\Leftrightarrow\chi))\to(\phi\Leftrightarrow\chi)\label{E:t-trans}\tag{t3}\\
		&\phi\Leftrightarrow\NEG\NEG\phi\label{E:t-double-neg}\tag{t4}\\
		&\NEG(\phi\wedge\psi)\Leftrightarrow(\NEG\phi\vee\NEG\psi)\label{E:t-neg-wedge}\tag{t5}\\
		&\NEG(\phi\vee\psi)\Leftrightarrow(\NEG\phi\wedge\NEG\psi)\label{E:t-neg-vee}\tag{t6}\\
		&\NEG(\phi\to\psi)\Leftrightarrow(\phi\to\NEG\psi)\label{E:t-neg-to}\tag{t7}
	\end{align}
\end{lemma}
We omit the obvious proof.

Note, however, that we need to be careful and avoid relating the failure of contraposition to the connexivity of $\to$, since the following proposition is well-known and can be easily established:
\begin{proposition}\label{P:c-strong-connexivity}
	$\Rightarrow$ is fully connexive in $\mathsf{C}$ (even though it is neither weakly nor plainly hyperconnexive in this logic).
\end{proposition}
\begin{proof}[Proof (a sketch)]
	As for the connexivity of $\Rightarrow$, the proof is similar to the proof of Proposition \ref{P:c-connexivity} but is longer and more tedious. As for the failure of hyperconnexivity, recall the model $\mathcal{M}_0 \in \mathbb{P}$ from the proof of Proposition \ref{P:c-strong-implication} and note that we have $\mathcal{M}_0, w \models_{\mathsf{C}}^+\NEG(p\Rightarrow q)$, but also $\mathcal{M}_0, w \not\models_{\mathsf{C}}^+p\Rightarrow \NEG q$. Thus, both (CBT$\Rightarrow$) and (WCBT$\Rightarrow$) fail in $\mathsf{C}$.
\end{proof}
We note in passing that some negation-inconsistencies of $\mathsf{C}$ allow for a uniform replacement of $\to$ with $\Rightarrow$; for instance, \eqref{Contr} has the following ($\Rightarrow$)-counterpart among the theorems of $\mathsf{C}$:
\begin{equation}\label{Contr-strong}
	((p_1\wedge\NEG p_1)\Rightarrow p_1)\wedge\NEG((p_1\wedge\NEG p_1)\Rightarrow p_1)
\end{equation}
To sum up, it is clear that, whereas many important properties of $\mathsf{C}$ are shared with other non-classical logics, it is the combination of constructivity, hyperconnexivity, contraclassicality, and non-trivial negation-inconsistency that is most closely related to the $\mathsf{C}$'s reading of implication, where the first property appears to be inherited from its (intuitionistic) reading of the true implications and the remaining ones are enforced by its reading of false implications.

\section{Modal logic $\mathsf{CnK}$}\label{S:modal}
\subsection{Definition and axiomatization}\label{sub:modal-axioms}
We are going to define $\mathsf{CnK}$, a basic modal logic expanding $\mathsf{C}$ to $MD$. We start by defining the corresponding notion of a Kripke bi-valuational model:
\begin{definition}\label{D:modal-model}
	A \textit{modal Fischer-Servi} model is a structure $\mathcal{M}\in \mathbb{M}$ which satisfies both of the following Fischer-Servi completion patterns\footnote{Originally introduced in \cite{fischer-servi}.}:
	\begin{align}
		(\leq^{-1}\circ R) &\subseteq (R\circ\leq^{-1})\label{Cond:1}\tag{c1}\\
		(R\circ\leq) &\subseteq (\leq\circ R)\label{Cond:2}\tag{c2}
	\end{align}	
	The class of all Fischer-Servi modal models will be denoted by $\mathbb{FSM}$.
\end{definition}
Conditions \eqref{Cond:1} and \eqref{Cond:2} can be reformulated as requirements to complete the dotted parts of each of the two diagrams given in Figure \ref{Fig:completion-patterns} once the respective straight-line part is given.
\begin{figure}
	\begin{center}
		\begin{tikzpicture}[scale=.7]
			\node (w1) at (-4,-1) {$w$};
			\node (w'1) at (-4,2) {$w'$};
			\node (v1) at (-1,-1) {$v$};
			\node (v'1) at (-1,2) {$v'$};
			\draw[-latex] (w1) to  node[midway, below] {$R$} (v1);
			\draw[-latex] (w1) to node[midway, left] {$\leq$} (w'1);
			\draw[dotted, -latex] (w'1) to  node[midway, above] {$R$} (v'1);
			\draw[dotted, -latex] (v1) to  node[midway, right] {$\leq$} (v'1);
			\node (w2) at (2,-1) {$w$};
			\node (w'2) at (2,2) {$w'$};
			\node (v2) at (5,-1) {$v$};
			\node (v'2) at (5,2) {$v'$};
			\draw[-latex] (w2) to  node[midway, below] {$R$} (v2);
			\draw[dotted, -latex] (w2) to node[midway, left] {$\leq$} (w'2);
			\draw[dotted, -latex] (w'2) to  node[midway, above] {$R$} (v'2);
			\draw[-latex] (v2) to  node[midway, right] {$\leq$} (v'2);
		\end{tikzpicture}
		\caption{}\label{Fig:completion-patterns}
	\end{center}
\end{figure}
We now define $\mathsf{CnK}:= \mathfrak{L}(MD, \mathbb{FSM}, (\models^+_m, \models^-_m))$, setting
$$
(\models^+_m, \models^-_m):= Ind(\models_{\mathsf{C}}^+\cup\{\eqref{Cl:box+},\eqref{Cl:diam+}\},\models_{\mathsf{C}}^-\cup\{\eqref{Cl:box-},\eqref{Cl:diam-}\})
$$
where we assume that:
\begin{align}
	\mathcal{M}, w&\models_{m}^+ \Box\psi \text{ iff } (\forall v \geq w)(\forall u \in W)(v\mathrel{R}u\text{ implies }\mathcal{M}, u\models_{m}^+ \psi)\label{Cl:box+}\tag{$\Box+$}\\
	\mathcal{M}, w&\models_{m}^- \Box\psi \text{ iff } (\forall v \geq w)(\forall u \in W)(v\mathrel{R}u\text{ implies }\mathcal{M}, u\models_{m}^- \psi)\label{Cl:box-}\tag{$\Box-$}\\
	\mathcal{M}, w&\models_{m}^+ \Diamond\psi \text{ iff } (\exists v \in W)(w\mathrel{R}v\text{ and }\mathcal{M}, v\models_{m}^+ \psi)\label{Cl:diam+}\tag{$\Diamond+$}\\
	\mathcal{M}, w&\models_{m}^- \Diamond\psi \text{ iff } (\exists v \in W)(w\mathrel{R}v\text{ and }\mathcal{M}, v\models_{m}^- \psi)\label{Cl:diam-}\tag{$\Diamond-$}	
\end{align}
 We will write $\Gamma\models_m\Delta$ and $\mathcal{M}, w \models_m (\Gamma, \Delta)$, meaning $\Gamma\models_{\mathsf{CnK}}\Delta$ and $\mathcal{M}, w \models_{\mathsf{CnK}} (\Gamma, \Delta)$, respectively.

In this subsection, we obtain a sound and (strongly) complete axiomatization of $\mathsf{CnK}$. We consider the Hilbert-style axiomatic system $\mathtt{CnK}$, for which\footnote{Probably the most contentious part of this axiomatization is given by the schemes \eqref{E:am5} and \eqref{E:am6}, which replace the usual duality principles of classical modal logic. However, it is easy to see that these principles are necessitated by the semantics of $\mathsf{CnK}$. In fact, they constitute the main input of $\mathsf{CnK}$ to the discussion of modal principles in the context of $\mathsf{C}$: replacing \eqref{E:am5} and \eqref{E:am6} with the classical dualities relating $\Box$ and $\Diamond$ simply brings us back to the first $\mathsf{C}$-based modal system ever proposed in the existing literature, namely the modal logic of \cite{w}.} we set $\mathtt{CnK}:= \mathtt{C}+(\eqref{E:am1}-\eqref{E:am6};\eqref{E:Rnec})$, where we assume that:
\begin{align}
	\Box(\phi\to\psi)&\to(\Box\phi\to\Box\psi)\label{E:am1}\tag{$\beta_1$}\\
	\Box(\phi\to\psi)&\to(\Diamond\phi\to\Diamond\psi)\label{E:am2}\tag{$\beta_2$}\\
	\Diamond(\phi\vee\psi)&\to(\Diamond\phi\vee\Diamond\psi)\label{E:am3}\tag{$\beta_3$}\\
	(\Diamond\phi\to\Box\psi)&\to\Box(\phi\to\psi)\label{E:am4}\tag{$\beta_4$}\\
	\NEG\Box\phi&\leftrightarrow\Box\NEG\phi\label{E:am5}\tag{$\beta_5$}\\
	\NEG\Diamond\phi&\leftrightarrow\Diamond\NEG\phi\label{E:am6}\tag{$\beta_6$}\\
	\text{From }\phi &\text{ infer }\Box\phi\label{E:Rnec}\tag{nec}	
\end{align}

Before we go on to prove the soundness and completeness of $\mathtt{CnK}$ relative to $\mathsf{CnK}$, we would like to quickly address the relations between $\mathtt{CnK}$ and $\mathsf{C}$: 
\begin{lemma}\label{L:c}
	The following statements hold:
	\begin{enumerate}
		\item If $\Gamma, \Delta \subseteq\mathcal{PL}$ are such that $\Gamma\models_{\mathsf{C}}\Delta$, and  $\Gamma', \Delta' \subseteq \mathcal{MD}$ are obtained from $\Gamma, \Delta$ by a simultaneous substitution of $\mathcal{MD}$-formulas for variables, then $\Gamma'\vdash_{\mathtt{CnK}}\Delta'$. Moreover, Deduction Theorem holds for $\mathtt{CnK}$ in that for all $\Gamma \cup \{\phi,\psi\}\subseteq \mathcal{MD}$ we have $\Gamma \vdash_{\mathtt{CnK}} \phi\to\psi$ iff $\Gamma, \phi\vdash_{\mathtt{CnK}} \psi$.
		
		\item If $\phi \in \mathcal{PL}$, then $\vdash_{\mathtt{CnK}} \phi$ iff $\phi\in \mathsf{C}$.
	\end{enumerate}
\end{lemma}
\begin{proof}[Proof (a sketch)]
	Part 1 is trivial. As for Part 2, its ($\Leftarrow$)-part is also trivial, and its ($\Rightarrow$)-part follows from the observation that, given a proof of $\phi \in \mathcal{PL}$ in ${\mathtt{CnK}}$ we can turn it into a proof in $\mathtt{C}$ by omitting all the boxes and diamonds from its formulas. Then all the instances of \eqref{E:am1}--\eqref{E:am6} turn into $\mathtt{C}$-provable instances of $\phi\to\phi$ and every application of \eqref{E:Rnec} turns into a repetition of its premise. The last formula of this proof is still  $\phi \in \mathcal{PL}$, since none of its subformulas gets replaced. 
\end{proof}
Turning now to the relations between $\mathtt{CnK}$ and $\mathsf{CnK}$, we observe, first, that  $\mathtt{CnK}$ only allows us to deduce theorems of $\mathsf{CnK}$:
\begin{lemma}\label{L:modal-soundness}
	For every $\phi\in\mathcal{MD}$, if $\vdash_{\mathtt{CnK}}\phi$, then $\phi\in\mathsf{CnK}$.
\end{lemma}
Its proof is straightforward to obtain and is therefore omitted. We are now going to show the converse of Lemma \ref{L:modal-soundness}, and we start our work by proving some theorems and derived rules in $\mathtt{CnK}$, which we collect in the following lemma: 
\begin{lemma}\label{L:mod-theorems}
	Let $\phi, \psi\in \mathcal{MD}$ and let $\mu \in \{\Box, \Diamond\}$. The following theorems and derived rules can be deduced in $\mathtt{CnK}$:
	\begin{align}
		\phi \to \psi &\vDdash \mu\phi \to \mu\psi\label{E:Rmod-box}\tag{RM$\mu$}\\
		\NEG\mu\phi&\Leftrightarrow\mu\NEG\phi\label{E:tm0}\tag{T0$\mu$}\\
		(\Box\phi\wedge\Box\psi)&\leftrightarrow\Box(\phi\wedge\psi)\label{E:tm1}\tag{T1}\\
		\Diamond(\phi\to\psi)&\to(\Box\phi\to\Diamond\psi)\label{E:tm4}\tag{T2}
	\end{align}
\end{lemma} 
\begin{proof}[Proof (a sketch)]
(RM$\Box$): The proof is the same as in the case of the well-known basic classical modal logic $\mathsf{K}$. (RM$\Diamond$): Same as for (RM$\Box$), but use \eqref{E:am2} instead of \eqref{E:am1}. 

Next, (T0$\Box$) (resp. (T0$\Diamond$)) is entailed by \eqref{E:am5} (resp. \eqref{E:am6}) and \eqref{E:t-double-neg}, the proof of \eqref{E:tm1} is the same as in $\mathsf{K}$, and \eqref{E:tm4} is proved by applying to $\phi\to((\phi\to\psi)\to \psi)$ first \eqref{E:Rnec} and then \eqref{E:am2}.  
\end{proof}

For the rest of this subsection we will write (in)consistency (resp. completeness, maximality) meaning $\mathtt{CnK}$-(in)consistency (resp. $MD$-completeness, $\mathtt{CnK}[MD]$-maximality). Next, we observe that Lemma \ref{L:c}.1 allows for the following equivalent definition of inconsistency:
\begin{lemma}\label{L:mod-alt-consistency}
	A bi-set $(\Gamma, \Delta)\in \mathcal{P}(\mathcal{MD})\times\mathcal{P}(\mathcal{MD})$ is $\mathtt{CnK}$-inconsistent iff, for some $m,n\in \omega$ some $\phi_1,\ldots,\phi_n\in \Gamma$ and some $\psi_1,\ldots,\psi_m\in \Delta$ we have: $
	\bigwedge^n_{i = 1}\phi_i\vdash\bigvee^m_{j = 1}\psi_j$, or, equivalently, $\vdash
	\bigwedge^n_{i = 1}\phi_i\to\bigvee^m_{j = 1}\psi_j$.
\end{lemma}
The next two lemmas present some properties of the consistent and maximal bi-sets, respectively:
\begin{lemma}\label{L:mod-consistent}
	Let	$(\Gamma, \Delta)\in \mathcal{P}(\mathcal{MD})\times\mathcal{P}(\mathcal{MD})$ be consistent. Then the following statements hold:
	\begin{enumerate}
		\item For every $\phi \in \mathcal{MD}$, either $(\Gamma \cup \{\phi\}, \Delta)$ or $(\Gamma, \Delta\cup \{\phi\})$ is consistent.
		
		\item For every $\phi \to \psi \in \Delta$, $(\Gamma \cup \{\phi\}, \{\psi\})$ is consistent.
		
		\item For every $\Box\phi\in \Delta$, the bi-set $(\{\psi\mid\Box\psi \in \Gamma\},\{\phi\})$ is consistent.
		
		\item For every $\Diamond\phi\in \Gamma$, the bi-set $(\{\phi\}\cup\{\psi\mid\Box\psi \in \Gamma\},\{\chi\mid\Diamond\chi\in\Delta\})$ is consistent.
	\end{enumerate}
\end{lemma}
\begin{proof}
	Parts 1 and 2 are proved as in the case of $\mathsf{C}$ (in which respect $\mathsf{C}$ also just repeats the similar reasoning for intuitionistic logic). As for Part 3, assume that  $\Box\phi\in \Delta$, and assume, towards  contradiction, that the bi-set $(\{\psi\mid\Box\psi \in \Gamma\},\{\phi\})$ is inconsistent. Then there must be $\Box\psi_1,\ldots,\Box\psi_n \in \Gamma$ such that, for $\psi:= \bigwedge^n_{i= 1}\psi_i$, we have $\psi\vdash_{\mathtt{CnK}} \phi$, hence also $\vdash_{\mathtt{CnK}}\psi\to\phi$ by Lemma \ref{L:c}.1 and $\vdash_{\mathtt{CnK}}\Box\psi\to\Box\phi$ by \eqref{E:Rmod-box}. Next, \eqref{E:tm1} implies that $\Gamma\vdash_{\mathtt{CnK}}\Box\psi$, whence also $\Gamma\vdash_{\mathtt{CnK}}\Box\phi$. But then  the assumption that $\Box\phi \in \Delta$ clearly contradicts the consistency of $(\Gamma,\Delta)$. The obtained contradiction shows that $(\{\chi\mid\phi\boxto\chi \in \Gamma\},\{\psi\})$ must be consistent.
	
	Finally, as for Part 4, assume that $\Diamond\phi \in \Gamma$, and assume, towards contradiction, that the bi-set $(\{\phi\}\cup\{\psi\mid\Box\psi \in \Gamma\},\{\chi\mid\Diamond\chi\in\Delta\})$ is inconsistent. Then there must be some $\Box\psi_1,\ldots,\Box\psi_n \in \Gamma$ and $\Diamond\chi_1,\ldots,\Diamond\chi_m \in \Delta$ such that, for $\psi:= \bigwedge^n_{i= 1}\psi_i$ and $\chi:= \bigvee^m_{j= 1}\chi_j$,  we have $\psi, \phi\vdash_{\mathtt{CnK}} \chi$. By \eqref{E:tm1}, we  get $\Gamma\vdash_{\mathtt{CnK}} \Box\psi$. On the other hand, by Lemma \ref{L:c}.1, we get $\vdash_{\mathtt{CnK}}\psi\to(\phi\to\chi)$, whence, by \eqref{E:Rmod-box}, it follows that $\vdash_{\mathtt{CnK}}\Box\psi\to\Box(\phi\to\chi)$. Hence, $\Gamma\vdash_{\mathtt{CnK}} \Box(\phi\to\chi)$, which entails, by \eqref{E:am2}, that $\Gamma\vdash_{\mathtt{CnK}} \Diamond\phi\to\Diamond\chi$, and, by the choice of $\phi$, that $\Gamma\vdash_{\mathtt{CnK}} \Diamond\chi$. Now \eqref{E:am3} yields that $\Gamma\vdash_{\mathtt{CnK}} \Diamond\chi_1\vee\ldots\vee\Diamond\chi_m$, which contradicts the consistency of $(\Gamma, \Delta)$.
\end{proof} 
\begin{lemma}\label{L:mod-maximal}
	Let	$(\Gamma, \Delta), (\Gamma_0,\Delta_0), (\Gamma_1,\Delta_1) \in \mathcal{P}(\mathcal{MD})\times\mathcal{P}(\mathcal{MD})$ be maximal bi-sets such that $\{\phi\mid \Box\phi\in \Gamma_0\}\subseteq \Gamma_1$, and $\{\Diamond\phi\mid \phi\in \Gamma_1\}\subseteq \Gamma_0$, and let $\phi,\psi\in\mathcal{MD}$. Then the following statements are true:
	\begin{enumerate}
		\item If $\Gamma\vdash\phi$, then $\phi\in \Gamma$.
		
		\item $\phi\wedge\psi\in\Gamma$ iff $\phi, \psi\in \Gamma$.
		
		\item $\phi\vee\psi \in \Gamma$ iff $\phi \in \Gamma$ or $\psi\in\Gamma$.
		
		\item If $\phi\to\psi, \phi \in \Gamma$, then $\psi \in \Gamma$.
		
%
%
		
		\item If $\Gamma_0 \subseteq \Gamma$,  then
		$(\Gamma_1 \cup \{\phi\mid \Box\phi\in \Gamma\}, \{\phi\mid \Diamond\phi\in \Delta\})$ is consistent.
		
		\item If $\Gamma_1 \subseteq \Gamma$, then $(\Gamma_0 \cup \{\Diamond\phi\mid \phi\in \Gamma\}, \{\Box\phi\mid \phi\in \Delta\})$ is consistent.
	\end{enumerate}
\end{lemma}
\begin{proof}
	The Parts 1--4 are handled as in the case of intuitionistic logic. As for Part 5, assume its hypothesis and suppose, towards contradiction, that $(\Gamma_1 \cup \{\phi\mid \Box\phi\in \Gamma\}, \{\phi\mid \Diamond\phi\in \Delta\})$  is inconsistent. Then there must exist some $\phi_1,\ldots,\phi_n \in \Gamma_1$, $\Box\psi_1,\ldots,\Box\psi_m\in \Gamma$ and some $\Diamond\chi_1,\ldots,\Diamond\chi_k\in \Delta$, such that, for $\phi:= \bigwedge^n_{i=1}\phi_i$, $\psi:= \bigwedge^m_{j=1}\psi_j$, and $\chi:= \bigvee^k_{r=1}\chi_r$ we have $\phi,\psi\vdash_{\mathtt{CnK}}\chi$. But then, by Lemma \ref{L:c}.1, $\phi\vdash_{\mathtt{CnK}}\psi\to\chi$, whence, by Part 1, $\psi\to\chi\in \Gamma_1$. Recall that we assume $\{\Diamond\phi\mid \phi\in \Gamma_1\}\subseteq \Gamma_0\subseteq \Gamma$, which entails that $\Diamond(\psi\to\chi)\in \Gamma$. By \eqref{E:tm4}, $\Gamma\vdash_{\mathtt{CnK}} \Box\psi\to\Diamond\chi$, whereas, by \eqref{E:tm1} $\Gamma\vdash_{\mathtt{CnK}} \Box\psi$. Now we have $\Gamma\vdash_{\mathtt{CnK}} \Diamond\chi$, whence \eqref{E:am3} implies $\Gamma\vdash_{\mathtt{CnK}} \Diamond\chi_1\vee\ldots\vee\Diamond\chi_m$, which contradicts the consistency of $(\Gamma, \Delta)$. 
		
	For Part 6,  assume its hypothesis and suppose that $(\Gamma_0 \cup \{\Diamond\phi\mid \phi\in \Gamma\}, \{\Box\phi\mid \phi\in \Delta\})$ is inconsistent. Then there must exist some $\phi_1,\ldots,\phi_n \in \Gamma_0$, $\psi_1,\ldots,\psi_m\in \Gamma$ and some $\chi_1,\ldots,\chi_k\in \Delta$, such that $\bigwedge^n_{i=1}\phi_i,\bigwedge^m_{j = 1}\Diamond\psi_j\vdash_{\mathtt{CnK}}\bigvee^k_{r = 1}\Box\chi_r$.
	Again, we set $\phi:= \bigwedge^n_{i=1}\phi_i$, $\psi:= \bigwedge^m_{j=1}\psi_j$, and $\chi:= \bigvee^k_{r=1}\chi_r$, and reason as follows. By Lemma \ref{L:c}.1, we get $\Gamma_0\vdash_{\mathtt{CnK}}\bigwedge^m_{j = 1}\Diamond\psi_j\to \bigvee^k_{r = 1}\Box\chi_r$, whence, by \eqref{E:Rmod-box} and ($\alpha_3$)--($\alpha_6$), we obtain that $\Gamma_0\vdash_{\mathtt{CnK}}\Diamond\psi\to \Box\chi$. Applying now \eqref{E:am4}, we get that $\Gamma_0\vdash_{\mathtt{CnK}}\Box(\psi\to \chi)$, whence, by Part 1 and the choice of $\Gamma_1$ it follows that $\Box(\psi\to \chi) \in \Gamma_0$ and $\psi\to \chi \in \Gamma_1\subseteq \Gamma$. Since clearly $\Gamma\vdash_{\mathtt{CnK}}\psi$, the latter contradicts the consistency of $(\Gamma, \Delta)$.
\end{proof}
We observe, next, that we can use the usual Lindenbaum construction to extend every consistent bi-set to a maximal one:
\begin{lemma}\label{L:mod-lindenbaum}
	Let $(\Gamma, \Delta)\in \mathcal{P}(\mathcal{MD})\times\mathcal{P}(\mathcal{MD})$ be consistent. Then there exists a maximal $(\Xi, \Theta)\in \mathcal{P}(\mathcal{MD})\times\mathcal{P}(\mathcal{MD})$ such that $\Gamma \subseteq \Xi$ and $\Delta \subseteq \Theta$. 
\end{lemma}
Next, we define the canonical model $\mathcal{M}_m$ for $\mathtt{CnK}$:
\begin{definition}\label{D:mod-canonical-model}
	The structure $\mathcal{M}_m$ is the tuple $(W_m, \leq_m, R_m, V^+_m, V^-_m)$ such that:
	\begin{itemize}
		\item $W_m:=\{(\Gamma, \Delta)\in \mathcal{P}(\mathcal{MD})\times\mathcal{P}(\mathcal{MD})\mid (\Gamma, \Delta)\text{ is maximal }\}$.
		
		\item $(\Gamma_0,\Delta_0)\leq_m(\Gamma_1,\Delta_1)$ iff $\Gamma_0\subseteq\Gamma_1$ for all $(\Gamma_0,\Delta_0),(\Gamma_1,\Delta_1)\in W_m$.
		
		\item For all $(\Gamma_0,\Delta_0),(\Gamma_1,\Delta_1)\in W_m$ we have $((\Gamma_0,\Delta_0),(\Gamma_1,\Delta_1)) \in R_m$ iff we have both:
		\begin{itemize}
			\item $\{\phi\mid\Box\phi\in \Gamma_0\}\subseteq \Gamma_1$.
			
			\item $\{\Diamond\psi\mid\psi\in \Gamma_1\}\subseteq \Gamma_0$.
		\end{itemize}
		
		\item $V^+_m(p):=\{(\Gamma,\Delta)\in W_m\mid p\in\Gamma\}$ for every $p \in Prop$.
		
		\item $V^-_m(p):=\{(\Gamma,\Delta)\in W_m\mid \NEG p\in\Gamma\}$ for every $p \in Prop$.	
	\end{itemize}
\end{definition}
First of all, we have to make sure that we have indeed just defined a model:
\begin{lemma}\label{L:mod-canonical-model}
	$\mathcal{M}_c \in \mathbb{FSM}$.	
\end{lemma}
\begin{proof}
	We show, first, that $W_m \neq \emptyset$. Indeed, choose any $p \in Prop$ and consider the bi-set $(\emptyset,\{p\})$. Clearly, $p\in\mathcal{PL}\setminus\mathsf{C}$, whence Lemma \ref{L:c} implies that $(\emptyset,\{p\})$ is consistent. Therefore, by Lemma \ref{L:mod-lindenbaum}, there must exist a maximal $(\Gamma, \Delta)\supseteq(\emptyset,\{p\})$; and we will have, by Definition \ref{D:mod-canonical-model}, that $(\Gamma, \Delta)\in W_m$.
	
	It is also clear from Definition  \ref{D:mod-canonical-model} that $\leq_m$ is a pre-order, and that $R_m\subseteq W_m\times W_m$. So it only remains to check the satisfaction of conditions \eqref{Cond:1} and \eqref{Cond:2} from Definition \ref{D:modal-model}. As for \eqref{Cond:1}, assume that $(\Gamma,\Delta), (\Gamma_0,\Delta_0), (\Gamma_1,\Delta_1) \in W_m$ are such that we have $(\Gamma,\Delta) \mathrel{\geq_m}(\Gamma_0,\Delta_0)\mathrel{R_m}(\Gamma_1,\Delta_1)$. Then, in particular, $\Gamma\supseteq\Gamma_0$. Moreover,  we have:
	\begin{align}
		&\{\phi\mid\Box\phi\in \Gamma_0\}\subseteq \Gamma_1\label{E:mod2}\\
		&\{\Diamond\psi\mid\psi\in \Gamma_1\}\subseteq \Gamma_0\label{E:mod3}
	\end{align}
	By Lemma \ref{L:mod-maximal}.5, the bi-set $(\Gamma_1 \cup \{\phi\mid \Box\phi\in \Gamma\}, \{\phi\mid \Diamond\phi\in \Delta\})$ must then be consistent, so that, by Lemma \ref{L:mod-lindenbaum}, this bi-set must be extendable to some $(\Gamma',\Delta')\in W_m$. We will have then $\Gamma'\supseteq \Gamma_1$ whence clearly $(\Gamma',\Delta')\mathrel{\geq_m}(\Gamma_1,\Delta_1)$. Now, we get $\{\phi\mid\Box\phi\in \Gamma\}\subseteq \Gamma'$ trivially by the choice of $(\Gamma',\Delta')$. Moreover, if $\psi \in \Gamma'$, then $\psi\notin\Delta'$ by the consistency of $(\Gamma',\Delta')$. But this means that we cannot have $\Diamond\psi\in\Delta$, so $\Diamond\psi\in \Gamma$ by the completeness of $(\Gamma,\Delta)$. Thus we have shown that also $\{\Diamond\psi\mid\psi\in \Gamma'\}\subseteq \Gamma$. Summing everything up, we get that $
	(\Gamma,\Delta)\mathrel{R_m}(\Gamma',\Delta')\mathrel{\geq_m}(\Gamma_1,\Delta_1)$, and condition \eqref{Cond:1} is shown to be satisfied.
	
	As for \eqref{Cond:2}, assume that $(\Gamma,\Delta), (\Gamma_0,\Delta_0), (\Gamma_1,\Delta_1)\in W_m$ are such that $(\Gamma_0,\Delta_0)\mathrel{R_m} (\Gamma_1,\Delta_1)\mathrel{\leq_m}(\Gamma,\Delta)$. Then $\Gamma\supseteq\Gamma_1$. By Lemma \ref{L:mod-maximal}.6, the bi-set $(\Gamma_0 \cup \{\Diamond\phi\mid \phi\in \Gamma\}, \{\Box\phi\mid \phi\in \Delta\})$ must then be consistent, and, by Lemma \ref{L:mod-lindenbaum}, extendable to a $(\Gamma',\Delta')\in W_m$. Clearly, $\Gamma'\supseteq \Gamma_0$ whence also $(\Gamma',\Delta')\mathrel{_m\geq}(\Gamma_0,\Delta_0)$.
 Next, assume that $\Box\phi\in \Gamma'$. If $\phi\notin\Gamma$, then $\phi \in\Delta$ by the completeness of $(\Gamma,\Delta)$, whence $\Box\phi\in \Delta'$. But the latter contradicts the consistency of $(\Gamma',\Delta')$. Therefore $\phi\in\Gamma$. Since the choice of $\phi$ was arbitrary, we have shown that $\{\phi\mid\Box\phi\in \Gamma'\}\subseteq \Gamma$. Moreover, we have $\{\Diamond\phi\mid \phi\in \Gamma\}\subseteq \Gamma'$ by the choice of $(\Gamma',\Delta')$. Thus we get that $(\Gamma_0,\Delta_0)\mathrel{\leq_c}(\Gamma',\Delta')\mathrel{R_m}(\Gamma,\Delta)$, and condition \eqref{Cond:2} is shown to be satisfied.
\end{proof}
The truth lemma for this model then looks as follows:
\begin{lemma}\label{L:mod-truth}
	For every $\phi\in\mathcal{MD}$ and for every $(\Gamma,\Delta)\in W_m$, the following statements hold:
	\begin{enumerate}
		\item $
		\mathcal{M}_m,(\Gamma,\Delta)\models_m^+\phi$ iff $\phi \in \Gamma$.
		
		\item $
		\mathcal{M}_m,(\Gamma,\Delta)\models_m^-\phi$ iff $\NEG\phi \in \Gamma$.
	\end{enumerate} 
\end{lemma}
\begin{proof}
	We prove both parts by simultaneous induction on the construction of $\phi$.
	
	\textit{Basis}. If $\phi = p \in Prop$, then the lemma holds by the definition of $\mathcal{M}_m$.
	
	\textit{Induction step}. The following cases arise:
	
	\textit{Case 1}. $\phi = \NEG\psi$. Then, as for Part 1, we have $
	\mathcal{M}_m,(\Gamma,\Delta)\models_m^+\phi$ iff $
	\mathcal{M}_m,(\Gamma,\Delta)\models_m^-\psi$ iff, by IH for Part 2, $\phi = \NEG\psi \in \Gamma$. As for Part 2, we have $
	\mathcal{M}_m,(\Gamma,\Delta)\models_m^-\phi$ iff $
	\mathcal{M}_m,(\Gamma,\Delta)\models_m^+\psi$ iff, by IH for Part 1, $\psi \in \Gamma$, iff, by \eqref{Ax:double-neg} and Lemma \ref{L:mod-maximal}.1, $\NEG\phi = \NEG\NEG\psi \in \Gamma$.
	
	\textit{Case 2}. $\phi = \psi\ast\chi$ for $\ast \in \{\wedge, \vee, \to\}$. Part 1 is proved as in the case of intuitionistic logic. Assuming this part, the parallel cases for Part 2 follow easily by the properties of negation in $\mathsf{C}$. As an example we consider the subcase when $\ast = \to$.
	
	Let $\phi = (\psi\to\chi)$. If $\NEG\phi = \NEG(\psi\to\chi)\in \Gamma$, then , by Lemma \ref{L:mod-maximal}.1 and \eqref{Ax:to-neg}, $\psi \to \NEG\chi \in \Gamma$. If $(\Gamma_0, \Delta_0)\in W_m$ is such that $(\Gamma,\Delta) \leq_m (\Gamma_0, \Delta_0)$, then $\psi \to \NEG\chi \in \Gamma\subseteq\Gamma_0$. If, furthermore, $\mathcal{M}_m,(\Gamma_0,\Delta_0)\models_m^+\psi$, then, by IH for Part 1, $\psi\in \Gamma_0$, whence, by Lemma \ref{L:mod-maximal}.4, $\NEG\chi\in \Gamma_0$, and, by IH for Part 2, $\mathcal{M}_m,(\Gamma,\Delta)\models_m^-\chi$.
	
	In the other direction, if $\NEG\phi = \NEG(\psi\to\chi)\notin \Gamma$, then, by Lemma \ref{L:mod-maximal}.1 and \eqref{Ax:to-neg}, $\psi \to \NEG\chi \notin \Gamma$, whence  $\psi \to \NEG\chi \in \Delta$ by completeness. By Lemma \ref{L:mod-consistent}.2, the bi-set $(\Gamma\cup\{\psi\},\{\NEG\chi\})$ is consistent, and thus, by Lemma \ref{L:mod-lindenbaum}, extendable to a $(\Gamma_0,\Delta_0)\in W_m$. By the choice of $(\Gamma_0,\Delta_0)$, we have $(\Gamma,\Delta) \leq_m (\Gamma_0, \Delta_0)$, $\psi\in \Gamma_0$, and $\NEG\chi\in \Delta_0$, whence, by maximality of $(\Gamma_0,\Delta_0)$, $\NEG\chi\notin\Gamma_0$. By IH, this means that both $\mathcal{M}_m,(\Gamma_0,\Delta_0)\models_m^+\psi$ and $\mathcal{M}_m,(\Gamma_0,\Delta_0)\not\models_m^-\chi$, which entails that $\mathcal{M}_m,(\Gamma_0,\Delta_0)\not\models_m^-\psi\to\chi$.

	We treat the modal cases next:
	
	\textit{Case 3}. $\phi = \Box\psi$. 
	
	\textit{Part 1}. Let $(\Gamma,\Delta)\in W_m$ be such that $\phi \in \Gamma$, and let $(\Gamma_0,\Delta_0), (\Gamma_1,\Delta_1)\in W_m$ be such that $(\Gamma,\Delta)\mathrel{\leq_m}(\Gamma_0,\Delta_0)\mathrel{R_m}(\Gamma_1,\Delta_1)$. Then $\Gamma \subseteq \Gamma_0$, so that $\Box\psi \in \Gamma_0$. 
It follows by $(\Gamma_0,\Delta_0)\mathrel{R_m}(\Gamma_1,\Delta_1)$, that $\psi\in \Gamma_1$. Next, IH implies that  $\mathcal{M}_m,(\Gamma_1,\Delta_1)\models_m^+\psi$. Since the choice of  $(\Gamma_0,\Delta_0), (\Gamma_1,\Delta_1)\in W_m$ such that $(\Gamma,\Delta)\mathrel{\leq_m}(\Gamma_0,\Delta_0)\mathrel{R_m}(\Gamma_1,\Delta_1)$ was made arbitrarily, it follows that we must have $\mathcal{M}_m,(\Gamma,\Delta)\models_m^+ \Box\psi = \phi$.
	
	Conversely, let $(\Gamma,\Delta)\in W_m$ be such that $\phi \notin \Gamma$. Then $\Box\psi \in \Delta$ by completeness of $(\Gamma, \Delta)$, and $(\{\chi\mid\Box\chi\in\Gamma\},\{\psi\})$ must be consistent by Lemma \ref{L:mod-consistent}.3. By Lemma \ref{L:mod-lindenbaum}, we can extend it to a maximal $(\Gamma',\Delta')\supseteq (\{\chi\mid\Box\chi\in\Gamma\},\{\psi\})$. Now, set $(\Gamma_0,\Delta_0):= (\Gamma\cup\{\Diamond\chi\mid\chi\in \Gamma'\}, \{\Box\theta\mid\theta\in\Delta'\})$. We claim that $(\Gamma_0,\Delta_0)$ is consistent. Otherwise, we can choose $\gamma_1,\ldots,\gamma_n\in\Gamma$, $\tau_1,\ldots,\tau_m\in\Gamma'$ and $\xi_1,\ldots,\xi_k\in\Delta'$ such that $\bigwedge^n_{i = 1}\gamma_i,\bigwedge^m_{j = 1}\Diamond\tau_j\vdash_{\mathtt{CnK}}\bigvee^k_{r = 1}\Box\xi_r$. But then, setting $\tau:= \bigwedge^m_{j = 1}\tau_j$ and $\xi:= \bigvee^k_{r = 1}\xi_r$, we have $\Gamma\vdash_{\mathtt{CnK}} \bigwedge^m_{j = 1}\Diamond\tau_j\to\bigvee^k_{r = 1}\Box\xi_r$ by Lemma \ref{L:c}.1, whence $\Gamma\vdash_{\mathtt{CnK}} \Diamond\tau\to\Box\xi$ by \eqref{E:Rmod-box} and ($\alpha_3$)--($\alpha_6$), whence further $\Gamma\vdash_{\mathtt{CnK}} \Box(\tau\to\xi)$ by \eqref{E:am4} and $\Box(\tau\to\xi)\in \Gamma$ by Lemma \ref{L:mod-maximal}.1. By the choice of $(\Gamma',\Delta')$, we get that $\tau\to\xi \in \Gamma'$ which contradicts the consistency of this bi-set and shows that $(\Gamma_0,\Delta_0)$ must have been consistent. Therefore, $(\Gamma_0,\Delta_0)$ is extendable to a maximal bi-set $(\Gamma_1,\Delta_1)\supseteq(\Gamma_0,\Delta_0)$. 
	
	We now claim that we have $(\Gamma,\Delta)\mathrel{\leq_m}(\Gamma_1,\Delta_1)\mathrel{R_m}(\Gamma',\Delta')$. The first part is trivial since we have $\Gamma_1\supseteq\Gamma_0 \supseteq\Gamma$ by the choice of $(\Gamma_1,\Delta_1)$ and $(\Gamma_0,\Delta_0)$. As for the second part, note that (a) for every $\theta\in\mathcal{MD}$, if $\Box\theta\in \Gamma_1$ and $\theta \notin \Gamma'$, then, by the completeness of $(\Gamma',\Delta')$, we must have $\theta \in \Delta'$. But then $\Box\theta\in\Delta_0\subseteq\Delta_1$, which contradicts the consistency  of $(\Gamma_1,\Delta_1)$. The obtained contradiction shows that  $\{\theta\mid\Box\theta\in \Gamma_1\}\subseteq \Gamma'$. Next, (b) we trivially get that $\{\Diamond\chi\mid\chi\in \Gamma'\}\subseteq\Gamma_0 \subseteq \Gamma_1$. Summing up (a) and (b), we get that $((\Gamma_1,\Delta_1),(\Gamma',\Delta'))\in R_m$.
	
	Since also $\psi\in \Delta'$, we have $\psi\notin\Gamma'$ by the consistency of $(\Gamma',\Delta')$ and $\mathcal{M}_m,(\Gamma',\Delta')\not\models_m^+ \psi$ by IH. Given that $(\Gamma,\Delta)\mathrel{\leq_m}(\Gamma_1,\Delta_1)\mathrel{R_m}(\Gamma',\Delta')$, we get $\mathcal{M}_m,(\Gamma',\Delta')\not\models_m^+ \Box\psi = \phi$, as desired.
	
	\textit{Part 2}. Let $(\Gamma,\Delta)\in W_m$ be such that $\NEG\phi \in \Gamma$, and let $(\Gamma_0,\Delta_0), (\Gamma_1,\Delta_1)\in W_m$ be such that $(\Gamma,\Delta)\mathrel{\leq_m}(\Gamma_0,\Delta_0)\mathrel{R_m}(\Gamma_1,\Delta_1)$. Then $\Gamma \subseteq \Gamma_0$, so that $\NEG\Box\psi \in \Gamma_0$; then \eqref{E:am5} and Lemma \ref{L:mod-maximal}.1 together imply that $\Box\NEG\psi \in \Gamma_0$.
	It follows by $(\Gamma_0,\Delta_0)\mathrel{R_m}(\Gamma_1,\Delta_1)$, that $\NEG\psi\in \Gamma_1$. Next, IH implies that  $\mathcal{M}_m,(\Gamma_1,\Delta_1)\models_m^-\psi$. Since the choice of  $(\Gamma_0,\Delta_0), (\Gamma_1,\Delta_1)\in W_m$ such that $(\Gamma,\Delta)\mathrel{\leq_m}(\Gamma_0,\Delta_0)\mathrel{R_m}(\Gamma_1,\Delta_1)$ was made arbitrarily, it follows that we must have $\mathcal{M}_m,(\Gamma,\Delta)\models_m^- \Box\psi = \phi$.
	
	Conversely, let $(\Gamma,\Delta)\in W_m$ be such that $\NEG\phi \notin \Gamma$. Then $\NEG\Box\psi \in \Delta$ by completeness of $(\Gamma, \Delta)$, which entails $\Box\NEG\psi \in \Delta$ by \eqref{E:am5} and maximality.  But then $(\{\chi\mid\Box\chi\in\Gamma\},\{\NEG\psi\})$ must be consistent by Lemma \ref{L:mod-consistent}.3. By Lemma \ref{L:mod-lindenbaum}, we can extend it to a maximal $(\Gamma',\Delta')\supseteq (\{\chi\mid\Box\chi\in\Gamma\},\{\NEG\psi\})$. Now, set $(\Gamma_0,\Delta_0):= (\Gamma\cup\{\Diamond\chi\mid\chi\in \Gamma'\}, \{\Box\theta\mid\theta\in\Delta'\})$. Arguing as in Part 1, we show that $(\Gamma_0,\Delta_0)$ is consistent and thus extendable to a maximal bi-set $(\Gamma_1,\Delta_1)\supseteq(\Gamma_0,\Delta_0)$ for which we then can demonstrate that $(\Gamma,\Delta)\mathrel{\leq_m}(\Gamma_1,\Delta_1)\mathrel{R_m}(\Gamma',\Delta')$. 
	
	Since also $\NEG\psi\in \Delta'$, whence $\NEG\psi\notin\Gamma'$ by the consistency of $(\Gamma',\Delta')$ and $\mathcal{M}_m,(\Gamma',\Delta')\not\models_m^-\psi$ by IH. Given that $(\Gamma,\Delta)\mathrel{\leq_m}(\Gamma_1,\Delta_1)\mathrel{R_m}(\Gamma',\Delta')$, we get $\mathcal{M}_m,(\Gamma',\Delta')\not\models_m^- \Box\psi = \phi$, as desired.
	
	\textit{Case 4}. $\phi = \Diamond\psi$. 
	
	\textit{Part 1}. Let $(\Gamma,\Delta)\in W_m$ be such that $\phi \in \Gamma$. By Lemma \ref{L:mod-consistent}.4, the bi-set $(\{\psi\}\cup\{\xi\mid\Box\xi \in \Gamma\},\{\chi\mid\Diamond\chi\in\Delta\})$ must be consistent and thus extendable to a $(\Gamma_0,\Delta_0)\in W_m$. We show that $(\Gamma,\Delta)\mathrel{R_m}(\Gamma_0,\Delta_0)$. Indeed, we have $\{\xi\mid\Box\xi \in \Gamma\}\subseteq \Gamma_0$ by the choice of $(\Gamma_0,\Delta_0)$, and if $\chi\in\Gamma_0$, then, by consistency, $\chi\notin\Delta_0$, which means that we must have $\Diamond\chi\notin\Delta$; the latter entails, by completeness, that $\Diamond\chi\in\Gamma$. Since $\chi\in\mathcal{MD}$ was chosen arbitrarily, this means that $\{\Diamond\chi\mid\chi\in\Gamma_0\}\subseteq\Gamma$. 
	
	So we have shown that $(\Gamma,\Delta)\mathrel{R_m}(\Gamma_0,\Delta_0)$ and the choice of $(\Gamma_0,\Delta_0)$ also implies that $\psi\in \Gamma_0$, whence, by IH $\mathcal{M}_m,(\Gamma_0,\Delta_0)\models_m^+ \psi$. But the latter means that $\mathcal{M}_m,(\Gamma,\Delta)\models_m^+ \Diamond\psi = \phi$.
	
	In the other direction, if $(\Gamma,\Delta)\in W_m$ is such that $\phi \notin \Gamma$ and $(\Gamma_0,\Delta_0),\in W_m$ is such that $(\Gamma,\Delta)\mathrel{R_m}(\Gamma_0,\Delta_0)$, then $\mathcal{M}_m,(\Gamma_0,\Delta_0)\models_m^+ \psi$ means, by IH for Part 1, that $\psi\in \Gamma_0$, which, by $(\Gamma,\Delta)\mathrel{R_m}(\Gamma_0,\Delta_0)$ implies that $\Diamond\psi = \phi \in \Gamma$, which contradicts our assumption. 
	
	\textit{Part 2}. Let $(\Gamma,\Delta)\in W_m$ be such that $\NEG\phi \in \Gamma$. Then \eqref{E:am6} and Lemma \ref{L:mod-maximal}.1 imply that $\Diamond\NEG\psi\in\Gamma$. By Lemma \ref{L:mod-consistent}.4, the bi-set $(\{\NEG\psi\}\cup\{\xi\mid\Box\xi \in \Gamma\},\{\chi\mid\Diamond\chi\in\Delta\})$ must be consistent and thus extendable to a $(\Gamma_0,\Delta_0)\in W_m$. Arguing as in Part 1, we show that $(\Gamma,\Delta)\mathrel{R_m}(\Gamma_0,\Delta_0)$; of course, the choice of $(\Gamma_0,\Delta_0)$ also implies that $\NEG\psi\in\Gamma_0$, which means, by IH for Part 2, that $\mathcal{M}_m,(\Gamma_0,\Delta_0)\models_m^- \psi$. But then we must also have $\mathcal{M}_m,(\Gamma,\Delta)\models_m^- \Diamond\psi = \phi$.
	
	Conversely, if $(\Gamma,\Delta)\in W_m$ is such that $\NEG\phi \notin \Gamma$, then, by \eqref{E:am6} and Lemma \ref{L:mod-maximal}.1, we must also have $\Diamond\NEG\psi\notin\Gamma$. Now, if $(\Gamma_0,\Delta_0),\in W_m$ is chosen arbitrarily under the condition that $(\Gamma,\Delta)\mathrel{R_m}(\Gamma_0,\Delta_0)$, then $\mathcal{M}_m,(\Gamma_0,\Delta_0)\models_m^- \psi$ means, by IH for Part 2, that $\NEG\psi\in \Gamma_0$, which, by $(\Gamma,\Delta)\mathrel{R_m}(\Gamma_0,\Delta_0)$ implies that $\Diamond\NEG\psi \in \Gamma$. The resulting contradiction shows that we must have $\mathcal{M}_m,(\Gamma_0,\Delta_0)\models_m^- \psi$ for every $(\Gamma_0,\Delta_0),\in W_m$ such that $(\Gamma,\Delta)\mathrel{R_m}(\Gamma_0,\Delta_0)$, which is equivalent to $\mathcal{M}_m,(\Gamma,\Delta)\not\models_m^- \Diamond\psi = \phi$.
\end{proof}
The truth lemma allows us to deduce the (strong) soundness and completeness of $\mathtt{CnK}$ relative to $\mathsf{CnK}$ in the usual way:
\begin{theorem}\label{T:mod-completeness}
	$\mathsf{CnK}= \mathtt{CnK}[MD]$. In particular, for every $\phi\in\mathcal{MD}$, $\vdash_{\mathtt{CnK}}\phi$ iff $\phi\in\mathsf{CnK}$.
\end{theorem}
We omit the standard proof. As a usual corollary, we obtain the compactness of $\mathsf{CnK}$ for bi-sets:
\begin{corollary}\label{C:mod-compactness}
	For $(\Gamma,\Delta)\in \mathcal{P}(\mathcal{MD})\times\mathcal{P}(\mathcal{MD})$, we have $\Gamma\not\models_{\mathsf{CnK}}\Delta$ iff, for all $\Gamma'\Subset\Gamma$, $\Delta'\Subset\Delta$, $\Gamma'\not\models_{\mathsf{CnK}}\Delta'$.
\end{corollary}

\subsection{Some properties of $\mathsf{CnK}$}\label{sub:mod-properties}
In this subsection, we consider several properties of $\mathsf{CnK}$ and assess the degree to which these properties can be viewed as a natural development of their counterparts in $\mathsf{C}$ in the richer environment provided by $MD$ and $\mathbb{FSM}$ as compared to $PL$ and $\mathbb{P}$ respectively.

We start by observing that several propositions about $\mathsf{C}$ easily extend to $\mathsf{CnK}$.
\begin{lemma}\label{L:mod-monotonicity}
	Given a $\phi \in \mathcal{MD}$, a $\star \in \{+, -\}$, an $\mathcal{M} \in \mathbb{FSM}$, and any $w,v \in W$ such that $w \leq v$, $\mathcal{M}, w\models_{m}^\star \phi$ implies $\mathcal{M}, v\models_{m}^\star \phi$. 
\end{lemma}
Again, the proof is by a straightforward induction on the complexity of $\phi \in \mathcal{MD}$.
\begin{proposition}\label{P:Cnk-constructivity}
	$\mathsf{CnK}$ has both DP and CFP.
\end{proposition}
The proof is the same as for Proposition \ref{P:C-constructivity} above, using Lemma \ref{L:mod-monotonicity} in place of Lemma \ref{L:monotonicity}, and setting, for the modal accessibility relation, $R:= R_1 \cup R_2$. The satisfaction of conditions \eqref{Cond:1} and \eqref{Cond:2} then follows immediately from their assumed satifaction by the component models $\mathcal{M}_1$ and $\mathcal{M}_2$, respectively.
\begin{proposition}\label{P:universal-mod-satisfaction}
	Every $\Gamma\subseteq \mathcal{MD}$ is satisfiable in $\mathsf{CnK}$, in other words, $(\Gamma, \emptyset) \notin \mathsf{CnK}$.
\end{proposition}
\begin{proof}[Proof (a sketch)]
	Expand $\mathcal{M} \in \mathbb{P}$ from the proof of Proposition \ref{P:universal-c-satisfaction} to $\mathcal{M}^m$ by setting $R^m := \{(w,w)\}$. It is straightforward to show that $\mathcal{M}^m\in \mathbb{FSM}$, and that we have $\mathcal{M}^m, w\models^\star_m \phi$ for every $\phi \in \mathcal{MD}$ and every $\star \in \{+, -\}$.
\end{proof}
Moreover, in view of Lemma \ref{L:c}, and  Propositions \ref{P:c-negation-non-trivial}, \ref{P:c-connexivity}, and \ref{P:c-strong-connexivity} it is easy to show the following:
\begin{corollary}\label{C-legacy-mod}
	$\mathsf{CnK}$ is both non-trivial and negation-inconsistent. Moreover, $\to$ is hyperconnexive in $\mathsf{CnK}$ and $\Rightarrow$ is connexive in $\mathsf{CnK}$, but neither plainly nor weakly hyperconnexive in $\mathsf{CnK}$. 
\end{corollary}
\begin{proof}[Proof (a sketch)]
By Lemma \ref{L:c} and the proof of Proposition \ref{P:c-negation-non-trivial} we know that \eqref{Contr} is a theorem of $\mathsf{CnK}$ and that, for any $p \in Prop$, $p$ is not provable in $\mathsf{CnK}$. Moreover, Lemma \ref{L:c} implies the full hyperconnexivity of $\to$ and the full connexivity of $\Rightarrow$ also in $\mathsf{CnK}$. The failure of both forms of hyperconnexivity for $\Rightarrow$ in $\mathsf{CnK}$ can be seen by considering the model $\mathcal{M}^m_0\in \mathbb{FSM}$ which expands the model $\mathcal{M}_0\in \mathbb{P}$ from the proof of Proposition \ref{P:c-strong-implication} by setting $R^m_0 := \{(w,w)\}$.
\end{proof}
However, $\mathsf{CnK}$ also develops the connexivity profile of $\mathsf{C}$ by allowing the definitions of further connexive connectives which can be introduced as follows:
\begin{itemize}
	\item $\phi\to_s\psi$ called \textit{strict implication} and abbreviating $\Box(\phi\to\psi)$.
	
	\item $\phi\leftrightarrow_s\psi$ called \textit{strict equivalence} and abbreviating $(\phi\to_s\psi)\wedge(\psi\to_s\phi)$.
	
	\item $\phi\Rightarrow_s\psi$ called \textit{strong strict implication} and abbreviating $\Box(\phi\Rightarrow\psi)$ (or, strongly equivalently over $\mathsf{CnK}$, $(\phi\to_s \psi)\wedge(\NEG\psi\to_s\NEG\phi)$).
	
	\item $\phi\Leftrightarrow_s\psi$ called \textit{strong strict equivalence} and abbreviating $(\phi\Rightarrow_s\psi)\wedge(\psi\Rightarrow_s\phi)$  (or, strongly equivalently over $\mathsf{CnK}$, $\Box(\phi\Leftrightarrow \psi)$).
\end{itemize}
Strict implication is well-known since the early days of modern logic (see, e.g. \cite{lewis}). To the best of our knowledge, strong strict implication has not been looked into as yet, which is easily explainable by the fact that it is both classically and intuitionistically equivalent to the strict implication. However, given the weakness of implication in $\mathsf{C}$ which was already pointed out in Proposition \ref{P:c-strong-implication}, it is no longer reasonable to expect this equivalence for $\mathsf{C}$-based modal logics. Before we assess the resemblance between these new connectives and their counterparts in $\mathsf{C}$, we would like to look into their properties a little bit further:
\begin{lemma}\label{L:strict-implication}
For all $\phi, \psi\in \mathcal{MD}$, $\phi\to_s\psi\in \mathsf{CnK}$ iff $\phi\to\psi\in \mathsf{CnK}$.	
\end{lemma}
\begin{proof}
Right-to-left part follows by \eqref{E:Rnec}. In the other direction, we argue by contraposition. If $\phi\to\psi\notin \mathsf{CnK}$, we can choose a model $(\mathcal{M}, w)\in Pt(\mathbb{FSM})$ such that, for $\mathcal{M},w\models^+_m(\{\phi\}, \{\psi\})$. Now choose a $v \notin W$ and consider the model $\mathcal{M}'$, setting $W':= W \cup \{v\}$, $\leq':= \leq \cup \{(v,v)\}$ and $R':= R\cup \{(v, u)\mid w\leq' u\}$. It is straightforward to check that we have both $\mathcal{M}'\in \mathbb{FSM}$ and $\mathcal{M}',v\not\models^+_m\phi\to_s\psi$.	
\end{proof}
This easy lemma immediately entails the following corollary:
\begin{corollary}\label{C:strict-implication}
For all $\phi, \psi\in \mathcal{MD}$, $\phi\leftrightarrow_s\psi\in \mathsf{CnK}$ (resp. $\phi\Rightarrow_s\psi\in \mathsf{CnK}$, $\phi\Leftrightarrow_s\psi\in \mathsf{CnK}$) iff $\phi\leftrightarrow\psi\in \mathsf{CnK}$ (resp. $\phi\Rightarrow\psi\in \mathsf{CnK}$, $\phi\Leftrightarrow\psi\in \mathsf{CnK}$).
\end{corollary}
Turning now to a direct comparison between the connective pairs $\{\Rightarrow, \Leftrightarrow\}$ and $\{\Rightarrow_s, \Leftrightarrow_s\}$, we observe that a full analogue to Proposition \ref{P:c-strong-implication} can be proven for $\{\Rightarrow_s, \Leftrightarrow_s\}$ in $\mathsf{CnK}$:
\begin{proposition}\label{P:c-strong-strict-implication}
	Let $p, q \in Prop$ and let $\phi, \psi, \theta\in \mathcal{MD}$. Then the following statements hold for all $\gg \in \{\Rightarrow, \Rightarrow_s, \to, \to_s\}$:
	\begin{enumerate}
		\item $(\phi \Rightarrow_s \psi)\gg(\NEG\psi\Rightarrow_s\NEG\phi)\in \mathsf{CnK}$.
		
		\item However, the same cannot be said about $\to_s$, since we have $(p \to_s q)\gg(\NEG q\to_s\NEG p)\notin \mathsf{CnK}$.
		
		\item We have $(\phi \Rightarrow_s \psi)\gg(\phi\to_s\psi)\in \mathsf{CnK}$, but not vice versa, so that $\Rightarrow_s$ is stronger than $\to_s$.
		
		\item If $(\phi\Leftrightarrow\psi) \in \mathsf{CnK}$ then also $\theta[\phi/p]\Leftrightarrow\theta[\psi/p] \in \mathsf{CnK}$.
		
		\item If $(\phi\Leftrightarrow_s\psi) \in \mathsf{CnK}$ then also $\theta[\phi/p]\Leftrightarrow_s\theta[\psi/p] \in \mathsf{CnK}$.
		
		\item However, the same cannot be said about $\leftrightarrow$ and $\leftrightarrow_s$, since we have, e.g. $((p \wedge q) \to p)\leftrightarrow(p \to p), ((p \wedge q) \to p)\leftrightarrow_s(p \to p) \in \mathsf{CnK}$, but $\NEG(((p \wedge q) \to p)\leftrightarrow \NEG(p \to p), \NEG(((p \wedge q) \to p)\leftrightarrow_s \NEG(p \to p) \notin \mathsf{CnK}$.
	\end{enumerate}
\end{proposition}
\begin{proof}[Proof (a sketch)]
	Part 1 follows from the defintion of $\Rightarrow_s$ and \eqref{E:Rnec}, for Part 2, recall the model $\mathcal{M}^m_0$ defined in the proof of Corollary \ref{C-legacy-mod}. As for Part 3, its positive statement is immediate to see, its negative statement follows by considering the model $\mathcal{M}^m_0$ again.
	
	The proof of Part 4 is similar to the proof of Proposition \ref{P:c-strong-implication}.4. As for Part 5, note that if $(\phi\Leftrightarrow_s\psi) \in \mathsf{CnK}$ then, by Corollary \ref{C:strict-implication}, also $(\phi\Leftrightarrow\psi) \in \mathsf{CnK}$, whence $\theta[\phi/p]\Leftrightarrow\theta[\psi/p] \in \mathsf{CnK}$ by Part 4, whence further $\theta[\phi/p]\Leftrightarrow_s\theta[\psi/p] \in \mathsf{CnK}$ by Corollary \ref{C:strict-implication} again.
	
	
	Finally, as for Part 6, its positive statement for $\leftrightarrow$ follows from Proposition \ref{P:c-strong-implication}.5 and Lemma \ref{L:c}; this statement also entails the positive statement for $\leftrightarrow_s$ by \eqref{E:Rnec}. The negative statements follow by consideration of model $\mathcal{M}^m_1\in \mathbb{FSM}$ which expands the model $\mathcal{M}_1\in \mathbb{FSM}$ defined in the proof of Proposition \ref{P:c-strong-implication} by setting $R^m_1:= \{(w,w), (v,v)\}$.
\end{proof}
Next, as far as the connexivity of the two versions of strict implication is concerned, $\mathsf{CnK}$ is in complete harmony with $\mathsf{C}$:
\begin{proposition}\label{P:modal-connexivity}
	$\to_s$ is fully hyperconnexive in $\mathsf{CnK}$ and $\Rightarrow_s$ is fully connexive (but neither plainly nor weakly hyperconnexive) in $\mathsf{CnK}$. 
\end{proposition}
\begin{proof}[Proof (a sketch)]
	Again, we omit the proof for the positive part, which consists in a straightforward but tedious check based on the definitions. As for the negative part of the statement, one can again consider the model $\mathcal{M}^m_0$ defined in the proof of Corollary \ref{C-legacy-mod}, which fails both (CBT$\Rightarrow_s$) and (WCBT$\Rightarrow_s$).	
\end{proof}
A similar harmony arises relative to the negation-inconsistency: $\mathsf{CnK}$ not only inherits the negation-inconsistencies provable in $\mathsf{C}$, but also systematically supplies modal counterparts for them, as the next proposition shows:
\begin{proposition}\label{P:mod-negation-non-trivial}
For every $\phi\wedge\NEG\phi\in\mathsf{CnK}$, we have $\Box\phi\wedge\NEG\Box\phi\in\mathsf{CnK}$. In particular, whenever $(\phi\to\psi)\wedge\NEG(\phi\to\psi)\in\mathsf{CnK}$, we have $(\phi\to_s\psi)\wedge\NEG(\phi\to_s\psi)\in\mathsf{CnK}$. Moreover, if $(\phi\Rightarrow\psi)\wedge\NEG(\phi\Rightarrow\psi)\in\mathsf{CnK}$, then $(\phi\Rightarrow_s\psi)\wedge\NEG(\phi\Rightarrow_s\psi)\in\mathsf{CnK}$.
\end{proposition}
\begin{proof}[Proof (a sketch)]
If  $\phi\wedge\NEG\phi\in\mathsf{CnK}$, then $\Box\phi\wedge\NEG\Box\phi\in\mathsf{CnK}$ follows by \eqref{E:Rnec}, \eqref{E:tm1} and \eqref{E:am5}. The particular cases follow by definitions of $\to_s$ and $\Rightarrow_s$.
\end{proof}
Now Lemma \ref{L:c} and Proposition \ref{P:mod-negation-non-trivial} together imply that:
\begin{corollary}\label{C:mod-negation-non-trivial}
For every $\phi\wedge\NEG\phi\in\mathsf{C}$, we have $\Box\phi\wedge\NEG\Box\phi\in\mathsf{CnK}$. In particular, whenever $(\phi\to\psi)\wedge\NEG(\phi\to\psi)\in\mathsf{C}$, we have $(\phi\to_s\psi)\wedge\NEG(\phi\to_s\psi)\in\mathsf{CnK}$. Moreover, if $(\phi\Rightarrow\psi)\wedge\NEG(\phi\Rightarrow\psi)\in\mathsf{C}$, then $(\phi\Rightarrow_s\psi)\wedge\NEG(\phi\Rightarrow_s\psi)\in\mathsf{CnK}$.	
\end{corollary}
In particular, by Corollary \ref{C:mod-negation-non-trivial}, Lemma \ref{L:c}, \eqref{Contr}, and \eqref{Contr-strong}, we get that the following formula
\begin{align}
	((p_1\wedge \NEG p_1)\gg p_1)&\wedge\NEG((p_1\wedge \NEG p_1)\gg p_1)\label{Contr-m}
\end{align}
is a theorem of $\mathsf{CnK}$ for all $\gg \in \{\to, \Rightarrow, \to_s, \Rightarrow_s\}$.

Another important group of analogies that needs to be observed when discussing the merits of different $\mathsf{C}$-based modal logics stems from the comparison between the modalities and the quantifiers. Although a full appreciation of this sort of analogies is beyond the scope of the present paper, 
it is easy to see that the interaction patterns between the modalities and $\NEG$ in $\mathsf{CnK}$ exactly mirror the interaction between $\NEG$ and the non-standard quantifiers recently proposed as an addition to $\mathsf{QC}$, the first-order version of $\mathsf{C}$ (see \cite{ww} and \cite[Section 6]{olkhovikov0}).

\section{Conditional logic $\mathsf{CnCK}$}\label{S:conditional}
\subsection{$\mathsf{CnCK}$ and its axiomatization}\label{sub:cond-axioms}
We now proceed to the discussion of an extension of $\mathsf{C}$ with conditional operators that we will refer to as $\mathsf{CnCK}$.

We will supply $\mathsf{CnCK}$ with a version of Chellas semantics of conditionals introduced in \cite{chellas}. In particular, the relevant notion of a model can be defined in the following way:
\begin{definition}\label{D:cond-model}
	A bi-valuational Fischer-Servi conditional model is a structure $\mathcal{M} = (W, \leq, R, V^+, V^-)\in \mathbb{C}$ such that, for all $X,Y \subseteq W$ the binary relation $R_{(X,Y)}$ defined by 
	$$
	R_{(X,Y)} := \{(w,v)\mid R(w, (X, Y), v)\}
	$$
	is such that $(W, \leq, R_{(X,Y)}, V^+, V^-)\in \mathbb{FSM}$.
The class of all bi-valuational Fischer-Servi conditional models will be denoted by $\mathbb{FSC}$.	
\end{definition} 
We now define $\mathsf{CnCK}:= \mathfrak{L}(CN, \mathbb{FSC}, (\models^+, \models^-))$, assuming that
$$
(\models^+, \models^-):= Ind(\models_{\mathsf{C}}^+\cup\{\eqref{Cl:boxto+},\eqref{Cl:diamto+}\},\models_{\mathsf{C}}^-\cup\{\eqref{Cl:boxto-},\eqref{Cl:diamto-}\})
$$
where the labels for the new clauses refer to
\begin{align}
	\mathcal{M}, w&\models^+ \psi \boxto \chi \text{ iff } (\forall v \geq w)(\forall u \in W)(R_{\|\psi\|_\mathcal{M}}(v, u) \text{ implies }\mathcal{M}, u\models^+ \chi)\label{Cl:boxto+}\tag{$\boxto+$}\\
	\mathcal{M}, w&\models^+ \psi \boxto \chi \text{ iff } (\forall v \geq w)(\forall u \in W)(R_{\|\psi\|_\mathcal{M}}(v, u) \text{ implies }\mathcal{M}, u\models^- \chi)\label{Cl:boxto-}\tag{$\boxto-$}\\
	\mathcal{M}, w&\models^+ \psi \diamondto \chi \text{ iff } (\exists u \in W)(R_{\|\psi\|_\mathcal{M}}(w, u)\text{ and }\mathcal{M}, u\models^+ \chi)\label{Cl:diamto+}\tag{$\diamondto+$}\\	
	\mathcal{M}, w&\models^- \psi \diamondto \chi \text{ iff } (\exists u \in W)(R_{\|\psi\|_\mathcal{M}}(w, u)\text{ and }\mathcal{M}, u\models^- \chi)\label{Cl:diamto-}\tag{$\diamondto-$}
\end{align}
Recall that, for a $\phi \in \mathcal{CN}$, we have defined:
$$
\|\phi\|_\mathcal{M}: = (\lvert\phi\rvert^+_\mathcal{M}, \lvert\phi\vert^-_\mathcal{M}) = (\{w \in W\mid \mathcal{M}, w\models^+\phi\}, \{w \in W\mid \mathcal{M}, w\models^-\phi\}).
$$
Throughout this section, we will write $\Gamma\models\Delta$ meaning $\Gamma\models_{\mathsf{CnCK}}\Delta$. 

In this subsection, we obtain a sound and (strongly) complete axiomatization of $\mathsf{CnCK}$. We consider the Hilbert-style axiomatic system $\mathtt{CnCK}$, for which we set $\mathtt{CnCK}:= \mathtt{C}+(\eqref{E:a1}-\eqref{E:a7};\eqref{E:RAbox},\eqref{E:RCbox},\eqref{E:RAdiam},\eqref{E:RCdiam})$ assuming that
\begin{align}
	((\phi \boxto \psi)\wedge(\phi \boxto \chi))&\leftrightarrow(\phi \boxto (\psi \wedge \chi))\label{E:a1}\tag{$\gamma_1$}\\
	((\phi\diamondto\psi)\wedge (\phi \boxto \chi))&\to(\phi \diamondto(\psi\wedge\chi))\label{E:a2}\tag{$\gamma_2$}\\
	((\phi \diamondto \psi)\vee(\phi \diamondto \chi))&\leftrightarrow(\phi \diamondto (\psi \vee \chi))\label{E:a3}\tag{$\gamma_3$}\\
	((\phi \diamondto \psi)\to(\phi \boxto \chi))&\to(\phi \boxto (\psi \to \chi))\label{E:a4}\tag{$\gamma_4$}\\
	\phi\boxto (\psi&\to\psi)\label{E:a5}\tag{$\gamma_5$}\\
	\NEG(\phi\boxto\psi)&\leftrightarrow(\phi\boxto\NEG\psi)\label{E:a6}\tag{$\gamma_6$}\\
	\NEG(\phi\diamondto\psi)&\leftrightarrow(\phi\diamondto\NEG\psi)\label{E:a7}\tag{$\gamma_7$}\\
	\text{From }\phi\Leftrightarrow\psi &\text{ infer } (\phi\boxto\chi)\Leftrightarrow(\psi\boxto\chi)\label{E:RAbox}\tag{RA$\boxto$}\\
	\text{From }\phi\leftrightarrow\psi &\text{ infer } (\chi\boxto\phi)\leftrightarrow(\chi\boxto\psi)\label{E:RCbox}\tag{RC$\boxto$}\\
		\text{From }\phi\Leftrightarrow\psi &\text{ infer } (\phi\diamondto\chi)\Leftrightarrow(\psi\diamondto\chi)\label{E:RAdiam}\tag{RA$\diamondto$}\\
	\text{From }\phi\leftrightarrow\psi &\text{ infer } (\chi\diamondto\phi)\leftrightarrow(\chi\diamondto\psi)\label{E:RCdiam}\tag{RC$\diamondto$}	
\end{align}
Within this subsection, in addition to assuming all the notions introduced in Section \ref{S:preliminaries}, we will also write $\vdash$ and $\vDdash$ to mean $\vdash_{\mathtt{CnCK}}$ and $\vDdash_{\mathtt{CnCK}}$, respectively, to avoid the clutter.

The soundness of $\mathtt{CnCK}$ relative to $\mathsf{CnCK}$ can be established by the usual method of checking the soundness of axioms and inference rules:
\begin{lemma}\label{L:cond-soundness}
	For every $\phi\in\mathcal{CN}$, if $\vdash\phi$, then $\phi\in\mathsf{CnCK}$.
\end{lemma}
Next, just like in Section \ref{sub:modal-axioms}, we prepare our completeness proof by looking into the relations between $\mathtt{CnCK}$ and $\mathsf{C}$: 
\begin{lemma}\label{L:cond-c}
	The following statements hold:
	\begin{enumerate}
		\item If $\Gamma, \Delta \subseteq\mathcal{L}$ are such that $\Gamma\models_{\mathsf{C}}\Delta$, and  $\Gamma', \Delta' \subseteq \mathcal{CN}$ are obtained from $\Gamma, \Delta$ by a simultaneous substitution of $\mathcal{CN}$-formulas for atoms, then $\Gamma'\vdash\Delta'$. Moreover, Deduction Theorem holds for $\mathtt{CnCK}$ in that for all $\Gamma \cup \{\phi,\psi\}\subseteq \mathcal{CN}$ we have $\Gamma \vdash \phi\to\psi$ iff $\Gamma, \phi\vdash \psi$.
		
		\item If $\phi \in \mathcal{PL}$, then $\vdash \phi$ iff $\phi\in \mathsf{C}$.
	\end{enumerate}
\end{lemma}
\begin{proof}[Proof (a sketch)]
	Part 1 is trivial. As for Part 2, its right-to-left direction follows from $\mathtt{C}\subseteq \mathtt{CnCK}$. As for the other direction, assume that $\phi \in \mathcal{PL}$ is such that $\phi\notin \mathsf{C}$, and let $\mathcal{M}\in \mathbb{P}$ and $w \in W$ be such that $\mathcal{M}, w\not\models^+_{\mathsf{C}}\phi$. Then let $\mathcal{M}'$ expand $\mathcal{M}$ with $R' := \emptyset$. It is clear that $\mathcal{M}'\in \mathbb{FSC}$ and that for every $v \in W$, $\psi\in \mathcal{PL}$ and $\star\in \{+, -\}$ we have $\mathcal{M}, v\models_{\mathsf{C}}^\star\psi$ iff $\mathcal{M}', v\models^\star\psi$; thus, in particular,  $\mathcal{M}, w\not\models^+\phi$, and, by Lemma \ref{L:cond-soundness}, $\not\vdash \phi$. 
\end{proof}
Next, we establish several theorems and derived rules of $\mathtt{CnCK}$:
\begin{lemma}\label{L:cond-theorems}
	Let $\phi, \psi, \chi \in \mathcal{CN}$. The following theorems and derived rules can be deduced in $\mathtt{CnCK}$:
	\begin{align}
		\phi \Leftrightarrow \psi &\vDdash (\chi \boxto \phi) \Leftrightarrow (\chi \boxto \psi)\label{E:RBbox}\tag{RB$\boxto$}\\
		\phi \Leftrightarrow \psi &\vDdash (\chi \diamondto \phi) \Leftrightarrow (\chi \diamondto \psi)\label{E:RBdiam}\tag{RB$\diamondto$}\\
		\phi &\vDdash (\psi\boxto \phi)\label{E:RNec}\tag{Nec}\\
		(\phi \to \psi) &\vDdash (\chi\boxto\phi)\to(\chi\boxto\psi)\label{E:Rmbox}\tag{RM$\boxto$}\\
		(\phi \to \psi) &\vDdash (\chi\diamondto\phi)\to(\chi\diamondto\psi)\label{E:Rmdiam}\tag{RM$\diamondto$}\\
		(\phi\boxto(\psi \to \chi))&\to((\phi\boxto\psi)\to(\phi\boxto\chi))\label{E:T1}\tag{Th1}\\
		(\phi\boxto(\psi \to \chi))&\to((\phi\diamondto\psi)\to(\phi\diamondto\chi))\label{E:T2}\tag{Th2}\\
		(\phi\diamondto (\psi\to \chi))&\to((\phi\boxto\psi)\to(\phi\diamondto\psi))\label{E:T3}\tag{Th3}\\
		\NEG(\phi\boxto\psi)&\Leftrightarrow(\phi\boxto\NEG\psi)\label{E:T4}\tag{Th4}\\
		\NEG(\phi\diamondto\psi)&\Leftrightarrow(\phi\diamondto\NEG\psi)\label{E:T5}\tag{Th5}
	\end{align}
\end{lemma}
A sketch of a proof is relegated to Appendix \ref{A:1}.

Throughout this section we will write (in)consistency (resp. completeness, maximality) meaning $\mathtt{CnCK}$-(in)consistency (resp. $CN$-completeness, $\mathtt{CnCK}[CN]$-maximality). The path to the completeness theorem for $\mathtt{CnCK}$ that is taken in the remaining part of this subsection is now basically the same as the path we went w.r.t. $\mathtt{CnK}$ in Section \ref{sub:modal-axioms} above. Both the formulations of the lemmas and their proofs are very much in analogy with one another. We will therefore mainly confine ourselves to reformulating the lemmas for the case of $\mathtt{CnCK}$ whereas their proofs will be either omitted altogether or replaced with very brief sketches.

In this way, we successively establish that:
\begin{lemma}\label{L:cond-alt-consistency}
	A bi-set $(\Gamma, \Delta)\in \mathcal{P}(\mathcal{CN})\times\mathcal{P}(\mathcal{CN})$ is inconsistent iff, for some $m,n\in \omega$ some $\phi_1,\ldots,\phi_n\in \Gamma$ and some $\psi_1,\ldots,\psi_m\in \Delta$ we have: $
	\bigwedge^n_{i = 1}\phi_i\vdash\bigvee^m_{j = 1}\psi_j$, or, equivalently, $\vdash
	\bigwedge^n_{i = 1}\phi_i\to\bigvee^m_{j = 1}\psi_j$.
\end{lemma}
\begin{lemma}\label{L:cond-consistent}
	Let	$(\Gamma, \Delta)\in \mathcal{P}(\mathcal{CN})\times\mathcal{P}(\mathcal{CN})$ be consistent. Then the following statements hold:
	\begin{enumerate}
		\item For every $\phi \in \mathcal{CN}$, either $(\Gamma \cup \{\phi\}, \Delta)$ or $(\Gamma, \Delta\cup \{\phi\})$ is consistent.
		
		\item For every $\phi \to \psi \in \Delta$, $(\Gamma \cup \{\phi\}, \{\psi\})$ is consistent.
		
		\item For every $\phi\boxto\psi \in \Delta$, the bi-set $(\{\chi\mid\phi\boxto\chi \in \Gamma\},\{\psi\})$ is consistent.
		
		\item For every $\phi\diamondto\psi \in \Gamma$, the bi-set $(\{\psi\}\cup\{\chi\mid\phi\boxto\chi \in \Gamma\},\{\theta\mid\phi\diamondto\theta\in\Delta\})$ is consistent.
	\end{enumerate}
\end{lemma}
\begin{lemma}\label{L:cond-maximal}
	Let	$(\Gamma, \Delta), (\Gamma_0,\Delta_0), (\Gamma_1,\Delta_1) \in \mathcal{P}(\mathcal{CN})\times\mathcal{P}(\mathcal{CN})$ be maximal bi-sets such that $\{\psi\mid \phi\boxto\psi\in \Gamma_0\}\subseteq \Gamma_1$, and $\{\NEG(\phi\boxto\psi)\mid \NEG\psi\in \Gamma_1\}\subseteq \Gamma_0$, and let $\phi,\psi\in\mathcal{CN}$. Then the following statements are true:
	\begin{enumerate}
		\item If $\Gamma\vdash\phi$, then $\phi\in \Gamma$.
		
		\item $\phi\wedge\psi\in\Gamma$ iff $\phi, \psi\in \Gamma$.
		
		\item $\phi\vee\psi \in \Gamma$ iff $\phi \in \Gamma$ or $\psi\in\Gamma$.
		
		\item If $\phi\to\psi, \phi \in \Gamma$, then $\psi \in \Gamma$.
		
		\item If $\Gamma_0 \subseteq \Gamma$, then $(\Gamma_1 \cup \{\psi\mid\phi\boxto\psi\in \Gamma\}, \{\psi\mid\phi\diamondto\psi\in \Delta\})$ is consistent.
		
		\item If $\Gamma_1 \subseteq \Gamma$, then $(\Gamma_0 \cup \{\phi\diamondto\psi\mid\psi\in \Gamma\}, \{\phi\boxto\psi\mid\psi\in \Delta\})$ is consistent.
	\end{enumerate}
\end{lemma}
\begin{lemma}\label{L:cond-lindenbaum}
	Let $(\Gamma, \Delta)\in \mathcal{P}(\mathcal{CN})\times\mathcal{P}(\mathcal{CN})$ be consistent. Then there exists a maximal $(\Xi, \Theta)\in \mathcal{P}(\mathcal{CN})\times\mathcal{P}(\mathcal{CN})$ such that $\Gamma \subseteq \Xi$ and $\Delta \subseteq \Theta$. 
\end{lemma}
Next, we define the canonical model $\mathcal{M}_c$ for $\mathtt{CnCK}$. Its definition is based on the same idea as the definition of the canonical model $\mathcal{M}_m$ in the modal case, but also needs to account for the increased complexity of $\mathcal{CN}$ as compared to $\mathcal{MD}$:
\begin{definition}\label{D:canonical-model}
	The structure $\mathcal{M}_c$ is the tuple $(W_c, \leq_c, R_c, V^+_c, V^-_c)$ such that:
	\begin{itemize}
		\item $W_c:=\{(\Gamma, \Delta)\in \mathcal{P}(\mathcal{CN})\times\mathcal{P}(\mathcal{CN})\mid (\Gamma, \Delta)\text{ is maximal }\}$.
		
		\item $(\Gamma_0,\Delta_0)\leq_c(\Gamma_1,\Delta_1)$ iff $\Gamma_0\subseteq\Gamma_1$ for all $(\Gamma_0,\Delta_0),(\Gamma_1,\Delta_1)\in W_c$.
		
		\item For all $(\Gamma_0,\Delta_0),(\Gamma_1,\Delta_1)\in W_c$ and $X, Y \subseteq W_c$, we have $((\Gamma_0,\Delta_0),(X,Y),(\Gamma_1,\Delta_1)) \in R_c$ iff there exists a $\phi\in\mathcal{CN}$, such that all of the following holds:
		\begin{itemize}
			\item $X = \{(\Gamma,\Delta)\in W_c\mid\phi\in\Gamma\}$.
			
			\item $Y = \{(\Gamma,\Delta)\in W_c\mid\NEG\phi\in\Gamma\}$
			
			\item $\{\psi\mid\phi\boxto\psi\in \Gamma_0\}\subseteq \Gamma_1$.
			
			\item $\{\phi\diamondto\psi\mid\psi\in \Gamma_1\}\subseteq \Gamma_0$.
		\end{itemize}
		
		\item $V^+_c(p):=\{(\Gamma,\Delta)\in W_c\mid p\in\Gamma\}$ for every $p \in Prop$.
		
		\item $V^-_c(p):=\{(\Gamma,\Delta)\in W_c\mid \NEG p\in\Gamma\}$ for every $p \in Prop$.	
	\end{itemize}
\end{definition}
Another complication is that, unlike in the modal case, we need to make sure that the definition of $R_c$ does not depend on the choice of the representative formula $\phi\in\mathcal{CN}$. The following lemma provides the necessary stepping stone:
\begin{lemma}\label{L:cond-representatives}
	Let $\phi,\psi \in \mathcal{CN}$ be such that both $\{(\Gamma,\Delta)\in W_c\mid\phi\in\Gamma\} = \{(\Gamma,\Delta)\in W_c\mid\psi\in\Gamma\}$ and $\{(\Gamma,\Delta)\in W_c\mid\NEG\phi\in\Gamma\} = \{(\Gamma,\Delta)\in W_c\mid\NEG\psi\in\Gamma\}$. Then, for every $(\Gamma',\Delta')\in W_c$ and every $\chi\in\mathcal{CN}$ we have:
	\begin{enumerate}
		\item $\phi\boxto\chi\in\Gamma'$ iff $\psi\boxto\chi\in\Gamma'$.
		
		\item $\phi\diamondto\chi\in\Gamma'$ iff $\psi\diamondto\chi\in\Gamma'$.
	\end{enumerate}
\end{lemma}
\begin{proof}
	Assume the hypothesis of the Lemma. We will show that in this case we must have $\vdash\phi\Leftrightarrow\psi$. Suppose not, and assume, for instance, that $\not\vdash\phi\to\psi$. Then $(\{\phi\},\{\psi\})$ must be consistent and thus extendable to some maximal $(\Gamma_0,\Delta_0)\supseteq (\{\phi\},\{\psi\})$. But then clearly $
	(\Gamma_0,\Delta_0)\in \{(\Gamma,\Delta)\in W_c\mid\phi\in\Gamma\} \setminus \{(\Gamma,\Delta)\in W_c\mid\psi\in\Gamma\}$, 
	in contradiction with our initial assumptions. On the other hand, if we have, e.g. $\not\vdash\NEG\phi\to\NEG\psi$, then $(\{\NEG\phi\},\{\NEG\psi\})$ must be consistent and thus extendable to some maximal $(\Gamma_1,\Delta_1)\supseteq (\{\NEG\phi\},\{\NEG\psi\})$. But then clearly
	$$
	(\Gamma_1,\Delta_1)\in \{(\Gamma,\Delta)\in W_c\mid\NEG\phi\in\Gamma\} \setminus \{(\Gamma,\Delta)\in W_c\mid\NEG\psi\in\Gamma\},
	$$ 
	again contradicting our initial assumptions. The reasoning in other cases is parallel to the examples considered above.
	
	Thus we see that we must have $\vdash\phi\Leftrightarrow\psi$. An application of \eqref{E:RAbox} then yields that also $\vdash(\phi\boxto\chi)\Leftrightarrow(\psi\boxto\chi)$
	for every $\chi\in\mathcal{CN}$, and an application of \eqref{E:RAdiam} similarly yields $\vdash(\phi\diamondto\chi)\Leftrightarrow(\psi\diamondto\chi)$, whence our Lemma clearly follows.  
\end{proof}
We have to make sure that we have indeed just defined a model:
\begin{lemma}\label{L:cond.canonical-model}
	$\mathcal{M}_c \in \mathbb{FSC}$.	
\end{lemma}
\begin{proof}[Proof (a sketch)]
	Again, the only non-trivial part is to show that conditions \eqref{Cond:1} and \eqref{Cond:2} are satisfied by $\leq_c$ and $(R_c)_{(X,Y)}$ for all $X, Y \subseteq W_c$. As for \eqref{Cond:1}, assume that $(\Gamma,\Delta)$,  $(\Gamma_0,\Delta_0)$, and  $(\Gamma_1,\Delta_1)$ are maximal, and that $X, Y\subseteq W_c$ are such that we have $(\Gamma,\Delta) \mathrel{\geq_c}(\Gamma_0,\Delta_0)\mathrel{(R_c)_{(X,Y)}}(\Gamma_1,\Delta_1)$. Then, in particular, $\Gamma\supseteq\Gamma_0$. Moreover, we can choose a $\phi \in\mathcal{CN}$ such that we have:
	\begin{align}
		&X = \{(\Xi,\Theta)\in W_c\mid\phi\in\Xi\}\label{E:mod1}\\
		&Y = \{(\Xi,\Theta)\in W_c\mid\NEG\phi\in\Xi\}\label{E:mod1a}\\
		&\{\psi\mid\phi\boxto\psi\in \Gamma_0\}\subseteq \Gamma_1\label{E:mod2}\\
		&\{\phi\diamondto\psi\mid\psi\in \Gamma_1\}\subseteq \Gamma_0\label{E:mod3}
	\end{align}
	By Lemma \ref{L:cond-maximal}.5, the bi-set $(\Gamma_1 \cup \{\psi\mid\phi\boxto\psi\in \Gamma\}, \{\psi\mid\phi\diamondto\psi\in \Delta\})$ must then be consistent, so that, by Lemma \ref{L:cond-lindenbaum}, this bi-set must be extendable to some maximal bi-set $(\Gamma',\Delta')$. We will have then $\Gamma'\supseteq \Gamma_1$ whence clearly $(\Gamma',\Delta')\mathrel{\geq_c}(\Gamma_1,\Delta_1)$.
	
	Next, we get $\{\psi\mid\phi\boxto\psi\in \Gamma\}\subseteq \Gamma'$ trivially by the choice of $(\Gamma',\Delta')$. Moreover, if $\psi \in \Gamma'$, then $\psi\notin\Delta'$ by the consistency of $(\Gamma',\Delta')$. But this means that we cannot have $\phi\diamondto\psi\in\Delta$, so $\phi\diamondto\psi\in \Gamma$ by the completeness of $(\Gamma,\Delta)$. Thus we have shown that also $\{\phi\diamondto\psi\mid\psi\in \Gamma'\}\subseteq \Gamma$. Summing this up with \eqref{E:mod1} and \eqref{E:mod1a}, we obtain that $(\Gamma,\Delta)\mathrel{(R_c)_{(X,Y)}}(\Gamma',\Delta')$. Thus we get that $
	(\Gamma,\Delta)\mathrel{(R_c)_{(X,Y)}}(\Gamma',\Delta')\mathrel{\geq_c}(\Gamma_1,\Delta_1)$, and condition \eqref{Cond:1} is shown to be satisfied.
	
	The satisfaction of \eqref{Cond:2} is shown similarly.
\end{proof}
The truth lemma for this model then looks as follows:
\begin{lemma}\label{L:truth}
	For every $\phi\in\mathcal{CN}$ and for every $(\Gamma,\Delta)\in W_c$, the following statements hold:
	\begin{enumerate}
		\item $
		\mathcal{M}_c,(\Gamma,\Delta)\models^+\phi$ iff $\phi \in \Gamma$.
		
		\item $
		\mathcal{M}_c,(\Gamma,\Delta)\models^-\phi$ iff $\NEG\phi \in \Gamma$.
	\end{enumerate} 
\end{lemma}
\begin{proof}[Proof (a sketch)] We only consider the case when  $\phi = \psi \boxto \chi$.
	
\textit{Part 1}. Let $(\Gamma,\Delta)\in W_c$ be such that $\phi = \psi \boxto \chi \in \Gamma$, and let $(\Gamma_0,\Delta_0), (\Gamma_1,\Delta_1)\in W_c$ be such that $(\Gamma,\Delta)\mathrel{\leq_c}(\Gamma_0,\Delta_0)\mathrel{(R_c)_{\|\psi\|_{\mathcal{M}_c}}}(\Gamma_1,\Delta_1)$. Then $\Gamma \subseteq \Gamma_0$, so that $\psi \boxto \chi \in \Gamma_0$. On the other hand, there must exist a $\theta\in \mathcal{CN}$ such that all of the following holds:
	\begin{align}
		&\|\psi\|_{\mathcal{M}_c} =(\{(\Xi,\Theta)\in W_c\mid\theta\in\Xi\}, \{(\Xi,\Theta)\in W_c\mid\NEG\theta\in\Xi\})\label{E:choice1}\\
		&\{\xi\mid\theta\boxto\xi\in \Gamma_0\}\subseteq \Gamma_1\label{E:choice2}\\
		&\{\theta\diamondto\xi\mid\xi\in \Gamma_1\}\subseteq \Gamma_0\label{E:choice3}
	\end{align}
	By IH, we know that also $\|\psi\|_{\mathcal{M}_c} = (\{(\Xi,\Theta)\in W_c\mid\psi\in\Xi\}, \{(\Xi,\Theta)\in W_c\mid\NEG\psi\in\Xi\})$. We thus get that:
	\begin{align}
		(\{(\Xi,\Theta)\in W_c\mid\psi\in\Xi\}, &\{(\Xi,\Theta)\in W_c\mid\NEG\psi\in\Xi\}) =\notag\\
		&= (\{(\Xi,\Theta)\in W_c\mid\theta\in\Xi\}, \{(\Xi,\Theta)\in W_c\mid\NEG\theta\in\Xi\})\label{E:choice4}
	\end{align}
	Since $\psi \boxto \chi \in \Gamma_0$, we know, by Lemma \ref{L:cond-representatives} and \eqref{E:choice4}, that also $\theta\boxto\chi \in \Gamma_0$. It follows by \eqref{E:choice2}, that $\chi\in \Gamma_1$. Next, IH implies that  $\mathcal{M}_c,(\Gamma_1,\Delta_1)\models^+\chi$. Since the choice of  $(\Gamma_0,\Delta_0), (\Gamma_1,\Delta_1)\in W_c$ under the condition that $(\Gamma,\Delta)\mathrel{\leq_c}(\Gamma_0,\Delta_0)\mathrel{(R_c)_{\|\psi\|_{\mathcal{M}_c}}}(\Gamma_1,\Delta_1)$ was made arbitrarily, it follows that we must have $\mathcal{M}_c,(\Gamma,\Delta)\models^+ \psi \boxto \chi = \phi$.
	
	In the other direction, let $(\Gamma,\Delta)\in W_c$ be such that $\phi \notin \Gamma$. Therefore, $\psi \boxto \chi \in \Delta$ by completeness of $(\Gamma, \Delta)$, and $(\{\theta\mid\psi\boxto\theta\in\Gamma\},\{\chi\})$ must be consistent by Lemma \ref{L:cond-consistent}.3. By Lemma \ref{L:cond-lindenbaum}, we can extend it to a maximal $(\Gamma',\Delta')\supseteq (\{\theta\mid\psi\boxto\theta\in\Gamma\},\{\chi\})$. Now, set $(\Gamma_0,\Delta_0):= (\Gamma\cup\{\psi\diamondto\theta\mid\theta\in \Gamma'\}, \{\psi\boxto\xi\mid\xi\in\Delta'\})$. We claim that $(\Gamma_0,\Delta_0)$ is consistent. Otherwise, we can choose $\gamma_1,\ldots,\gamma_n\in\Gamma$, $\tau_1,\ldots,\tau_m\in\Gamma'$ and $\xi_1,\ldots,\xi_k\in\Delta'$ such that $\bigwedge^n_{i = 1}\gamma_i,\bigwedge^m_{j = 1}\psi\diamondto\tau_j\vdash\bigvee^k_{r = 1}(\psi\boxto\xi_r)$. But then, for $\gamma:= \bigwedge^n_{i = 1}\gamma_i$, $\tau:= \bigwedge^m_{j = 1}\tau_j$, and $\xi:= \bigvee^k_{r = 1}\xi_r$, we have:
	\begin{align}
		&\gamma\vdash \bigwedge^m_{j = 1}(\psi\diamondto\tau_j)\to\bigvee^k_{r = 1}(\psi\boxto\xi_r)\label{E:can5} &&\text{Lemma \ref{L:cond-c}.1}\\
		&\vdash(\psi\diamondto\tau)\to\bigwedge^m_{j = 1}(\psi\diamondto\tau_j)\label{E:can7} &&\text{$\mathsf{C}$, \eqref{E:Rmdiam}}\\
		&\vdash\bigvee^k_{r = 1}(\psi\boxto\xi_r)\to(\psi\boxto\xi)\label{E:can9} &&\text{$\mathsf{C}$, \eqref{E:Rmbox}}\\
		&\gamma\vdash(\psi\diamondto\tau)\to(\psi\boxto\xi)\label{E:can10} &&\text{\eqref{E:can5}, \eqref{E:can7}, \eqref{E:can9}}\\
		&(\psi\boxto(\tau\to\xi))\in \Gamma\label{E:can12} &&\text{\eqref{E:can10}, \eqref{E:a4}, Lemma \ref{L:cond-maximal}.1}
	\end{align}
	By \eqref{E:can12} and the choice of $(\Gamma',\Delta')$, we know that $(\tau\to\xi)\in \Gamma'$, therefore, $(\Gamma',\Delta')$ must be inconsistent, which contradicts its choice and shows that $(\Gamma_0,\Delta_0)$ must have been consistent. Therefore, $(\Gamma_0,\Delta_0)$ is extendable to a maximal bi-set $(\Gamma_1,\Delta_1)\supseteq(\Gamma_0,\Delta_0)$. 
	
	We now claim that we have both $(\Gamma,\Delta)\mathrel{\leq_c}(\Gamma_1,\Delta_1)$ and 
	$$
	((\Gamma_1,\Delta_1),(\{(\Xi,\Theta)\in W_c\mid\psi\in\Xi\}, \{(\Xi,\Theta)\in W_c\mid\NEG\psi\in\Xi\}),(\Gamma',\Delta'))\in R_c.
	$$
	The first part is trivial since we have $\Gamma_1\supseteq\Gamma_0 \supseteq\Gamma$ by the choice of $(\Gamma_1,\Delta_1)$ and $(\Gamma_0,\Delta_0)$. As for the second part, note that (a) for every $\theta\in\mathcal{CN}$, if $\psi\boxto\theta\in \Gamma_1$ and $\theta \notin \Gamma'$, then, by the completeness of $(\Gamma',\Delta')$, we must have $\theta \in \Delta'$. But then $\psi\boxto\theta\in\Delta_0\subseteq\Delta_1$, which contradicts the consistency  of $(\Gamma_1,\Delta_1)$. The obtained contradiction shows that  $\{\theta\mid\psi\boxto\theta\in \Gamma_1\}\subseteq \Gamma'$. Next, (b) we trivially get that $\{\psi\diamondto\theta\mid\theta\in \Gamma'\}\subseteq\Gamma_0 \subseteq \Gamma_1$. Summing up (a) and (b), we get that $((\Gamma_1,\Delta_1),(\{(\Xi,\Theta)\in W_c\mid\psi\in\Xi\},\{(\Xi,\Theta)\in W_c\mid\NEG\psi\in\Xi\}),(\Gamma',\Delta'))\in R_c$.
	
	It remains to notice that, by IH, we must have 
	$$
	\|\psi\|_{\mathcal{M}_c} = (\{(\Xi,\Theta)\in W_c\mid\psi\in\Xi\},\{(\Xi,\Theta)\in W_c\mid\NEG\psi\in\Xi\}),
	$$
	so that we have shown, in effect that $((\Gamma_1,\Delta_1),\|\psi\|_{\mathcal{M}_c},(\Gamma',\Delta'))\in R_c$.
	
	Observe, next, that also $\chi\in \Delta'$, whence $\chi\notin\Gamma'$ by the consistency of $(\Gamma',\Delta')$. Therefore, by IH, $\mathcal{M}_c,(\Gamma',\Delta')\not\models^+ \chi$. Together with the fact that $(\Gamma,\Delta)\mathrel{\leq_c}(\Gamma_1,\Delta_1)$ and $((\Gamma_1,\Delta_1),\|\psi\|_{\mathcal{M}_c},(\Gamma',\Delta'))\in R_c$, this finally implies that $\mathcal{M}_c,(\Gamma',\Delta')\not\models^+ \psi\boxto\chi = \phi$.
	
	\textit{Part 2}. Let $(\Gamma,\Delta)\in W_c$ be such that $\NEG\phi = \NEG(\psi \boxto \chi) \in \Gamma$, and let $(\Gamma_0,\Delta_0), (\Gamma_1,\Delta_1)\in W_c$ be such that $(\Gamma,\Delta)\mathrel{\leq_c}(\Gamma_0,\Delta_0)\mathrel{(R_c)_{\|\psi\|_{\mathcal{M}_c}}}(\Gamma_1,\Delta_1)$. Then $\Gamma \subseteq \Gamma_0$, so that $\NEG(\psi \boxto \chi) \in \Gamma_0$. On the other hand, there must exist a $\theta\in \mathcal{CN}$ such that \eqref{E:choice1}--\eqref{E:choice3} hold. By IH, we know that also $\|\psi\|_{\mathcal{M}_c} = (\{(\Xi,\Theta)\in W_c\mid\psi\in\Xi\}, \{(\Xi,\Theta)\in W_c\mid\NEG\psi\in\Xi\})$. We thus get that \eqref{E:choice4} also holds.
	
	Next, Lemma \ref{L:cond-maximal}.1 and \eqref{E:a6} together imply that
	$\psi \boxto \NEG\chi \in \Gamma_0$, therefore, we know, by Lemma \ref{L:cond-representatives} and \eqref{E:choice4}, that also $\theta\boxto\NEG\chi \in \Gamma_0$. It follows by \eqref{E:choice2}, that $\NEG\chi\in \Gamma_1$. Now IH implies that  $\mathcal{M}_c,(\Gamma_1,\Delta_1)\models^-\chi$. Since the choice of  $(\Gamma_0,\Delta_0), (\Gamma_1,\Delta_1)\in W_c$ under the condition that $(\Gamma,\Delta)\mathrel{\leq_c}(\Gamma_0,\Delta_0)\mathrel{(R_c)_{\|\psi\|_{\mathcal{M}_c}}}(\Gamma_1,\Delta_1)$ was made arbitrarily, it follows that we must have $\mathcal{M}_c,(\Gamma,\Delta)\models^- \psi \boxto \chi = \phi$.
	
	In the other direction, let $(\Gamma,\Delta)\in W_c$ be such that $\NEG\phi \notin \Gamma$. Therefore, $\NEG(\psi \boxto \chi) \in \Delta$ by completeness of $(\Gamma, \Delta)$. By \eqref{E:a6} and Lemma \ref{L:cond-maximal}.1 we must have $\psi \boxto \NEG\chi \in \Delta$, so that $(\{\theta\mid\psi\boxto\theta\in\Gamma\},\{\NEG\chi\})$ must be consistent by Lemma \ref{L:cond-consistent}.3. By Lemma \ref{L:cond-lindenbaum}, we can extend it to a maximal $(\Gamma',\Delta')\supseteq (\{\theta\mid\psi\boxto\theta\in\Gamma\},\{\NEG\chi\})$. Now, set $(\Gamma_0,\Delta_0):= (\Gamma\cup\{\psi\diamondto\theta\mid\theta\in \Gamma'\}, \{\psi\boxto\xi\mid\xi\in\Delta'\})$. Arguing as in Part 1, we show that $(\Gamma_0,\Delta_0)$ is consistent and thus extendable to a maximal bi-set $(\Gamma_1,\Delta_1)\supseteq(\Gamma_0,\Delta_0)$, for which we can show that both $(\Gamma,\Delta)\mathrel{\leq_c}(\Gamma_1,\Delta_1)$ and 
	$((\Gamma_1,\Delta_1),\|\psi\|_{\mathcal{M}_c},(\Gamma',\Delta'))\in R_c$.
	
	Recall that $\NEG\chi\in \Delta'$, whence $\NEG\chi\notin\Gamma'$ by the consistency of $(\Gamma',\Delta')$. Next, $\mathcal{M}_c,(\Gamma',\Delta')\not\models^- \chi$ by IH. By the choice of $(\Gamma',\Delta')$, this finally implies that $\mathcal{M}_c,(\Gamma',\Delta')\not\models^- \psi\boxto\chi = \phi$.
\end{proof}
Again the deduction of the (strong) soundness and completeness of $\mathtt{CnCK}$ relative to  (and of the compactness of $\mathsf{CnCK}$ itself) from the above truth lemma is completely standard:
\begin{theorem}\label{T:cond-completeness}
	$\mathsf{CnCK}= \mathtt{CnCK}[CN]$. In particular, for every $\phi\in\mathcal{CN}$, $\vdash\phi$ iff $\phi\in\mathsf{CnCK}$.
\end{theorem}
\begin{corollary}\label{C:cond-compactness}
	For $(\Gamma,\Delta)\in \mathcal{P}(\mathcal{CN})\times\mathcal{P}(\mathcal{CN})$, we have $\Gamma\not\models\Delta$ iff, for all $\Gamma'\Subset\Gamma$, $\Delta'\Subset\Delta$, $\Gamma'\not\models\Delta'$.
\end{corollary}

\subsection{Some properties of $\mathsf{CnCK}$}\label{sub:cond-properties}
Just as in the modal case above, we start by showing that $\mathsf{CnCK}$  shares several basic properties of $\mathsf{C}$.
\begin{lemma}\label{L:cond-monotonicity}
	Given a $\phi \in \mathcal{CN}$, a $\star \in \{+, -\}$, an $\mathcal{M} \in \mathbb{FSC}$, and any $w,v \in W$ such that $w \leq v$, $\mathcal{M}, w\models^\star \phi$ implies $\mathcal{M}, v\models^\star \phi$. 
\end{lemma}
Again, the proof is by a straightforward induction on the complexity of $\phi \in \mathcal{CN}$.
\begin{proposition}\label{P:cond-constructivity}
	$\mathsf{CnK}$ has both DP and CFP.
\end{proposition}
\begin{proof}[Proof (a sketch)]
	Again we use the construction employed in the proofs of Propositions \ref{P:C-constructivity} and \ref{P:Cnk-constructivity}; but this time we need to re-define our conditional accessibility relation relative to $\mathcal{P}(W_1\cup W_2)$. 
	
	Therefore, given a $\phi_1\vee\phi_2 \in \mathsf{CnCK}$ such that both  $\phi_1 \notin \mathsf{CnCK}$ and $\phi_2 \notin \mathsf{CnCK}$, we choose pointed models $(\mathcal{M}_1,w_1)$ and $(\mathcal{M}_2,w_2)$ such that $\mathcal{M}_i,w_i\not\models^+\phi_i$ for all $i \in \{1,2\}$; we may assume, wlog, that $W_1\cap W_2 = \emptyset$. We then choose an element $w$ outside $W_1\cup W_2$ and define the following pointed model $(\mathcal{M},w)$ for which we set:
	\begin{align*}
		W &:= \{w\}\cup W_1\cup W_2\\
		\leq &:= \{(w,v)\mid v \in W\} \cup \leq_1 \cup \leq_2\\
		R &:= \{(v,(X, Y), u)\mid v,u \in W_1,\,(v,(X\cap W_1, Y\cap W_1), u)\in R_1\}\cup\\
		&\qquad\qquad\qquad\qquad\cup \{(v,(X,Y),u)\mid v,u \in W_2,\,(v,(X\cap W_2, Y\cap W_2), u)\in R_2\}\\
		V^\star(p) &:= V^\star_1(p)\cup V^\star_2(p)\qquad\qquad\qquad \text{for }p \in Prop,\,\star\in\{+,-\}
	\end{align*}
	We need to show that $\mathcal{M}\in \mathbb{FSC}$. The only non-trivial part is the satisfaction of conditions \eqref{Cond:1} and \eqref{Cond:2} from Definition \ref{D:modal-model} by $\leq$ and $R_{(X,Y)}$ for all $X, Y \subseteq W$.
	
	As for \eqref{Cond:1}, assume that some $v',v,u \in W$ and $X, Y\subseteq W$ are such that $v'\geq v\mathrel{R_{(X,Y)}}u$. Then, by defintion of $R$, we must have either $v,u \in W_1$ or $v,u\in W_2$. Assume, wlog, that $v,u \in W_1$. Then we must have, first, that $v\mathrel{(R_1)_{(X\cap W_1, Y\cap W_1)}}u$, and, second, that $v'\geq_1 v$ so that also $v'\in W_1$. But then, since $\mathcal{M}_1$ satisfies  \eqref{Cond:1}, there must be a $u'\in W_1$ such that $v'\mathrel{(R_1)_{(X\cap W_1, Y\cap W_1)}}u'\geq_1u$ whence clearly also  $v'\mathrel{R_{(X,Y)}}u'\geq u$, so that  \eqref{Cond:1} is shown to hold for $\mathcal{M}$. We argue similarly for  \eqref{Cond:2}.
	
	The following claim can be shown by induction on the construction of $\phi \in\mathcal{CN}$:
	
	\textit{Claim}. For every $\star \in \{+, -\}$, every $i\in \{1,2\}$, every $v \in W_i$, and every $\phi \in \mathcal{CN}$, we have $\mathcal{M}, v\models^\star\phi$ iff $\mathcal{M}_i, v\models^\star\phi$. 
	
	It follows now that $\mathcal{M}, w_i\not\models^+\phi_i$ for all $i \in \{1,2\}$, and, since we have $w\leq w_1, w_2$, Lemma \ref{L:cond-monotonicity} implies that  $\mathcal{M}, w\not\models^+\phi_i$ for all $i \in \{1,2\}$, or, equivalently, that $\mathcal{M}, w\not\models^+\phi_1\vee\phi_2$, contrary to our assumption. The obtained contradiction shows that $\mathsf{CnCK}$ must have DP.
	
	Again, the argument for CFP is similar.
\end{proof}
\begin{proposition}\label{P:universal-cond-satisfaction}
	Every $\Gamma\subseteq \mathcal{CN}$ is satisfiable in $\mathsf{CnCK}$, in other words, $(\Gamma, \emptyset) \notin \mathsf{CnCK}$.
\end{proposition}
\begin{proof}[Proof (a sketch)]
	Consider $\mathcal{M} \in \mathbb{P}$ from the proof of Proposition \ref{P:universal-c-satisfaction} and expand it to $\mathcal{M}^c$ by setting $R^c:= \{(w, (\{w\}, \{w\}), w)\}$; it is straightforward to show that $\mathcal{M}^c\in \mathbb{FSC}$, and that we have $\mathcal{M}^c, w\models^\star \phi$ for every $\phi \in \mathcal{CN}$ and every $\star \in \{+, -\}$.
\end{proof}
\begin{corollary}\label{C-legacy-cond}
	$\mathsf{CnCK}$ is both non-trivial and negation-inconsistent. Moreover, $\to$ is hyperconnexive in $\mathsf{CnCK}$ and $\Rightarrow$ is connexive in $\mathsf{CnCK}$, but neither plainly nor weakly hyperconnexive in $\mathsf{CnCK}$. 
\end{corollary}
\begin{proof}[Proof (a sketch)]
	By Lemma \ref{L:cond-c} and the proof of Proposition \ref{P:c-negation-non-trivial} we know that \eqref{Contr} is a theorem of $\mathsf{CnCK}$ and that, for any $p \in Prop$, $p$ is not provable in $\mathsf{CnCK}$. Moreover, Lemma \ref{L:cond-c} implies the full hyperconnexivity of $\to$ and the full connexivity of $\Rightarrow$ also in $\mathsf{CnCK}$. The failure of both forms of hyperconnexivity for $\Rightarrow$ in $\mathsf{CnCK}$ can be seen by considering the model $\mathcal{M}^c_0\in \mathbb{FSC}$ which expands the model $\mathcal{M}_0\in \mathbb{P}$ from the proof of Proposition \ref{P:c-strong-implication} by setting $R^c_0 := \{(w,(W_0, \emptyset), w), (w,(W_0, W_0), w)\}$.
\end{proof}
We saw that hyperconnexivity of $\to$ is the source for the unique combination of properties displayed by $\mathsf{C}$. Now, the conditional operators of $\mathsf{CnCK}$, or, indeed, of any $\mathsf{C}$-based logic of conditionals might be viewed as generalizations of $\to$ in $\mathsf{C}$ (just like the conditional operators of classical logics of conditionals generalize the material implication). Therefore, we are interested to see, to what extent the connexivity properties of $\to$ in $\mathsf{C}$ carry over to $\boxto$ and $\diamondto$ in $\mathsf{CnCK}$. Moreover, $\Rightarrow$, another connexive element of $\mathsf{C}$ seems to have the following counterparts over $CN$:
\begin{itemize}
	\item $\phi\boxTo \psi$ defined as abbreviation for $(\phi\boxto\psi)\wedge(\NEG\psi\boxto\NEG\phi)$.
	
	\item $\phi\diamondTo \psi$ defined as abbreviation for $(\phi\diamondto\psi)\wedge(\NEG\psi\diamondto\NEG\phi)$.
\end{itemize}
Since $\boxto$ and $\diamondto$ are often referred to in the literature as would- and might-conditional, respectively, we will call $\boxTo$  (resp. $\diamondTo$) strong would-conditional (resp. strong might-conditional). 

Again we prepare our inquiry with a lemma:
\begin{lemma}\label{L:boxto}
For all $\phi, \psi \in \mathcal{CN}$:
\begin{enumerate}
	\item $\phi\boxto\psi \in \mathsf{CnCK}$ iff $\psi\in\mathsf{CnCK}$.
	
	\item $\phi\diamondto\psi \notin \mathsf{CnCK}$. 
\end{enumerate} 	
\end{lemma}
\begin{proof}
	(Part 1) Right-to-left part follows by \eqref{E:RNec}. In the other direction, we argue by contraposition. If $\psi\notin \mathsf{CnCK}$, we can choose a model $(\mathcal{M}, w)\in Pt(\mathbb{FSC})$ such that $\mathcal{M},w\not\models^+\psi$. Now choose a $v \notin W$ and consider the model $\mathcal{M}'$, setting $W':= W \cup \{v\}$, $\leq':= \leq \cup \{(v,v)\}$ and 
	
\begin{align*}
		R':= \{(u_1,& (X, Y), u_2)\mid u_1, u_2 \in W,\,R(u_1, (X\cap W, Y\cap W), u_2)\}\cup\\
		&\cup \{(v, \|\phi\|_\mathcal{M}, u), (v, (\lvert\phi\rvert^+_\mathcal{M}\cup\{v\},\lvert\phi\rvert^-_\mathcal{M}), u), (v, (\lvert\phi\rvert^+_\mathcal{M},\lvert\phi\rvert^-_\mathcal{M}\cup\{v\}), u), \\
		&\qquad\qquad\qquad\qquad\qquad\qquad\qquad\quad(v, (\lvert\phi\rvert^+_\mathcal{M}\cup\{v\},\lvert\phi\rvert^-_\mathcal{M}\cup\{v\}), u)\mid w \leq' u\}.
\end{align*}
It is straightforward to check that $\mathcal{M}'\in \mathbb{FSC}$. Now, arguing as in the proof of Proposition \ref{P:cond-constructivity} we can show the following:

\textit{Claim}. For every $\star \in \{+, -\}$, every $u \in W$, and every $\chi \in \mathcal{CN}$, we have $\mathcal{M}, u\models^\star\chi$ iff $\mathcal{M}', u\models^\star\chi$.

In particular, we have $\mathcal{M}', w\not\models^+\psi$. Now, depending on whether we have $\mathcal{M}', v\models^+\phi$ and whether we have $\mathcal{M}', v\models^-\phi$, our Claim implies that $\|\phi\|_{\mathcal{M}'}$ must coincide with one of the following bi-sets:
\begin{align*}
	\|\psi\|_\mathcal{M}, (\lvert\psi\rvert^+_\mathcal{M}\cup\{v\},\lvert\psi\rvert^-_\mathcal{M}), (\lvert\psi\rvert^+_\mathcal{M},\lvert\psi\rvert^-_\mathcal{M}\cup\{v\}),(\lvert\psi\rvert^+_\mathcal{M}\cup\{v\},\lvert\psi\rvert^-_\mathcal{M}\cup\{v\}),	
\end{align*}
and the definition of $R'$ implies that in every case we end up having $\mathcal{M}', v\not\models^+\phi\boxto\psi$.
	 
	(Part 2) Observe that an $\mathcal{M}\in \mathbb{FSC}$ such that $R = \emptyset$ fails $\phi\diamondto\psi$ at every node. 
\end{proof}
It is instructive to compare Lemma \ref{L:boxto} with Lemma \ref{L:strict-implication} to appreciate the differences between the modal and the conditional case. This difference is also evident from the fact, that, compared to Proposition \ref{P:c-strong-strict-implication}, the following two propositions display a much thinner analogy to Proposition \ref{P:c-strong-implication}. We treat the would-conditionals first:
\begin{proposition}\label{P:c-strong-would-conditional}
	Let $p, q \in Prop$, and let $\phi, \psi, \theta\in \mathcal{CN}$. Then the following statements hold for all $\gg \in \{\Rightarrow, \to\}$ and all $\ggg \in \{\boxto, \boxTo\}$:
	\begin{enumerate}
		\item $(\phi \boxTo \psi)\gg(\NEG\psi\boxTo\NEG\phi)\in \mathsf{CnCK}$, but $(\phi \boxTo \psi)\ggg(\NEG\psi\boxTo\NEG\phi)\notin \mathsf{CnCK}$.
		
		\item However, the same cannot be said about $\boxto$, since we have $(p \boxto q)\gg(\NEG q\boxto\NEG p)\notin \mathsf{CnCK}$. We also have $(p \boxto q)\ggg(\NEG q\boxto\NEG p)\notin \mathsf{CnCK}$ 
		
		\item We have $(\phi \boxTo \psi)\gg(\phi\boxto\psi)\in \mathsf{CnCK}$, but not vice versa, so that $\boxTo$ is stronger than $\boxto$. On the other hand, $(\phi\boxto\psi)\ggg(\phi \boxTo \psi), (\phi \boxTo \psi)\ggg(\phi\boxto\psi)\notin \mathsf{CnCK}$.
		
	\item It follows from $(\phi\Leftrightarrow\psi) \in \mathsf{CnCK}$ that $(\theta[\phi/p]\Leftrightarrow\theta[\psi/p])\in \mathsf{CnCK}$.
	
	\item If $(\phi\boxTo\psi)\wedge(\psi\boxTo\phi) \in \mathsf{CnCK}$ then also $(\theta[\phi/p]\boxTo\theta[\psi/p])\wedge (\theta[\psi/p]\boxTo\theta[\phi/p])\in \mathsf{CnCK}$.
	
		\item $(\phi\boxto\psi)\wedge(\psi\boxto\phi) \in \mathsf{CnCK}$ does \textbf{not} entail $(\theta[\phi/p]\boxto\theta[\psi/p])\wedge (\theta[\psi/p]\boxto\theta[\phi/p])\in \mathsf{CnCK}$. Similarly, $(\phi\leftrightarrow\psi) \in \mathsf{CnCK}$ does not entail $(\theta[\phi/p]\leftrightarrow\theta[\psi/p])\in \mathsf{CnCK}$.
	\end{enumerate}
\end{proposition}
\begin{proof}[Proof (a sketch)]
	(Part 1): The positive statements follow from the definition of $\boxTo$. As for the negative statements, consider $\mathcal{M}_2 \in \mathbb{FSC}$ such that $W_2 = \{w\}$, $\leq_2 = \{(w,w)\}$, $R_2 = \{(w, (\emptyset, \emptyset), w)\}$ and $V_+(r) = V_-(r) = \emptyset$ for all $r \in Prop$. It is easy to establish then, that we have all of the following:
	\begin{align*}
		\| p\|_{\mathcal{M}_2} = &\| q\|_{\mathcal{M}_2} = \| \NEG p\|_{\mathcal{M}_2} = \| \NEG q\|_{\mathcal{M}_2} =  \| p\boxto q\|_{\mathcal{M}_2}  =\\
		&= \| \NEG q\boxto \NEG p\|_{\mathcal{M}_2} = \| p\boxTo q\|_{\mathcal{M}_2} = \| \NEG q\boxTo \NEG p\|_{\mathcal{M}_2} = (\emptyset, \emptyset).
	\end{align*}
	Thus $\mathcal{M}_2$ fails $(\phi \boxTo \psi)\ggg(\NEG\psi\boxTo\NEG\phi)$ for all $\ggg \in \{\boxto, \boxTo\}$.
	
	(Part 2): Recall the model $\mathcal{M}^c_0\in \mathbb{FSC}$ from the proof of Corollary \ref{C-legacy-cond}. It is easy to check that 	$\| p\|_{\mathcal{M}^c_0} = (W_0, \emptyset)$, $\| \NEG q\|_{\mathcal{M}^c_0} = \| p\boxto q\|_{\mathcal{M}^c_0} = (W_0, W_0)$, and $\| \NEG q\boxto \NEG p\|_{\mathcal{M}^c_0} = (\emptyset, W_0)$.  Therefore, $\mathcal{M}^c_0, w\not\models^+\phi$ for all $\phi \in \{(p \boxto q)\gg(\NEG q\boxto\NEG p), (p \boxto q)\ggg(\NEG q\boxto\NEG p)\mid\gg \in \{\Rightarrow, \to\},\,\ggg \in \{\boxto, \boxTo\}\}$.
	
	(Part 3): Observe that $\mathcal{M}_2 \in \mathbb{FSC}$ defined in the proof of Part 1 fails both $(p \boxTo q)\ggg(p\boxto q)$ and $(p\boxto q)\ggg(p \boxTo q)$ for all $\ggg \in \{\boxto, \boxTo\}$, and  $\mathcal{M}^c_0\in \mathbb{FSC}$ from the proof of Corollary \ref{C-legacy-cond} fails $(p \boxto q)\gg(p\boxTo q)$ for all $\gg \in \{\Rightarrow, \to\}$. Again, the positive statements follow by definition of $\boxTo$.
	
	The proof of Part 4 is similar to the proofs of Proposition \ref{P:c-strong-implication}.4 and Proposition \ref{P:c-strong-strict-implication}.4.
	
	(Part 5): $(\phi\boxTo\psi)\wedge(\psi\boxTo\phi) \in \mathsf{CnCK}$ then, by Lemma \ref{L:boxto}, all of the formulas $\phi, \psi, \NEG\phi, \NEG\psi$ must be provable in $\mathsf{CnCK}$, whence it easily follows that $\phi\Leftrightarrow\psi\in \mathsf{CnCK}$. It remains to apply Part 4.
	
	(Part 6): By Lemma \ref{L:cond-c}.1, we must have $p \to p, (p \wedge q) \to p\in \mathsf{CnCK}$; therefore, it follows by \eqref{E:RNec} that $((p \wedge q) \to p)\boxto(p \to p), (p \to p)\boxto((p \wedge q) \to p)\in \mathsf{CnCK}$. On the other hand, recall the model $\mathcal{M}_1 \in \mathbb{P}$ described in the proof of Proposition \ref{P:c-strong-implication} and expand it to the structure $\mathcal{M}^c_1$ setting $R^c_1 := \{(w, (W_1,W_1), w), (v, (W_1,W_1), v)\}$. It is easy to check that $\mathcal{M}^c_1\in \mathbb{FSC}$, and that, moreover, we have $\| \NEG((p \wedge q) \to p)\|_{\mathcal{M}^c_1} = (W, W)$. Therefore, since we also have $w \leq w\mathrel{R_{\| \NEG((p \wedge q) \to p)\|_{\mathcal{M}^c_1}}}w$ and $\mathcal{M}^c_1, w\not\models^+\NEG(p \to p)$, it follows that $\mathcal{M}^c_1, w\not\models^+\NEG((p \wedge q) \to p)\boxto\NEG(p \to p)$, and thus also $\NEG((p \wedge q) \to p)\boxto\NEG(p \to p)\notin\mathsf{CnCK}$. 
\end{proof}
The comparison between Proposition \ref{P:c-strong-would-conditional} and Proposition \ref{P:c-strong-strict-implication} is especially telling. Not only the positive statements in Proposition \ref{P:c-strong-would-conditional}.1 and \ref{P:c-strong-would-conditional}.3 fail to extend to the newly introduced conditionals (unlike in the modal case), but also Proposition \ref{P:c-strong-would-conditional}.5, even though in complete analogy with Proposition \ref{P:c-strong-strict-implication}.5, has (as it is witnessed by its proof) an air of vacuousness or triviality about it that is completely absent in the modal case. These phenomena can also be observed, and even in a much stronger sense when we turn to might-conditionals:  
\begin{proposition}\label{P:c-strong-might-conditional}
	Let $p, q \in Prop$, and let $\phi, \psi, \theta\in \mathcal{CN}$. Then the following statements hold for all $\gg \in \{\Rightarrow, \to\}$ and all $\ggg \in \{\boxto, \boxTo, \diamondto, \diamondTo\}$:
	\begin{enumerate}
		\item $(\phi \diamondTo \psi)\gg(\NEG\psi\diamondTo\NEG\phi)\in \mathsf{CnCK}$, but $(\phi \diamondTo \psi)\ggg(\NEG\psi\diamondTo\NEG\phi)\notin \mathsf{CnCK}$.
		
		\item However, the same cannot be said about $\diamondto$, since we have $(p \diamondto q)\gg(\NEG q\diamondto\NEG p)\notin \mathsf{CnCK}$. We also have $(p \diamondto q)\ggg(\NEG q\diamondto\NEG p)\notin \mathsf{CnCK}$. 
		
		\item We have $(\phi \diamondTo \psi)\gg(\phi\diamondto\psi)\in \mathsf{CnCK}$, but not vice versa, so that $\diamondTo$ is stronger than $\diamondto$. On the other hand, $(\phi\diamondto\psi)\ggg(\phi \diamondTo \psi), (\phi \diamondTo \psi)\ggg(\phi\diamondto\psi)\notin \mathsf{CnCK}$.
		
		\item For all $\geqq \in \{\diamondto, \diamondTo\}$, $(\phi\geqq\psi)\wedge(\psi\geqq\phi) \in \mathsf{CnCK}$ (vacuously) entails $(\theta[\phi/p]\geqq\theta[\psi/p])\wedge (\theta[\psi/p]\geqq\theta[\phi/p])\in \mathsf{CnCK}$.
	\end{enumerate}
\end{proposition}
\begin{proof}[Proof (a sketch)]
Again, the positive statements in Parts 1 and 3 are established by checking the definitions; the negative statements in Parts 1--3 for $\ggg \in \{\diamondto, \diamondTo\}$ follow from Lemma \ref{L:boxto}.2; the negative statements for $\gg \in \{\Rightarrow, \to\}$ and $\ggg \in \{\boxto, \boxTo\}$ are mostly established in the same way as for Proposition \ref{P:c-strong-would-conditional}. In particular, the negative statements of Part 2 and $(p \diamondto q)\gg(p\diamondTo q)$ for all $\gg \in \{\Rightarrow, \to\}$ are failed by $\mathcal{M}^c_0\in \mathbb{FSC}$ from the proof of Corollary \ref{C-legacy-cond}. Next, $(p\diamondto q)\ggg(p \diamondTo q), (p \diamondTo q)\ggg(p\diamondto q)$ for all $\ggg \in \{\boxto, \boxTo\}$ fail in the model $\mathcal{M}_2\in \mathbb{FSC}$ constructed in the proof of Proposition \ref{P:c-strong-would-conditional}. 

Finally, Part 4 follows from Lemma \ref{L:boxto}.2
\end{proof}  
We look into the connexive properties of these conditional operators next:
\begin{proposition}\label{P:cond-connexivity}
$\boxto$ and $\diamondto$ are weakly partially hyperconnexive in $\mathsf{CnCK}$ but neither weakly nor plainly connexive. $\boxTo$ and $\diamondTo$ are weakly partially connexive in $\mathsf{CnCK}$ but neither weakly plainly connexive nor weakly partially hyperconnexive. 	
\end{proposition}
\begin{proof}[Proof (a sketch)]
	The positive statements of the Proposition are established by checking the definitions. As for the negative parts, recall again the model $\mathcal{M}_0\in \mathbb{P}$ constructed in the proof of Proposition \ref{P:c-strong-implication} and expand it to $\mathcal{M}^{c1}_0$ by setting $R^{c1}_0:= \{(w, (W_0, \emptyset), w), (w, (\emptyset, W_0), w)\}$. Observe that $\mathcal{M}^{c1}_0$ is both in $\mathbb{FSC}$ and fails every formula in $\{\NEG(\NEG p\gg p), (p\gg\NEG p)\gg\NEG(p\gg p)\mid \gg \in \{\boxto, \diamondto, \boxTo, \diamondTo\}\}$.
	
	Moreover, if we recall $\mathcal{M}^c_0\in \mathbb{FSC}$ from the proof of Corollary \ref{C-legacy-cond} again, we can easily check that both  $\mathcal{M}^{c}_0, w\models^+ \NEG(p \boxTo q)\wedge\NEG(p \diamondTo q)$ and $\mathcal{M}^{c}_0, w\not\models^+ (p \boxTo \NEG q)\vee(p \diamondTo \NEG q)$.
\end{proof} 
In relation to the negation-inconsistencies of $\mathsf{C}$, the situation in $\mathsf{CnCK}$ is similar to the one with the conditional analogues of strong implication: whereas $\mathsf{CnCK}$ is able to provide a purely conditional counterpart to each of these negation-inconsistencies, this reflection looks less conspicuous than the one established in Proposition \ref{P:mod-negation-non-trivial} for $\mathsf{CnK}$, especially given that it fails to replace implications with conditionals in the case when a given negation-inconsistency happens to be an implicative sentence. More precisely, the following proposition holds:
\begin{proposition}\label{P:cond-negation-non-trivial}
	For every $\phi\wedge\NEG\phi\in\mathsf{CnCK}$ and every $\psi \in \mathcal{CN}$, we have $(\psi\boxto\phi)\wedge\NEG(\psi\boxto\phi)\in\mathsf{CnCK}$. 
\end{proposition}
\begin{proof}[Proof (a sketch)]
	By \eqref{E:RNec}, \eqref{E:a1}, and \eqref{E:a5}. 
\end{proof}
Just as in the modal case, Lemma \ref{L:cond-c} and Proposition \ref{P:cond-negation-non-trivial} together imply that:
\begin{corollary}\label{C:cond-negation-non-trivial}
	For every $\phi\wedge\NEG\phi\in\mathsf{C}$ and every $\psi \in \mathcal{CN}$, we have $(\psi\boxto\phi)\wedge\NEG(\psi\boxto\phi)\in\mathsf{CnCK}$.	
\end{corollary}
However, the failure to replace the main implication sign with a conditional is obvious from the following example: the following formula fails to be a theorem of $\mathsf{CnCK}$:
\begin{equation}\label{Contr-cond}
	((p_1\wedge \NEG p_1)\boxto p_1)\wedge\NEG((p_1\wedge \NEG p_1)\boxto p_1)
\end{equation} 
Indeed, \eqref{Contr-cond} fails, for instance, in $\mathcal{M}_2\in \mathbb{FSC}$ from the proof of Proposition \ref{P:c-strong-would-conditional}.

Even though the $\mathsf{CnCK}$ counterparts to many special features of $\mathsf{C}$ are but a faint echo of the originals, $\mathsf{CnCK}$ shows a very tight connection with $\mathsf{CnK}$, which, as we saw in Section \ref{S:modal} shows a high degree of harmony with $\mathsf{C}$. It is this connection with $\mathsf{CnK}$ which provides one (albeit admittedly insufficient) argument for viewing $\mathsf{CnCK}$ as the correct minimal $\mathsf{C}$-based logic of conditionals, therefore, we will look into this connection in more detail.

Note that the semantics of conditionals in $\mathsf{CnCK}$ encourages the reading of $\phi\boxto\psi$ (resp. $\phi\diamondto\psi$) almost as sentence-indexed modal boxes and diamonds of the form $\Box_\phi\psi$ (resp. $\Diamond_\phi\psi$): the conditionals are evaluated like modalities, the only difference being that the binary accessibility relation is in each case generated by the bi-extension of the antecedent in the corresponding conditional model. Thus one can legitimately ask, what is the modal logic of these indexed boxes and diamonds induced by the conditionals of $\mathsf{CnCK}$.

This question can be made precise in the following way: for a given $\phi \in \mathcal{CN}$, consider the mapping $Tr_\phi:\mathcal{MD}\to\mathcal{CN}$ inductively defined by
\begin{align*}
	Tr_\phi(p)&:= p&&p\in Prop\\
	Tr_\phi(\NEG\psi)&:=\NEG Tr_\phi(\psi)\\
	Tr_\phi(\psi\ast\chi)&:= Tr_\phi(\psi)\ast Tr_\phi(\chi)&&\ast\in \{\wedge, \vee, \to\}\\
	Tr_\phi(\Box\psi)&:= \phi\boxto Tr_\phi(\psi)&&Tr_\phi(\Diamond\psi):= \phi\diamondto Tr_\phi(\psi)
\end{align*} 
The boxes and diamonds induced by conditionals of $\mathsf{CnCK}$ are described by a logic $\mathsf{L}$ over $\mathcal{MD}$ iff $Tr_\phi$ is a faithful embedding of $\mathsf{L}$ into $\mathsf{CnCK}$ for every $\phi\in \mathcal{CN}$, in other words iff for all $\Gamma, \Delta\subseteq \mathsf{CnCK}$ we have:
$$
\Gamma\models_\mathsf{L}\Delta\text{ iff }Tr_\phi(\Gamma)\models_{\mathsf{CnCK}}Tr_\phi(\Delta).
$$
In this case, we call  $\mathsf{L}$ a \textit{modal companion} of $\mathsf{CnCK}$.

Now, the following is easily shown:
\begin{proposition}\label{P:modal-companion}
	$\mathsf{CnK}$ is a modal companion of $\mathsf{CnCK}$.
\end{proposition}
\begin{proof}
	We argue by contraposition. Fix a $\phi \in \mathcal{CN}$ and assume that $\Gamma\not\models_{\mathsf{CnK}}\Delta$ and choose any $(\mathcal{M}, w)\in Pt(\mathbb{FSM})$ such that $\mathcal{M}, w\models^+_m (\Gamma, \Delta)$. Then consider the model $\mathcal{N} = (W, \leq, R_n, V^+, V^-)$ (that is to say, replace $R$ with $R_n$ in $\mathcal{M}$) such that
	$$
	R_n:= \{(v, (X,Y), u)\mid X,Y\subseteq W,\,R(v,u)\}.
	$$
	Given that $\mathcal{M}\in\mathbb{FSM}$, it is immediate that $\mathcal{N}\in \mathbb{FSC}$. Next, a simple inductive argument shows that for every $\psi \in \mathcal{MD}$, every $\ast \in \{+, -\}$ and every $v \in W$ we must have $\mathcal{M}, v\models^+_m \psi$ iff $\mathcal{N}, v\models^+ Tr_\phi(\psi)$ so that $\mathcal{N}, w\models^+ (Tr_\phi(\Gamma), Tr_\phi(\Delta))$ also holds.
	
	In the other direction, assume that $Tr_\phi(\Gamma)\not\models_{\mathsf{CnCK}}Tr_\phi(\Delta)$ and choose any $(\mathcal{M}, w)\in Pt(\mathbb{FSC})$ such that $\mathcal{M}, w\models^+ (Tr_\phi(\Gamma), Tr_\phi(\Delta))$; now consider the model $\mathcal{M}_m = (W, \leq, R_m, V^+, V^-)$ (that is to say, replace $R$ with $R_m$ in $\mathcal{M}$) such that
	$$
	R_m:= \{(v, u)\mid R(v,\|\phi\|_\mathcal{M},u)\}.
	$$
	Again, given that $\mathcal{M}\in\mathbb{FSC}$, it is immediate that $\mathcal{M}_m\in \mathbb{FSM}$, and a simple inductive argument shows that for every $\psi \in \mathcal{MD}$, every $\ast \in \{+, -\}$ and every $v \in W$ we must have $\mathcal{M}_m, v\models^+_m \psi$ iff $\mathcal{M}, v\models^+ Tr_\phi(\psi)$ so that $\mathcal{M}_m, w\models^+ (\Gamma, \Delta)$ also holds.	
\end{proof}
Before we end this section, we would like to observe that, reversing the direction of $Tr$ is possible to a limited extent, in other words, that $\mathsf{CnCK}$ can be interpreted in $\mathsf{CnK}$. Indeed, consider the following mapping $I:\mathcal{CN}\to\mathcal{MD}$ defined inductively by:
\begin{align*}
	I(p)&:= p&&p\in Prop\\
	I(\NEG\psi)&:=\NEG I(\psi)\\
	I(\psi\ast\chi)&:= I(\psi)\ast I(\chi)&&\ast\in \{\wedge, \vee, \to\}\\
	I(\psi\boxto\chi)&:= I(\psi)\to_s I(\chi)&&I(\psi\diamondto\chi):= \Diamond(I(\psi)\wedge I(\chi))
\end{align*} 
We can show the following
\begin{proposition}\label{P:interpretability}
	For all $\Gamma, \Delta \subseteq \mathcal{CN}$, $\Gamma\models_{\mathsf{CnCK}}\Delta$ implies $I(\Gamma)\models_{\mathsf{CnK}}I(\Delta)$ but not vice versa.
\end{proposition}
\begin{proof}[Proof (a sketch)]
	For the positive part, we argue by contraposition. Assume that $I(\Gamma)\not\models_{\mathsf{CnK}}I(\Delta)$ and choose an $(\mathcal{M}, w)\in Pt(\mathbb{FSM})$ such that $\mathcal{M}, w\models^+_m (I(\Gamma), I(\Delta))$. Then consider the model $\mathcal{N} = (W, \leq, R_n, V^+, V^-)$ (that is to say, replace $R$ with $R_n$ in $\mathcal{M}$) setting that $R_n(v, (X,Y), u)$ iff all of the following holds: (1) $X,Y\subseteq W$ are closed under $\leq$, (2) $R(v,u)$, and (3) $u \in X$.
	
	It is clear then that both $\mathcal{N}\in \mathbb{FSC}$ and that, for every $\phi\in \mathcal{CN}$, every $\ast \in \{+, -\}$, and every $v \in W$ we must have $\mathcal{M}, v\models^+_m I(\psi)$ iff $\mathcal{N}, v\models^+ \psi$ so that $\mathcal{N}, w\models^+ (\Gamma, \Delta)$ also holds.	
	
	As for the negative part, observe that we have $I(p\diamondto (q \wedge r))\models_{\mathsf{CnK}} I((p \wedge q)\diamondto r)$ but $(p\diamondto (q \wedge r))\not\models_{\mathsf{CnCK}} ((p \wedge q)\diamondto r)$; similarly, we have $I(p\boxto (q \to r))\models_{\mathsf{CnK}} I((p \wedge q)\boxto r)$ but $(p\boxto (q \to r))\not\models_{\mathsf{CnCK}} ((p \wedge q)\boxto r)$.
\end{proof}
\section{$\mathsf{CnCK}_R$, the reflexive extension of $\mathsf{CnCK}$}\label{S:reflexive}

Although the fading away of connexivity properties at the level of would- and might-conditionals of $\mathsf{CnCK}$ can be explained by the wide scope of generalization provided by these conditionals for the implication in $\mathsf{C}$, this generalization must not, at the very least, exclude the possibility of restoring these properties in some natural and well-behaved extensions of $\mathsf{CnCK}$. Indeed, otherwise one might legitimately question whether the conditionals of $\mathsf{CnCK}$ really generalize the implication of $\mathsf{C}$ rather than some completely different binary connective.

If we focus on the properties of would-conditionals, it is indeed easy to find exactly such an extension, and in this respect our source of inspiration is, again \cite{wu}. The extension in question is given by restricting our attention to the class $\mathbb{FSC}_R$ of \textit{reflexive conditional Fischer-Servi models} given by:
$$
\mathbb{FSC}_R:= \{\mathcal{M} \in \mathbb{FSC}\mid(\forall w, v \in W)(\forall X, Y\subseteq W)(R(w (X,Y), v)\text{ implies }v \in X)\}.
$$
The logic of this class, which we will denote by $\mathsf{CnCK}_R$ is then defined as $\mathsf{CnCK}_R:= \mathfrak{L}(CN, \mathbb{FSC}_R, (\models^+,\models^-))$, and the methods described in Section \ref{sub:cond-axioms}, allow, with minimal changes, to obtain an axiomatization of this logic; namely, it turns out that we have $\mathsf{CnCK}_R = \mathtt{CnCK}_R[CN]$, such that $\mathtt{CnCK}_R:= \mathtt{CnCK}+(\eqref{E:a8};)$ assuming that 
\begin{equation}\label{E:a8}\tag{$\gamma_8$}
	\phi\boxto\phi
\end{equation}
We immediately see that the reflexivity refers in this case to the reflexivity of $\boxto$ rather than to any sort of reflexivity of the conditional accessibility relation; indeed, note that every $\mathcal{M}\in \mathbb{FSC}$ such that $R = \emptyset$ is also in $\mathbb{FSC}_R$.

Another immediate observation is that many properties established for $\mathsf{CnCK}$ in Section \ref{sub:cond-properties}, can be shown to hold also for $\mathsf{CnCK}_R$ by a simple repetition of their proof. We collect some of these properties in the next proposition:
\begin{proposition}\label{P:reflexive-legacy}
	Propositions \ref{P:cond-constructivity}, \ref{P:universal-cond-satisfaction}, and Corollary \ref{C-legacy-cond} can be shown to hold also for $\mathsf{CnCK}_R$.
\end{proposition}
As soon as we consider the analogues of Lemma \ref{L:boxto} for $\mathsf{CnCK}_R$, the differences between the two conditional logics become apparent; more precisely we find that:
\begin{lemma}\label{L:boxto-r}
	For all $\phi, \psi \in \mathcal{CN}$:
	\begin{enumerate}
		\item $\phi\boxto\psi \in \mathsf{CnCK}_R$ iff $\phi\to\psi\in\mathsf{CnCK}_R$.
		
		\item $\phi\diamondto\psi \notin \mathsf{CnCK}_R$. 
	\end{enumerate} 	
\end{lemma}
\begin{proof}
	(Part 1) To show the right-to left direction we reason as follows:
	\begin{align}
		&\phi\to\psi\in \mathsf{CnCK}_R\label{E:f1}&&\text{assumption}\\
		&\phi\boxto(\phi\to\psi)\in \mathsf{CnCK}_R\label{E:f2}&&\text{by \eqref{E:f1}, \eqref{E:RNec}}\\
		&(\phi\boxto\phi)\to(\phi\boxto\psi)\in \mathsf{CnCK}_R\label{E:f3}&&\text{by \eqref{E:f2}, \eqref{E:T1}}\\
		&\phi\boxto\psi\in \mathsf{CnCK}_R\label{E:f4}&&\text{by \eqref{E:f3}, \eqref{E:a8}}
	\end{align}
	 In the other direction, we argue by contraposition. If $\phi\to\psi\notin \mathsf{CnCK}_R$, we can choose a model $(\mathcal{M}, w)\in Pt(\mathbb{FSC}_R)$ such that $\mathcal{M},w\models^+(\{\phi\}, \{\psi\})$. Now choose a $v \notin W$ and consider the model $\mathcal{M}'$ as described in the proof of Lemma \ref{L:boxto}. It is immediate to show that both $\mathcal{M}', v\not\models^+\phi\boxto\psi$ and, given that $w \in \lvert\phi\rvert^+_\mathcal{M}$, that also $\mathcal{M}'\in \mathbb{FSC}_R$.
	
	For Part 2, we can simply repeat the proof of Lemma \ref{L:boxto}.2.
\end{proof}
W.r.t. the defined conditionals, Lemma \ref{L:boxto-r}.1 entails the following corollary:
\begin{corollary}\label{C:boxto-r}
	For all $\phi, \psi\in \mathcal{CN}$, $(\phi\boxto\psi)\wedge(\psi\boxto\phi)\in \mathsf{CnCK}_R$ (resp. $\phi\boxTo\psi\in \mathsf{CnCK}_R$, $(\phi\boxTo\psi)\wedge(\psi\boxTo\phi)\in \mathsf{CnCK}_R$) iff $\phi\leftrightarrow\psi\in \mathsf{CnCK}_R$ (resp. $\phi\Rightarrow\psi\in \mathsf{CnCK}_R$, $\phi\Leftrightarrow\psi\in \mathsf{CnCK}_R$).
\end{corollary}
We see that Lemma \ref{L:boxto-r}.1 and Corollary \ref{C:boxto-r} are completely analogous to Lemma \ref{L:strict-implication} and Corollary \ref{C:strict-implication} in the modal case, whereas Lemma \ref{L:boxto-r}.2 simply repeats Lemma \ref{L:boxto}.2 for $\mathsf{CnCK}_R$. This can be taken as an indication that the properties shown above for strict implication and its related connectives can be also shown to hold for $\boxto$ and its relatives in $\mathsf{CnCK}_R$, whereas all the forms of might-conditionals in $\mathsf{CnCK}_R$ remain very close to the might-conditionals in $\mathsf{CnCK}$. The following series of propositions confirms this expectation: 
\begin{proposition}\label{P:c-strong-would-conditional-reflexive}
	Let $p, q \in Prop$, and let $\phi, \psi, \theta\in \mathcal{CN}$. Then the following statements hold for all $\gg \in \{\Rightarrow, \to, \boxto, \boxTo\}$:
	\begin{enumerate}
		\item $(\phi \boxTo \psi)\gg(\NEG\psi\boxTo\NEG\phi)\in \mathsf{CnCK}_R$.
		
		\item However, the same cannot be said about $\boxto$, since we have $(p \boxto q)\gg(\NEG q\boxto\NEG p)\notin \mathsf{CnCK}_R$. 
		
		\item We have $(\phi \boxTo \psi)\gg(\phi\boxto\psi)\in \mathsf{CnCK}_R$, but not vice versa, so that $\boxTo$ is stronger than $\boxto$.
		
		\item It follows from $(\phi\Leftrightarrow\psi) \in \mathsf{CnCK}_R$ that $(\theta[\phi/p]\Leftrightarrow\theta[\psi/p])\in \mathsf{CnCK}_R$.
		
		\item If $(\phi\boxTo\psi)\wedge(\psi\boxTo\phi) \in \mathsf{CnCK}_R$ then also $(\theta[\phi/p]\boxTo\theta[\psi/p])\wedge (\theta[\psi/p]\boxTo\theta[\phi/p])\in \mathsf{CnCK}_R$.
		
		\item $(\phi\boxto\psi)\wedge(\psi\boxto\phi) \in \mathsf{CnCK}_R$ does \textbf{not} entail $(\theta[\phi/p]\boxto\theta[\psi/p])\wedge (\theta[\psi/p]\boxto\theta[\phi/p])\in \mathsf{CnCK}_R$.
	\end{enumerate}
\end{proposition}
\begin{proof}[Proof (a sketch)]
	As for Part 1, for $\gg\in \{\Rightarrow, \to\}$ it follows from Proposition \ref{P:c-strong-would-conditional}.1 and $\mathtt{CnCK}_R \supseteq\mathtt{CnCK}$. These two cases, together with Lemma \ref{L:boxto-r}.1 and Corollary \ref{C:boxto-r}, also imply Part 1 for $\gg \in \{\boxto, \boxTo\}$. Similar considerations apply to positive statements in Part 3.
	
	As for the negative statements of Part 3, we need to recall $\mathcal{M}^{c}_0$ defined in the proof of Corollary \ref{C-legacy-cond}. In the proof of Proposition \ref{P:c-strong-would-conditional} it was shown that $(\phi\boxto\psi)\gg(\phi \boxTo \psi)$ for  $\gg\in \{\Rightarrow, \to\}$, the negative statements for $\gg \in \{\boxto, \boxTo\}$ now follow Lemma \ref{L:boxto-r}.1 and Corollary \ref{C:boxto-r}.
	
	As for Part 2, note, again, that $\mathcal{M}^c_0\in \mathbb{FSC}_R$. Part 4 is proven similarly to Proposition \ref{P:c-strong-would-conditional}.4 and Part 5 is proven similarly to Proposition \ref{P:c-strong-strict-implication}.5. Finally, as for Part 6, observe that the counter-model $\mathcal{M}^c_1$ described in the Proof of Proposition \ref{P:c-strong-would-conditional}.6 is, again, clearly in $\mathbb{FSC}_R$.
\end{proof}
\begin{proposition}\label{P:c-strong-might-conditional-reflexive}
	Let $p, q \in Prop$, and let $\phi, \psi, \theta\in \mathcal{CN}$. Then the following statements hold for all $\gg \in \{\Rightarrow, \to,\boxto, \boxTo\}$ and all $\ggg \in \{\diamondto, \diamondTo\}$:
	\begin{enumerate}
		\item $(\phi \diamondTo \psi)\gg(\NEG\psi\diamondTo\NEG\phi)\in \mathsf{CnCK}_R$, but $(\phi \diamondTo \psi)\ggg(\NEG\psi\diamondTo\NEG\phi)\notin \mathsf{CnCK}_R$.
		
		\item However, the same cannot be said about $\diamondto$, since we have $(p \diamondto q)\gg(\NEG q\diamondto\NEG p)\notin \mathsf{CnCK}_R$. We also have $(p \diamondto q)\ggg(\NEG q\diamondto\NEG p)\notin \mathsf{CnCK}_R$. 
		
		\item We have $(\phi \diamondTo \psi)\gg(\phi\diamondto\psi)\in \mathsf{CnCK}_R$, but not vice versa, so that $\diamondTo$ is stronger than $\diamondto$. On the other hand, $(\phi \diamondTo \psi)\ggg(\phi\diamondto\psi)\notin \mathsf{CnCK}_R$.
		
		\item  $(\phi\ggg\psi)\wedge(\psi\ggg\phi) \in \mathsf{CnCK}_R$ (vacuously) entails $(\theta[\phi/p]\ggg\theta[\psi/p])\wedge (\theta[\psi/p]\ggg\theta[\phi/p])\in \mathsf{CnCK}_R$.
	\end{enumerate}
\end{proposition}
\begin{proof}[Proof (a sketch)]
	The proof is very similar to that of Proposition \ref{P:c-strong-might-conditional}; the additional positive statements in Parts 1 and 3 follow from the respective parts of Proposition \ref{P:c-strong-might-conditional}, Lemma \ref{L:boxto-r}.1, and Corollary \ref{C:boxto-r}. The negative statements of Parts 1--3 all follow from Lemma \ref{L:boxto-r}.2.
	
%
\end{proof}  
\begin{proposition}\label{P:cond-would-connexivity-reflexive}
$\boxto$ is fully hyperconnexive in $\mathsf{CnCK}_R$, $\boxTo$ is fully connexive (but neither weakly nor plainly hyperconnexive) in $\mathsf{CnCK}_R$, $\diamondto$ is weakly partially hyperconnexive but neither weakly nor plainly connexive in $\mathsf{CnCK}_R$, and $\diamondTo$ is weakly partially connexive but neither weakly plainly connexive nor partially hyperconnexive in $\mathsf{CnCK}_R$.	
\end{proposition}
\begin{proof}
	As for Aristotle's Thesis, all of the following formulas can be shown to be theorems of $\mathsf{CnCK}_R$ for every $\phi \in \mathcal{CN}$ (at every step, Proposition \ref{P:c-strong-would-conditional-reflexive}.4 is applied):
	\begin{align}
		&\NEG\NEG(\NEG\phi\boxto\NEG\phi)\label{E:ss1}&&\eqref{E:a8},\,\eqref{E:t-double-neg}\\
		&\NEG(\NEG\phi\boxto\NEG\NEG\phi)\label{E:ss2}&&\eqref{E:ss1},\,\eqref{E:T4}\\
			&\NEG(\NEG\phi\boxto\phi)\label{E:ss3}&&\eqref{E:ss2},\,\eqref{E:t-double-neg}\\
			&\NEG(\NEG\phi\boxTo\phi)\Leftrightarrow(\NEG(\NEG\phi\boxto\phi)\vee\NEG(\NEG\phi\boxto\NEG\NEG\phi))\label{E:ss4}&&\eqref{E:t-neg-vee}\\
			&\NEG(\NEG\phi\boxTo\phi)\Leftrightarrow(\NEG(\NEG\phi\boxto\phi)\vee\NEG(\NEG\phi\boxto\phi))\label{E:ss5}&&\eqref{E:ss4},\,\eqref{E:t-double-neg}\\
			&\NEG(\NEG\phi\boxTo\phi)\label{E:ss6}&&\eqref{E:ss3},\,\eqref{E:ss5}		
	\end{align}
We treat the numerous forms of Boethius' Thesis next. (WBT$\boxto$), (WCBT$\boxto$) and (WBT$\boxTo$) all follow from Proposition \ref{P:cond-connexivity} together with the fact that $\mathsf{CnCK}_R\supseteq \mathsf{CnCK}$.  (BT$\boxto$), (CBT$\boxto$) and (BT$\boxTo$) follow from the latter theses by Lemma \ref{L:boxto-r} and Corollary \ref{C:boxto-r}. Similarly, Proposition \ref{P:cond-connexivity} together with $\mathsf{CnCK}_R\supseteq \mathsf{CnCK}$ imply the weak partial connexivity (resp. hyperconnexivity) for $\diamondto$ (resp. $\diamondTo$).

We turn to the negative statements and recall the model $\mathcal{M}^{c}_0\in \mathbb{FSC}_R$ defined in the proof of Corollary \ref{C-legacy-cond}. We have shown in the proof of Proposition \ref{P:cond-connexivity} that this model fails both (WCBT$\boxTo$) and (WCBT$\diamondTo$), therefore, by Lemma \ref{L:boxto-r}, (CBT$\boxTo$) must also fail. Finally, as for (BT$\gg$) and (CBT$\gg$) for $\gg\in\{\diamondto, \diamondTo\}$, their failure is entailed by Lemma \ref{L:boxto-r}.2.
\end{proof}
When it comes to the reflection of the negation-inconsistencies of $\mathsf{C}$ on the level of $CN$, the situation that we find with the would-conditionals in $\mathsf{CnCK}_R$ is also much closer to the modal case. More precisely, the following can be shown to hold:
\begin{proposition}\label{P:cond-negation-non-trivial-r}
	The following statements hold:
	\begin{enumerate}
		\item For every $\phi\wedge\NEG\phi\in\mathsf{CnCK}_R$ and every $\psi \in \mathcal{CN}$, we have $(\psi\boxto\phi)\wedge\NEG(\psi\boxto\phi)\in\mathsf{CnCK}_R$.
		
		\item Whenever $(\phi\to\psi)\wedge\NEG(\phi\to\psi)\in\mathsf{CnCK}_R$, we have $(\phi\boxto\psi)\wedge\NEG(\phi\boxto\psi)\in\mathsf{CnCK}_R$.
		
		\item If $(\phi\Rightarrow\psi)\wedge\NEG(\phi\Rightarrow\psi)\in\mathsf{CnCK}_R$, then $(\phi\boxTo\psi)\wedge\NEG(\phi\boxTo\psi)\in\mathsf{CnCK}_R$.
	\end{enumerate}
\end{proposition}
\begin{proof}[Proof (a sketch)]
	Part 1 is proved similarly to Proposition \ref{P:cond-negation-non-trivial}. Part 2 (resp. Part 3) follows from Lemma \ref{L:boxto-r}.1 (resp. Corollary \ref{C:boxto-r}) and \eqref{E:T4}.
\end{proof}
Again, Lemma \ref{L:cond-c} and Proposition \ref{P:cond-negation-non-trivial-r} together imply that:
\begin{corollary}\label{C:cond-negation-non-trivial-r}
	The following statements hold:
\begin{enumerate}
	\item For every $\phi\wedge\NEG\phi\in\mathsf{C}$ and every $\psi \in \mathcal{CN}$, we have $(\psi\boxto\phi)\wedge\NEG(\psi\boxto\phi)\in\mathsf{CnCK}_R$.
	
	\item Whenever $(\phi\to\psi)\wedge\NEG(\phi\to\psi)\in\mathsf{C}$, we have $(\phi\boxto\psi)\wedge\NEG(\phi\boxto\psi)\in\mathsf{CnCK}_R$.
	
	\item If $(\phi\Rightarrow\psi)\wedge\NEG(\phi\Rightarrow\psi)\in\mathsf{C}$, then $(\phi\boxTo\psi)\wedge\NEG(\phi\boxTo\psi)\in\mathsf{CnCK}_R$.
\end{enumerate}	
\end{corollary}
In particular, by Corollary \ref{C:cond-negation-non-trivial-r}, Lemma \ref{L:cond-c}, \eqref{Contr}, and \eqref{Contr-strong}, we get that \eqref{Contr-m} is a theorem of $\mathsf{CnCK}_R$ for all $\gg \in \{\to, \Rightarrow, \boxto, \boxTo\}$.

Even though the would-conditionals in $\mathsf{CnCK}_R$ display a strong connexive profile, might-conditionals remain pretty much in the state in which we saw them in $\mathsf{CnCK}$. Although the extensions of $\mathsf{CnCK}$ strengthening might-conditionals in much the same way as we saw $\mathsf{CnCK}_R$ strengthening would-conditionals seem theoretically possible, we do not expect that any such extensions will look particularly nice or well-behaved. The reason for this belief of ours is that, as the formulations of the clauses \eqref{Cl:diamto+} and \eqref{Cl:diamto-} betray, $\diamondto$ generalizes conjunction rather than implication, its official appellation ``might-conditional'' notwithstanding. Therefore, any imposition of the specifically implicational principles upon such a generalized conjunction will always fall completely out of harmony with the inner nature of this generalized conjunction and will only be possible under very special conditions. 

Before we end this section, we would like to make two more observations. First, despite its substantial strengthening of $\mathsf{CnCK}$, $\mathsf{CnCK}_R$ retains a non-trivial connection to $\mathsf{CnK}$, even though this connection is more tenuous than the modal companionship established for $\mathsf{CnCK}$ in Proposition \ref{P:modal-companion} earlier:
\begin{proposition}\label{P:modal-companion-r}
	For every $\phi \in \mathsf{CnCK}_R$, $Tr_\phi$ is a faithful embedding of $\mathsf{CnK}$ into $\mathsf{CnCK}_R$. In other words, for all $\Gamma, \Delta\subseteq \mathsf{CnCK}_R$ we have:
	$$
	\Gamma\models_\mathsf{CnK}\Delta\text{ iff }Tr_\phi(\Gamma)\models_{\mathsf{CnCK}_R}Tr_\phi(\Delta).
	$$
\end{proposition}
\begin{proof}[Proof (a sketch)]
Again, we argue by contraposition. Fix a $\phi \in  \mathsf{CnCK}_R$, assume that $\Gamma\not\models_{\mathsf{CnK}}\Delta$, and choose any $(\mathcal{M}, w)\in Pt(\mathbb{FSM})$ such that $\mathcal{M}, w\models^+_m (\Gamma, \Delta)$. Then consider the model $\mathcal{N} = (W, \leq, R_n, V^+, V^-)$ (that is to say, replace $R$ with $R_n$ in $\mathcal{M}$) such that
$$
R_n:= \{(v, (W,X), u)\mid X\subseteq W,\,R(v,u)\}.
$$
Given that $\mathcal{M}\in \mathbb{FSM}$, it is immediate that $\mathcal{N}\in \mathbb{FSC}_R$. Moreover, since $\phi \in  \mathsf{CnCK}_R$, we know that $\lvert\phi\rvert^+_\mathcal{N} = W$. Therefore, a simple inductive argument shows that for every $\psi \in \mathcal{MD}$, every $\ast \in \{+, -\}$ and every $v \in W$ we must have $\mathcal{M}, v\models^+_m \psi$ iff $\mathcal{N}, v\models^+ Tr_\phi(\psi)$ so that $\mathcal{N}, w\models^+ (Tr_\phi(\Gamma), Tr_\phi(\Delta))$ also holds.

In the other direction, the argument is the same as in the proof of Proposition \ref{P:modal-companion}.
\end{proof}
That Proposition \ref{P:modal-companion-r} can not be strengthened to a full analogue of Proposition \ref{P:modal-companion} follows from the fact that, for every $p \in Prop$, $Tr_p(p) = p\boxto p \in  \mathsf{CnCK}_R$, even though $p \notin \mathsf{CnK}$.

Secondly, Proposition \ref{P:interpretability} carries over to $\mathsf{CnCK}_R$ without any changes in that we have
\begin{proposition}\label{P:interpretability-r}
	For all $\Gamma, \Delta \subseteq \mathcal{CN}$, $\Gamma\models_{\mathsf{CnCK}_R}\Delta$ implies $I(\Gamma)\models_{\mathsf{CnK}}I(\Delta)$ but not vice versa.
\end{proposition}
The proof is also the same as for Proposition \ref{P:interpretability}.

\section{Conclusions and future work}\label{S:conclusions}
This paper was written in the hope to make first steps in an in-depth investigation of $\mathsf{C}$-based conditional and modal logics, focusing on the motivation for the introduced systems and systematically charting their relations to other logics based on similar ideas.

Thus far, we have introduced three new systems, the $\mathsf{C}$-based modal logic $\mathsf{CnK}$ and the $\mathsf{C}$-based conditional logics $\mathsf{CnCK}$ and $\mathsf{CnCK}_R$ all of which are so tightly related to one another that introducing just one or two of them and omitting the remaining ones would lead to noticeable gaps in the discourse of this paper. For instance, $\mathsf{CnK}$ and $\mathsf{CnCK}$ are easily recognized as two expressions of the same set of logical principles which are merely projected onto two different language environments, the modal and the conditional one. Besides, $\mathsf{CnK}$ faithfully reproduces many principles of $\mathsf{C}$ related to connexivity, the inter-relations between plain and strong implication, the provability of negation-inconsistencies in a purely modal form, whereas in $\mathsf{CnCK}$ all of these principle only appear in a very diluted and generalized form. Therefore, the tight link between $\mathsf{C}$ and $\mathsf{CnCK}$ is best seen through the medium of $\mathsf{CnK}$ and omitting the latter logic from our exposition would  very much weaken the case for $\mathsf{CnCK}$ as a correct minimal $\mathsf{C}$-based conditional logic. 

On the other hand, omitting $\mathsf{CnCK}_R$ would, as we mentioned in the opening paragraphs of Section \ref{S:reflexive} remove from the text our argument for the  claim that $\mathsf{CnCK}$ generalizes the implicational connectives of $\mathsf{C}$ rather than some other propositional connectives. Finally, omission of $\mathsf{CnCK}$ itself would leave the two remaining systems, $\mathsf{CnK}$ and $\mathsf{CnCK}_R$, which are clearly based on different principles and thus would invite the question of how a true conditional analogue of $\mathsf{CnK}$ looks like.

Of course, all of this is  but a beginning, and many other open questions present themselves in every possible direction of future research. In the introduction, we have already mentioned the necessity to systematically relate this new family of systems with modal and conditional logics extending the propositional relatives of $\mathsf{C}$, including intuitionistic logic and Nelson's logic of strong negation. But even within the realm of $\mathsf{C}$-based logics much remains to be done. We list several most interesting directions below.

\textit{1. Further extensions of $\mathsf{CnCK}$}. Their systematic investigation or at least a deeper look into the possibilities opening themselves up in this direction is a necessary condition for getting anything approximating a final answer to many questions already touched upon in the present paper. It is easy to show, for example, that some proper extensions of $\mathsf{CnCK}$ also have $\mathsf{CnK}$ as their modal companion and that some proper extensions of $\mathsf{CnCK}_R$ display the same set of connexivity principles for would- and might-conditionals.

\textit{2. Relations of $\mathsf{CnK}$ and $\mathsf{CnCK}$ to $\mathsf{C}$-based first-order logics}. In particular, various sorts of faithful embeddings via standard translation similar to the ones considered in \cite{olkhovikov1} and \cite{olkhovikov3} for intuitionistic and Nelsonian conditional logics seem to be very promising. However, unlike in the latter two cases, where the correspondence languages for such an embedding were known almost by default, coming up with a right first-order $\mathsf{C}$-based correspondence language appears to be a problem in itself. Such a language may be absent in the existing literature although the right direction to look for it would be among proper extensions of the logic proposed in \cite{ww}.

\textit{3. First-order $\mathsf{C}$-based modal and conditional logics}. First-order extensions of intensional logics are often a neglected topic, but in this case they might prove to be highly fruitful, at least as long as the similar developments in the field of intuitionistic and Nelsonian conditional logics are in place to support such a move.

Our hope is to take up all of the aforementioned topics in our future research --- not necessarily in the indicated order.

%
%
%

\bmhead{Funding} This research has received funding from the European Research Council (ERC) under the European Union's Horizon 2020 research and innovation programme, grant agreement ERC-2020-ADG, 101018280, ConLog

\bmhead{Acknowledgments}
To be inserted.

\noindent

\bigskip
%
%
%
%

\begin{appendices}
	
\section{Proof of Lemma \ref{L:cond-theorems}}\label{A:1}
	\eqref{E:RBbox} (resp. \eqref{E:RBdiam}) is entailed by \eqref{E:RCbox} and \eqref{E:a6} (resp. \eqref{E:RCdiam} and \eqref{E:a7}). Next, \eqref{E:T4} (resp. \eqref{E:T5}) follows by \eqref{E:a6} (resp. \eqref{E:a7}), \eqref{E:t-double-neg}, and \eqref{E:RBbox} (resp. \eqref{E:RBdiam}).
	
	We sketch the other proofs in more detail below:
\begin{align}
	\eqref{E:RNec}:\qquad\phi\label{E:p1}& &&\text{premise}\\
	\phi& \leftrightarrow (\phi\to\phi)\label{E:p3}&&\text{\eqref{E:p1}, $\mathsf{C}$}\\
	(\psi&\boxto\phi)\leftrightarrow(\psi\boxto(\phi\to\phi))\label{E:p4}&&\text{\eqref{E:p3}, \eqref{E:RCbox}}\\
	\psi&\boxto\phi\label{E:p5}&&\text{\eqref{E:p4}, \eqref{E:a5}, $\mathsf{C}$}
\end{align}
\begin{align}
	\eqref{E:Rmbox}:\qquad\phi&\to\psi\label{E:p6} &&\text{premise}\\
	(\phi&\wedge\psi)\leftrightarrow \phi\label{E:p7}&&\text{\eqref{E:p6}, $\mathsf{C}$}\\
	(\chi&\boxto(\phi\wedge\psi))\leftrightarrow(\chi\boxto\phi)\label{E:p8}&&\text{\eqref{E:p7}, \eqref{E:RCbox}}\\
	(\chi&\boxto\phi)\to(\chi\boxto\psi)\label{E:p10}&&\text{\eqref{E:p8},  \eqref{E:a1}, $\mathsf{C}$}
\end{align}
\begin{align}
	\eqref{E:Rmdiam}:\quad\phi&\to\psi\label{E:q0} &&\text{premise}\\
	(\phi&\vee\psi)\leftrightarrow \psi\label{E:q1}&&\text{\eqref{E:q0}, $\mathsf{C}$}\\
	(\chi&\diamondto(\phi\vee\psi))\leftrightarrow(\chi\diamondto\psi)\label{E:q2}&&\text{\eqref{E:q1}, \eqref{E:RCdiam}}\\
	(\chi&\diamondto\phi)\to(\chi\diamondto\psi)&&\text{\eqref{E:q2}, \eqref{E:a3}, $\mathsf{C}$}
\end{align}
\begin{align}
	\eqref{E:T1}:\,(\psi&\wedge(\psi\to\chi))\to\chi\label{E:q5}&&\text{$\mathsf{C}$}\\
	(\phi&\boxto(\psi\wedge(\psi\to\chi)))\to(\phi\boxto\chi)\label{E:q6}&&\text{\eqref{E:q5}, \eqref{E:Rmbox}}\\
	((\phi&\boxto\psi)\wedge(\phi\boxto(\psi\to\chi)))\to(\phi\boxto\chi)&&\text{\eqref{E:a1}, \eqref{E:q6}, $\mathsf{C}$}
\end{align}
\begin{align}
	\eqref{E:T2}:\,	((\phi&\diamondto\psi)\wedge(\phi\boxto(\psi\to\chi)))\to(\phi\diamondto(\psi\wedge(\psi\to\chi)))\label{E:q7}\qquad\quad\text{\eqref{E:a2}}\\
	(\psi&\wedge(\psi\to\chi))\to\chi\label{E:q8}\qquad\qquad\qquad\qquad\qquad\qquad\qquad\qquad\qquad\qquad\text{$\mathsf{C}$}\\
	(\phi&\diamondto(\psi\wedge(\psi\to\chi)))\to(\phi\diamondto\chi)\label{E:q9}\qquad\qquad\qquad\quad\text{\eqref{E:q8}, \eqref{E:Rmdiam}}\\
	((\phi&\diamondto\psi)\wedge(\phi\boxto(\psi\to\chi)))\to(\phi\diamondto\chi)\,\,\qquad\quad\text{\eqref{E:q7}, \eqref{E:q9}, $\mathsf{C}$}
\end{align}
\begin{align}
	\eqref{E:T3}:\,\psi&\to((\psi\to\chi)\to\chi)\label{E:r0}\qquad\qquad\qquad\qquad\qquad\qquad\qquad\quad\qquad\qquad\quad\text{$\mathsf{C}$}\\
	(\phi&\boxto\psi)\to(\phi\boxto((\psi\to\chi)\to\chi))\label{E:r1}\qquad\qquad\qquad\quad\text{\eqref{E:r0}, \eqref{E:Rmbox}}\\
	(\phi&\boxto\psi)\to((\phi\diamondto(\psi\to\chi))\to(\phi\diamondto\chi))\,\,\qquad\quad\text{\eqref{E:r1},\eqref{E:T2}, $\mathsf{C}$}
\end{align}




\end{appendices}



\end{document}